\documentclass[10pt]{ijart}

\usepackage[utf8x]{inputenc}
\usepackage{amsmath,amsfonts,amssymb,mathrsfs,latexsym,amscd,amsthm}
\usepackage{dsfont,url}
\usepackage{verbatim}
\usepackage{imakeidx}
\usepackage{hyperref}
\usepackage{graphicx}
\usepackage[all,cmtip]{xy}
\usepackage{stmaryrd}
\usepackage{enumitem}

\theoremstyle{plain}

\newtheorem{theorem}{Theorem}[section]
\newtheorem{conjecture}[theorem]{Conjecture}
\newtheorem{lemma}[theorem]{Lemma}

\newtheorem{proposition}[theorem]{Proposition}
\newtheorem{propdef}[theorem]{Proposition-Definition}

\newtheorem*{principle}{Principle}

\theoremstyle{definition}
\newtheorem{definition}[theorem]{Definition}
\newtheorem{example}[theorem]{Example}
\newtheorem{examples}[theorem]{Examples}
\newtheorem{remark}[theorem]{Remark}
\newtheorem{remarks}[theorem]{Remarks}
\newtheorem*{remark*}{Remark}

\newcommand\RR{\mathbf{R}}
\newcommand\CC{\mathbf{C}}
\newcommand\ZZ{\mathbf{Z}}
\newcommand\QQ{\mathbf{Q}}
\newcommand\NN{\mathbf{N}}
\newcommand\PP{\mathbf{P}}
\newcommand{\Gm}{\mathbb{G}_{\operatorname{m}}}
\newcommand{\Spec}{\operatorname{Spec}}
\newcommand{\Vect}{\operatorname{Vect}}
\newcommand{\aess}{\alpha_{\operatorname{ess}}}
\newcommand{\mumin}{\mu_{\operatorname{min}}}
\newcommand{\Norm}{\operatorname{Norm}}
\newcommand{\ctildeT}{\widetilde{\mathfrak{T}}_\mathbf{c}}
\newcommand{\ctildepi}{\widetilde{\pi_\mathbf{c}}}
\newcommand{\hetale}{H^1_{\text{\'et}}}
\newcommand{\aesslv}{\alpha_{\text{ess},L,\nu}}
\newcommand{\bY}{\boldsymbol{Y}}

\title{Rational approximations on toric varieties}
\author{Zhizhong Huang}

\address{Institute of Mathematics, Academy of Mathematics and Systems Science,\\ Chinese Academy of Sciences,\\ 100190 Beijing, China}
\email{zhizhong.huang@yahoo.com} 

\subjclass[2010]{14G05 (primary) 14M25 11G50 11J99 (secondary)}
\keywords{Diophantine approximation of rational points, toric varieties, universal torsors.}
\makeindex

\begin{document}
	\begin{abstract}
		Using the universal torsor method due to Salberger, we study the approximation of a general fixed point by rational points on split toric varieties. We prove that under certain geometric hypothesis the best approximations (in the sense of McKinnon-Roth's work) can be achieved on rational curves passing through the fixed point of minimal degree, confirming a conjecture of McKinnon. These curves are also minimal in the sense of deformation theory, and they correspond, according to Batyrev's terminology, to the centred primitive collections of the structural fan. 
	\end{abstract}
		\maketitle
	\tableofcontents
	\section{Introduction}\label{se:Sectionintro}
	\subsection{Background and motivation}
	K. Roth's theorem \footnote{In this article, we quote contributions from two mathematicians named Roth -- Klaus F. Roth and Mike Roth.} \cite{Roth} is one of the most outstanding and beautiful results in classical Diophantine approximation. Let $\theta\in\RR$ be a real number and $\mu(\theta)>0$\index{mutheta@$\mu(\theta)$} be the \emph{approximation exponent}, that is, the \emph{supremum} of positive real numbers $\mu$ such that the inequality 
	$$\left|\frac{p}{q}-\theta\right|<\frac{1}{\max(|p|,|q|)^\mu}$$
	has infinitely many solutions $\frac{p}{q}\in\QQ$. 	The exponent $\mu(\theta)$ measures how well the real number $\theta$ can be approximated by rational numbers with error term controlled by their ``size''. It is easy to see that if $\theta\in\QQ$, then $\mu(\theta)=1$. A classical result of Dirichlet \cite{Dirichlet1} asserts that, if $\theta$ is irrational, then $\mu(\theta)\geqslant 2$. Thus any irrational number is better approximated than any rational number. It was commonly recognized that the main difficulty lay in bounding $\mu(\theta)$ from above. K. Roth's theorem states
	\begin{equation}\label{eq:Roth}
		\mu(\theta)=2  \quad\text{if}\quad 2\leqslant[\QQ(\theta):\QQ]<\infty.
	\end{equation}
Thus K. Roth's theorem gives the exact approximation exponent $2$ for all irrational algebraic numbers. 
	
	Amongst the generalizations of K. Roth's theorem to higher dimensional cases, let us mention the Schmidt subspace theorem (see  \cite{Schmidt}) and the Faltings-Wüstholz Theorem (see \cite{FaltingsWustholz}). Recently, in a series of works \cite{McKinnon2007}, \cite{McKinnon-Roth1} and \cite{McKinnon-Roth2}, McKinnon and M. Roth introduced the notion of \emph{approximation constant} $\alpha$ (Definition \ref{def:appconst}) and formulated a framework of Diophantine approximation of rational points on arbitrary algebraic varieties. For $X$ a variety defined over a number field $K$\index{K@$K,\overline{K}$} (embedded into a fixed algebraic closure $\overline{K}$), and for every fixed $Q\in X(\overline{K})$, choose some distance function $d_\nu(\cdot,Q)$\index{dist@$d_\nu(\cdot,Q)$} with respect to some fixed place $\nu$ of $K$. Choose a height function $H_L$\index{HL@$H_L$} associated to some fixed line bundle $L$. Then the \emph{(best) approximation constant} $\alpha_{L,\nu}(Q,X)$\index{alpha@$\alpha_{L,\nu}(Q,Y)$} is defined as the \emph{infimum} of positive real numbers $\gamma$ such that the inequality
\begin{equation}\label{eq:alphadef}
	d_\nu(P,Q)^\gamma H_L(P) \leqslant 1, 
\end{equation}
	has infinitely many solutions $P_i\in X(K)$ satisfying $d_\nu(P_i,Q)\to 0$.
	It measures local behaviour of rational points around $Q$ by means of how fast their heights must grow when approaching the fixed point $Q$ on the variety $X$. It also plays a central role in recent investigations \cite{Huang1} \cite{Huang2} \cite{Huang3} of the author on local distribution of rational points. As we put the exponent $\gamma$ on the distance rather than on the height, smaller $\alpha$ means better approximation. For example, on $\mathbf{P}^1$, its relationship with the approximation exponent is $\alpha_{\mathcal{O}(1),\infty}(\theta,\PP^1)=\mu(\theta)^{-1}$ (Example \ref{ex:Roth}). As pointed out before, bounding $\alpha$ from below and even computing its value seem to be challenging problems. Inaugurated by Nakamaye, it now becomes a classical fact that local positivity of a line bundle should govern the its Diophantine approximation quality. McKinnon and M. Roth called it the ``local Bombieri-Lang phenomena''. Bearing this spirit, they provide lower bounds for the constant $\alpha$ using local geometric invariants. One version \cite[Theorem 6.3]{McKinnon-Roth1} of their main results is that for any rational point $Q\in X(\overline{K})$ and any ample line bundle $L$,
\begin{equation}\label{eq:McKinnonRoth}
	\alpha_{L,\nu}(Q,X)\geqslant \frac{1}{2}\varepsilon_L(Q)\index{epsilon@$\varepsilon_L(Q)$}.
\end{equation}
	 where $\varepsilon_L(Q)$ is the \emph{Seshadri constant} of $L$ at $Q$.  Moreover, the inequality above is an equality if and only if both $\alpha$ and $\varepsilon$ are computed on some rational curve on $X$ passing through $Q$.
	See also \cite{Grieve} for analogous results over function fields.
	
	According to some heuristic due to Batyrev and Manin, there exist many similarities between the distribution of rational points and that of rational curves. For example, let us mention Manin's conjecture \cite[RCC]{Manin} on existence of rational curves via the number of rational points of bounded height. In an attempt to formulate a local analogue, McKinnon made the following conjecture, based on the empirical fact 
	that rational points tend to accumulate on rational curves when approaching a fixed point.
	\begin{conjecture}[\cite{McKinnon2007} Conjecture 2.7]\label{conj:mckinnon}
		Let $X$ be a variety over a number field $K$, $L$ be an ample line bundle and $\nu$ be a place of $K$. Suppose that $Q\in X(K)$ and that there exists a rational curve defined over $K$ passing through $Q$ on $X$. Then there exists a rational curve $C$ on $X$ passing through $Q$ achieving the best approximation constant at $Q$ with respect to $L$ and $\nu$, i.e., $$\alpha_{L,\nu}(Q,X)=\alpha_{L,\nu}(Q,C).$$
	\end{conjecture}
	Assuming Vojta's Conjecture (see \cite{Vojta}), McKinnon \cite[\S4]{McKinnon2007} showed the consistency of Conjecture \ref{conj:mckinnon} for varieties of general type \footnote{Further evidence for toric varieties is given in a recent paper of McKinnon and Satriano \cite{McKinnon-Satriano}.}. 
	
	There have been a number of works on computation of the approximation constant.  See \cite{McKinnon2007}, \cite{McKinnon-Roth2}, \cite{Huang1}, \cite{Huang2}, \cite{Huang3}. All of them mostly consider rational surfaces, often (weak) del Pezzo surfaces of degree $\geqslant 3$, and they all satisfy Conjecture \ref{conj:mckinnon}.
	 Such surfaces are special kinds of \emph{rationally connected varieties} (see \cite[IV.~3.2]{Kollar}). One big advantage of working with them is that there exist \emph{free} rational curves (over $\overline{K}$) through every general point, so conjecturally (see \cite[p. 174]{C-T}) they contain many rational points (if there exists at least one).
	  However, apart from the simplest varieties like the projective spaces and naive constructions such as products of varieties, there are very few higher dimensional cases for which the constant $\alpha$ is known to have been computed. A similar difficulty appears for the Seshadri constants, even though they are known for all del Pezzo surfaces (see \cite{Broustet}, \cite{Gottsche-Pandharipande}). 
	 
	 
	 In contrast to the notion of \emph{(globally) accumulating varieties} which appears in the Batyrev-Manin-Peyre Principle (see \cite{batyrev-manin}, \cite{Peyre1}) and refers to the subvarieties on which the growth of rational points of bounded height dominates the whole variety, the $\alpha$-constant helps detect the \emph{locally accumulating subvarieties} (Definition \ref{def:achieve} (2)). These subvarieties contain rational points that are ``closer'' to the given point $Q$, in the sense that when $\gamma$ is sufficiently close to $\alpha_{L,\nu}(Q,X)$, almost all solutions of the inequality \eqref{eq:alphadef} are located there. 
	 Thus it makes sense to study what happens on open dense subsets obtained by removing some closed locally accumulating subvariety. 
	 The \emph{essential constant} $\alpha_{\text{ess},L,\nu}(Q)$\index{alphaess@$\aess(Q)$} (Definition \ref{def:essconst}) first introduced by Pagelot \cite{Pagelot} provides such a characterization, which we shall call \emph{generic (best) approximations}. It is defined as the supremum of $\alpha_{L,\nu}(Q,U)$ as $U$ ranges over all Zariski dense open sets. If $\aesslv(Q)$ is finite and the supremum can be achieved on some open set, then in the light of Conjecture \ref{conj:mckinnon}, rational curves realizing $\aesslv(Q)$ should be deformable while fixing $Q$ (hence \emph{very free}) and should cover this open set.
	 
	 In all, even though local behaviour of rational points is very rich and complicated, as already seen on surfaces, all known results are in favour of the following enhanced version of Conjecture \ref{conj:mckinnon}:
	\begin{principle}
		Assume that the variety is rationally connected and the point to be approximated is general. Then 
		the best (resp. generic) approximations should be achieved on subvarieties swept out by free (resp. very free) rational curves of small degree.
	\end{principle}

\subsection{Main results}
Toric varieties are special kinds of rational varieties admitting a generically transitive group action. Their arithmetic has been intensively studied.
Rational points are very well distributed in the sense of the Batyrev-Manin-Peyre Principle, thanks to the works of Batyrev-Tschinkel \cite{B-T1}, \cite{B-T2} and Salberger \cite{Salberger}.


We shall be interested in smooth projective toric varieties $X$ of dimension $\geqslant 2$ satisfying the following geometric condition.
	\begin{center}
		$(*)$ The cone of pseudo-effective divisors $\overline{\operatorname{Eff}}(X)$\index{EffX@$\overline{\operatorname{Eff}}(X)$} is simplicial.\index{Hyp1@Hypothesis $(*)$}
	\end{center}
Typical examples are products of projective spaces, projectivizations of direct sums of line bundles over projective spaces, and can have arbitrarily large Picard number.
We assume that the point to be approximated is general, so that it lies on the open orbit, i.e. the torus $\mathcal{T}=\Gm^{\dim X}$\index{T@$\mathcal{T}$}.
The (global) accumulating subvariety is the union of boundary divisors and does not play a role here. 
We prove the following two results, which confirm the Principle and answer affirmatively McKinnon's Conjecture \ref{conj:mckinnon}.
The first one concerns the best approximations.
\begin{theorem}[see Theorem \ref{thm:mainthm2}]\label{thm:mainthm}
	Let $X$ be a split smooth projective toric variety defined over a number field $K$ and equipped with a line bundle $L$. Let $Q\in\mathcal{T}(K)$. Suppose that $X$ verifies Hypothesis $(*)$.
	\begin{enumerate}
		\item If $L$ is nef, then the best approximations for $Q$ \emph{can be achieved} on any free rational curve through $Q$ of minimal $L$-degree.
		\item If $L$ is ample, then the best approximations for $Q$ \emph{are properly achieved} on the subvariety swept out by the free rational curves through $Q$ of minimal $L$-degree. 
	\end{enumerate}
\end{theorem}

The precise meaning of ``properly achieved'' and ``can be achieved''
will be discussed in Section \ref{se:Sectionalpha} (Definition \ref{def:achieve}, Remarks \ref{rmk:achieve}). Our result not only gives the precise value of the approximation constant $\alpha$, but also reveals the exact shape of locally accumulating subvarieties, therefore it can be seen as an effective version of the main theorems in \cite[\S6]{McKinnon-Roth1} in the toric setting. Theorem \ref{thm:mainthm} also generalises \cite[Corollary 3.4]{McKinnon2007} which considers all dimension $2$ cases.

All rational curves achieving the best approximation in Theorem \ref{thm:mainthm} are smooth. In fact, they correspond to the so-called \emph{centred primitive collections}, a notion first invented by Batyrev \cite{Batyrev}. They turn out to be parametrised by the components of the space of rational curves $\operatorname{RatCurve^n}(X)$\index{RatCurve@$\operatorname{RatCurve^n}(X)$} (see \cite[Definition 2.11]{Kollar}) that are \emph{minimal}, whose existence for general varieties now relies on  Mori's theory. 

Our second result concerns generic approximations.
\begin{theorem}[see Theorem \ref{thm:mainthm3}]\label{thm:generic}
	Under the assumption of Theorem \ref{thm:mainthm}, suppose that $X$ has Picard number $\leqslant 2$ (in particular they verify $(*)$), and that $L$ is nef. Then the generic best approximations for $Q$ can be achieved on every very free rational curve through $Q$ of minimal $L$-degree.
\end{theorem}
Under the weaker assumption that $L$ is moreover big, then in many cases we can show that (see Theorem \ref{thm:mainthm3} (2)) the subvariety swept out by the minimal free rational curves through $Q$ is locally accumulating as in Theorem \ref{thm:mainthm} (2).

Based on results about surfaces, we observe that both the best approximations and the generic approximations, especially the latter, seem to show some degeneration-invariance amongst families of polarized varieties. Studying Diophantine approximation on toric varieties may give some evidence on what happens about other varieties admitting toric degenerations. It would also be interesting to compare our result with \cite{ito}, which gives an estimate of the Seshadri constant for toric varieties.
By performing the geometric argument as in \cite[\S3]{McKinnon2007} \cite[\S3]{McKinnon-Roth2}, we would be able to work out Conjecture \ref{conj:mckinnon} for a larger class of varieties not necessarily toric but admitting birational morphisms to toric ones. 

 Without Hypothesis $(*)$, the situation is noticeably more complicated. Indeed, it is not always true that smooth rational curves of minimal degree contribute to the best or the generic approximations, especially when the pseudo-effective cone or the nef cone has too many generators.
McKinnon \cite[\S4]{McKinnon2007} exhibits first examples --- smooth cubic surfaces --- on which the best approximations for a general point are properly achieved on a singular cubic curve (see \cite[Theorem 4.5]{McKinnon-Roth2} for a detailed statement). 
So does their toric degeneration with $3\mathbf{A}_2$ singularities. \footnote{See also \cite[\S8]{McKinnon-Satriano} for another example of a weighted projective space.} Note that (the desingularisation of) this toric variety does not verify Hypothesis $(*)$. 
For the generic approximations, see the surfaces $Y_3,Y_4$ in \cite{Huang2} and \cite{Huang3}. The phenomenon is that, certain singular rational curves, whose approximation constants are equal to their degree divided by the factor $2$ coming from K. Roth's theorem \eqref{eq:Roth} or by the multiplicity at the singular point, give better approximations than the smooth ones. Moreover, the nodal type and the cuspidal type singularities have different contributions. See  \cite[Theorem 2.16]{McKinnon-Roth1}. The appearance of such singular curves is quite general (see a family of examples in \cite[\S5.5]{Huang2}) and they merit further investigation. By Theorem \ref{thm:mainthm}, Hypothesis $(*)$ indicates a sufficient condition for which singular curves do not enter. 
Nevertheless, all known results shed light on the definition of the Seshadri constant: looking for (singular) rational curves whose multiplicities at a fixed point are comparable with their degrees.

\subsection{Outline of the proof} 
We first outline a general strategy about how to compute the approximation constant. Let us denote by $\alpha_{L,\nu}(Q,Y)$ the approximation constant for a point $Q$ computed with respect to a subvariety $Y$.  
To prove that $\alpha_{L,\nu}(Q,Y)\leqslant\gamma$ for some $\gamma>0$, it suffices to find a rational curve $l$ such that $Y$ contains some open dense part of $l$ and that $\alpha_{L,\nu}(Q,l)=\gamma$.
If $l$ is smooth at $Q$, then $\alpha_{L,\nu}(Q,l)$ is just $\deg_L (l)$ (see Proposition \ref{prop:propertiesofalpha}).
The main difficulty lies frequently in obtaining a lower bound.
Assume for simplicity that $L$ is ample. If we could choose properly a height $H_L(\cdot)$ and a distance function $d_\nu(\cdot,Q)$ locally around the fixed point $Q$, and prove a Liouville-type inequality of the form 
\begin{equation}\label{eq:outline0}
d_\nu(P,Q)^\gamma H_L(P)\geqslant C
\end{equation} 
for certain $C>0$ and uniformly for all $K$-points $P$ of $Y$ near $Q$,
this would imply that $\alpha_{L,\nu}(Q,Y)\geqslant \gamma$ (see Proposition \ref{prop:lowerbd}). Combining the previous upper bound, we get the exact value of $\alpha_{L,\nu}(Q,Y)$.

Now assume that $Y$ is Zariski closed. To derive that the constant $\alpha_{L,\nu}(Q,X)$ is properly achieved on $Y$, we first need to show that $$\alpha_{L,\nu}(Q,X)=\aesslv(Q,Y).$$
This amounts to saying that $Y$ itself does not contain any proper Zariski closed subset with smaller approximation constant. This is usually the case when $Y$ is the deformation locus of a class of free rational curves (but not very free) achieving $\alpha_{L,\nu}(Q,X)$. Secondly, we need to do better than \eqref{eq:outline0}, that is, we need to prove that there exists some $\delta>0$ such that
\begin{equation}\label{eq:outline}
d_{\nu}(P,Q)^{\alpha_{L,\nu}(Q,X)+\delta}H_L(P)\geqslant C^\prime>0,
\end{equation}
uniformly for every $K$-point $P$ near $Q$ not in Y.
This implies (see Proposition \ref{prop:lowerbd}) $$\aesslv(Q)\geqslant\alpha_{L,\nu}(Q,X\setminus Y)\geqslant \alpha_{L,\nu}(Q,X)+\delta >\alpha_{L,\nu}(Q,X).$$

To parametrize rational points, we make use of universal torsors \emph{à la} Colliot-Thélène and Sansuc \cite{ct-sansuc}, which allow to lift rational points into integral points in some affine space. And surprisingly, the incorporation of Hypothesis $(*)$ minimizes the complexity introduced by such integral coordinates. We appeal to the work of Salberger \cite{Salberger}, which pioneers combinatorial ways of computing toric height functions. He derives explicit height formulas that also encode information on positivity of the line bundle, and he uses it to prove the Batyrev-Manin-Peyre Principle for split toric varieties over $\QQ$. Pieropan \cite{Pieropan} extends Salberger's result to imaginary quadratic fields, and Frei \cite{Frei} treats the singular cubic surface with $3\textbf{A}_2$ singularities over arbitrary number fields. In the case of function fields, Bourqui \cite{Bourqui2009} \cite{Bourqui2016} studies the distribution of families of rational curves and proves the geometric Batyrev-Manin-Peyre Principle for many types of toric varieties. 
Carrying out the estimation of \eqref{eq:outline0} and \eqref{eq:outline} is a sophisticated task, and is very different from the procedure of counting rational points of bounded height in \cite[\S11]{Salberger}, although it is essentially a comparison between the growth of height and the decreasing of distance.
We shall explain more in Section \ref{se:SectionTheorem}, based on toric geometry,  how Hypothesis $(*)$ and centred primitive collections together help to deduce stronger positivity (e.g. Proposition \ref{prop:sigmai00}) of toric heights.


Thanks to the classification due to Kleinschmidt \cite{Kleinschmidt}, we know all possible fans defining smooth complete toric varieties of Picard number $2$. In particular they always satisfy Hypothesis $(*)$. With more explicit information, we can improve the estimate \eqref{eq:outline} for a properly chosen $Y$ by adapting the exponent on the distance to be the expected value, namely the minimal $L$-degree of very free rational curves, so as to bound the constant $\alpha_{L,\nu}(Q,X\setminus Y)$ from below. It remains to find a dominant family of rational curves in $X\setminus Y$ all passing through $Q$ and achieving $\aesslv(Q)$, and the family of general lines will do.



\subsection{Layouts of the article.} In Section \ref{se:Sectiongeometry} we shall recall some basic toric geometry and the notion of freeness for rational curves, including the geometry of centred primitive collections. We also derive a criterion of characterizing very free rational curves on toric varieties, which may be of independent interest. The parametrization of rational points on toric varieties via universal torsors, together with the formula of calculating heights associated to globally generated line bundles, are given in Section \ref{se:Sectiontorsor}. We define the best approximation constant $\alpha$ and the essential constant $\aess$ in Section \ref{se:Sectionalpha} and discuss several fundamental properties. In Section \ref{se:SectionNumber} we recall a few useful classical facts about algebraic number fields. Section \ref{se:SectionTheorem}, the most technical part, is devoted to the proof of Theorem \ref{thm:mainthm}. 
In Section \ref{se:SectionPic2} we study toric varieties of Picard number $2$ in detail including their structural fans, various cones of divisors, very free curves of minimal degree, and we prove Theorem \ref{thm:generic}. 
In this article, most of the intermediate results are formulated in the language of toric geometry. 

\subsection{Notation} We fix throughout this paper a number field $K$\index{K@$K,\overline{K}$}. Let $\mathcal{O}_K$\index{OK@$\mathcal{O}_K$} be the ring of integers, $\operatorname{Cl}_K$\index{CLK@$\operatorname{Cl}_K$} be the class group and $\mathcal{M}_K$\index{MK@$\mathcal{M}_K,\mathcal{M}_K^f, \mathcal{M}_K^\infty$} be the set of places of $K$. The set $\mathcal{M}_K=\mathcal{M}_K^f\sqcup \mathcal{M}_K^\infty$ comprises finite places and infinite ones. For $\nu\in\mathcal{M}_K^f$, we shall use the absolute value $|\cdot|_\nu$\index{nu@$\lvert\cdot\rvert_\nu$} normalized with respect to $K$. That is, if $p$ is a prime number such that $\nu\mid p$, then $|x|_\nu=|N_{K_\nu/\QQ_p}(x)|_p$. If $\nu\in\mathcal{M}_K^\infty$, we put $|\cdot|_\nu=|\cdot|$ if $\nu$ is real and $|\cdot|_\nu=|\cdot|^2$ if $\nu$ is complex, where $|\cdot|$ is induced by the usual absolute value on the completion $K_\nu$ via the embedding $\varsigma_\nu:K\hookrightarrow K_\nu$\index{Knusigmanu@$K_\nu,\varsigma_\nu$}. 
Let $\Norm(\cdot)$\index{Norm@$\Norm(\cdot)$} be the norm function defined for all fractional ideals of $K$. For $\mathfrak{p}\in\operatorname{Spec}\mathcal{O}_K$, $\operatorname{ord}_{\mathfrak{p}}(\cdot)$\index{ordp@$\operatorname{ord}_{\mathfrak{p}}(\cdot)$} denotes the valuation order in the ring $\mathcal{O}_{K,\mathfrak{p}}$.
Let $V$ be a vector space over a field $F$ and $P\subset V$. Then $\operatorname{Vect}_F(P)$\index{Vect@$\operatorname{Vect}$} denotes the vector subspace of $V$ spanned by elements in $P$. Further notation for toric varieties will be introduced in subsequent texts.

\subsection*{Acknowledgements} This paper grew out of part of my Ph.D. thesis realised at Université Grenoble Alpes. I would like to thank Emmanuel Peyre for constant encouragement over the past few years, and I'm grateful to David McKinnon for his interest. The idea of considering primitive collections was brought up by Michel Brion, to whom I address my gratitude. Special thanks go to the anonymous referee for numerous suggestions which lead to significant improvements in the exposition. When working on this project the author was partly supported by the project ANR GARDIO, by a Riemann fellowship and by grant DE 1646/4-2 of the Deutsche Forschungsgemeinschaft.

\section{Geometric preliminaries}\label{se:Sectiongeometry}
\subsection{Toric geometry}
We refer the reader to excellent books  \cite{Fulton} and \cite{CoxLittleSchenck} for general introduction to toric varieties.
In this section we state several well-known facts needed mostly without proof and fix notation.

Fix a rank $n$\index{n@$n$} lattice $N\simeq \ZZ^n$\index{N@$N$} and let $M=N^\vee=\operatorname{Hom}_\ZZ(N,\ZZ)$\index{M@$M$}. We denote by $\mathcal{T}=\operatorname{Spec}(K[M])\simeq \mathbb{G}_{\operatorname{m},K}^n$ the open orbit. The lattice $N$ (resp. $M$) is naturally identified with the set of co-characters (resp. characters) of the torus $\mathcal{T}$. For $m\in M$, write $\chi^m:\mathcal{T}\to\Gm$\index{chim@$\chi^m$} for its associated character. For any $v \in N$, write $\lambda_v:\Gm\to \mathcal{T}$\index{lambdarho@$\lambda_v$} be the co-character associated to $v$.

Let $\triangle$\index{fan@$\triangle,\triangle_{\max},X(\triangle)$} be an $n$-dimensional fan consisting of a finite collection of (strongly convex, rational polyhedral and simplicial) \emph{cones} $\sigma\subset N_\RR$ whose \emph{support} $\operatorname{Supp}(\triangle)$ is $\cup_{\sigma\in\triangle}\sigma$ (\cite[Definition 3.1.2]{CoxLittleSchenck}). We denote by $\triangle_{\max}$ the set of maximal cones. For any $\sigma\in\triangle_{\max}$, $\sigma^\vee\subset M_\RR$ denotes its dual cone and $U_\sigma=\operatorname{Spec}(K[\sigma^\vee\cap M])\simeq \mathbf{A}_K^n$\index{Usigma@$\sigma,\sigma^\vee,U_\sigma$} denotes its associated affine open neighbourhood.
The toric variety $X=X(\triangle)$ associated to $\triangle$ is constructed by gluing the data $(U_\sigma,\sigma\in\triangle)$. 

Each one-dimensional cone (called a \emph{ray}) contains a unique primitive element in $N$, which we shall call \emph{generator} (of the ray). Let $\triangle(1)$\index{fan1@$\triangle(1)$} be the set of generators, so that every $\rho\in\triangle(1)$ generates the ray $\RR_{\geqslant 0} \rho$. 
For every cone $\sigma$, $\sigma(1)=\triangle(1)\cap \sigma$\index{sigma@$\sigma,\sigma(1)$} denotes the set of generators of its rays. We call $\sigma$ \emph{regular} if elements of $\sigma(1)$ form part of a basis of the lattice $N$.
The toric variety $X$ is complete and smooth if and only if $\operatorname{Supp}(\triangle)=N_\RR$ and all cones are regular (\cite[Theorem 3.1.19]{CoxLittleSchenck}). We suppose throughout this paper that $X$ is projective and smooth, i.e. complete, smooth and admitting at least one ample divisor, unless otherwise specified.

  The group $\operatorname{Pic}(X)$ being torsion-free, we let $r=\operatorname{rank}_\ZZ\operatorname{Pic}(X)$\index{r@$r$}. 
We recall that for smooth toric varieties there is no difference between the numerical equivalence and the rational equivalence of divisors, in other words, $\operatorname{Pic}^0(X)=0,\operatorname{Pic}(X)\simeq \operatorname{NS}(X)$\index{NSX@$\operatorname{NS}(X)$}, the Néron-Severi group (\cite[Proposition 6.3.15]{CoxLittleSchenck}). 
Each $\rho\in\triangle(1)$ corresponds to a $\mathcal{T}$-invariant boundary divisor $D_{\rho}$\index{Drho@$D_\rho$}. Moreover let $C_\tau$ be the torus invariant curve corresponding to a $(n-1)$-dimensional cone $\tau$ such that the elements of $\tau(1)$ together with $\rho$ generate a maximal cone. Then we have $\langle D_\rho,C_\tau\rangle=1$ (\cite[Proposition 6.4.3]{CoxLittleSchenck}).
So the intersection product induces a non-degenerate and perfect paring  $\operatorname{Pic}(X)\times A_1(X)\to \ZZ$, where $A_1(X)$\index{A1X@$A_1(X)$} denotes the Chow group of $1$-cycles modulo numerical equivalence. 

\begin{definition}
In this article, we call any non-trivial equality $\mathcal{P}:\sum_{\rho\in\triangle(1)}a_\rho\rho=0$ a \textit{relation between generators}, or simply a \textit{relation}. It is called \textit{positive} if all coefficients $a_\rho$ are non-negative. We denote by $\mathcal{P}(1)$\index{P1@$\mathcal{P},\mathcal{P}(1)$} the subset of $\rho\in\triangle(1)$ such that $a_\rho\neq 0$.
\end{definition}
We may identify the set of relations as a subgroup of $\ZZ^{\triangle(1)}$ with addition operated respectively on each coefficient. The set of positive relations forms a semi-group.

Recall the following fundamental exact sequence of $\ZZ$-modules (\cite[Proposition 2.12]{Batyrev}, \cite[\S3.4]{Fulton}, \cite[Theorem 4.1.3]{CoxLittleSchenck}):

\begin{equation}\label{eq:exactseqstr}
\xymatrix{
	0\ar[r]&M\ar[r]^-h& \ZZ^{\triangle(1)}\ar[r]^-i&
	\operatorname{Pic}(X)\ar[r]& 0,}
\end{equation}
where $\ZZ^{\triangle(1)}$ is naturally identified with the abelian group of $\mathcal{T}$-invariant divisors on $X$. In particular $\sharp\triangle(1)=n+r$. By taking duals, we obtain
\begin{equation}\label{eq:exactseqstr2}
\xymatrix{
	0\ar[r]
	&\operatorname{Pic}(X)^\vee\ar[r]^-f& \ZZ^{\triangle(1)}\ar[r]^-g
	&N\ar[r]&0, }
\end{equation}
where $\operatorname{Pic}(X)^\vee$ is the dual lattice of $\operatorname{Pic}(X)$, identified with $A_1(X)$.
The maps $f,g,h,i$ are given as follows.
For $m\in M$, 
$$h(m)=(\langle m,\rho\rangle)_{\rho\in\triangle(1)}\in \ZZ^{\triangle(1)}.$$ 
For $(a_\rho)_{\rho\in\triangle(1)}\in \ZZ^{\triangle(1)}$, 
$$i((a_\rho)_{\rho\in\triangle(1)})=\sum_{\rho\in\triangle(1)}a_\rho[D_{\rho}]\in \operatorname{Pic}(X).$$
For a curve $l\subset X$, 
$$f([l])=\left(\langle D_{\rho}\cdot l\rangle\right)_{\rho\in\triangle(1)}\in\ZZ^{\triangle(1)}.$$
We extend $f$ to $A_1(X)$ by linearity.
Finally for $(a_\rho)_{\rho\in\triangle(1)}\in \ZZ^{\triangle(1)}$, 
$$g((a_\rho)_{\rho\in\triangle(1)})=\sum_{\rho\in\triangle(1)}a_\rho \rho\in N.$$
We may identify the group $A_1(X)$ as the kernel of $g$, a subset of $\ZZ^{\triangle(1)}$.
That is, we view curve classes as their associated relations via $g$, whose coefficients are precisely the intersection multiplicities with boundary divisors.
\begin{theorem}\label{thm:positiverelation}
	The following three sets are in one-to-one correspondence with each other:
		\begin{enumerate}
			\item the classes of (rational) curves in $A_1(X)$ intersecting with the open orbit;
			\item the set of positive relations; 
			\item the equivalent families of non-zero homogeneous polynomials $(f_{\rho}(u,v))_{\rho\in\triangle(1)}$ with coefficients in $K$ indexed by $\triangle(1)$ such that $\sum_{\rho\in\triangle(1)}\deg (f_{\rho})\rho=0$ and satisfying the coprimality condition:
			\begin{equation}\label{eq:coprime2}
			\text{for every } \mathcal{I}\subset\triangle(1),\bigcap_{\rho\in\mathcal{I}}D_{\rho}=\varnothing\Longrightarrow \gcd_{\rho\in\mathcal{I}}(f_{\rho}(u,v))=1.
			\end{equation}
	\end{enumerate}
\end{theorem}
\begin{proof}[Sketch of proof]
	On the one hand, if a curve meets the open orbit $\mathcal{T}$, then it intersects properly with all boundary divisors. So all coefficients of its associated relation are non-negative.
	On the other hand,  
	given a positive relation $\sum_{\rho\in\triangle(1)}a_\rho\rho=0,a_\rho\geqslant 0$, we check that for every $(n+r)$-tuple of pairwise distinct elements $(b_\rho)_{\rho\in\triangle(1)}\in K^{\triangle(1)}$, the Zariski closure of the map 
	$$\Gm\dashrightarrow \mathcal{T},\quad x\mapsto \prod_{\rho\in\triangle(1)}(\lambda_\rho(x-b_\rho))^{a_\rho}$$
	is a rational curve. The image of its class under $f$ in \eqref{eq:exactseqstr2} is $(a_\rho)_{\rho\in\triangle(1)}$. This establishes the equivalence between (1) and (2).

	The equivalence between (1) and (3) can be seen as a description of universal torsors for toric varieties (see Section \ref{se:Sectiontorsor} below) over rational function fields of one variable. It is also a particular case of the \emph{functoriality of toric varieties} due to Cox \cite{Cox-tohoku}. We refer to \cite[\S1.2]{BourquiFourier} for a presentation.
\end{proof}

To every $\mathcal{T}$-invariant divisor $D=\sum_{\rho\in\triangle(1)}a_{\rho}D_{\rho}$, 
we associate a polyhedron
$$P_D\index{PD@$P_D$}=\{m\in M_\RR:\text{for all }\rho\in\triangle(1),\langle m,\rho\rangle\geqslant -a_{\rho}\}\subset M_\RR\simeq \RR^n,$$
whose lattice points correspond to global sections of $\mathcal{O}_X(D)$ (see \cite[Proposition 4.3.3]{CoxLittleSchenck}). We also associate a piecewise affine (i.e. linear on every maximal cone) function $\phi_D:N_\RR\to \RR$\index{phiD@$\phi_D$} as follows.
 For any $\gamma\in N_\RR$, choose a maximal cone $\sigma=\sum_{i=1}^{n}\RR_{\geqslant 0} \rho_i$ containing $\gamma$. Let $\{\rho_1^*,\cdots,\rho_n^*\}$ be the dual basis of $\{\rho_{1},\cdots\rho_n\}$ and
 \begin{equation}\label{eq:mdsigma}
 m_D(\sigma)\index{mDsigma@$m_D(\sigma)$}=\sum_{i=1}^n -a_{\rho_i}\rho_i^*\in M, \footnote{Our definition of $m(\sigma)$ differs from \cite[\S8]{Salberger} by a minus sign.}
 \end{equation} 
 that is, $m_D(\sigma)$ is the unique element in $M$ determined by $\langle m_D(\sigma),\rho\rangle=-a_{\rho}$ for every $\rho\in\sigma(1)$.
 Then we define $\phi_D(\gamma)=\langle m_D(\sigma),\gamma\rangle$. 
The function $\phi_D$ is called \emph{convex}, if for all $\sigma\in\triangle_{\max}$,
\begin{equation}\label{eq:convex1}
\phi_{D}(\cdot)\leqslant\langle m_D(\sigma),\cdot\rangle.
\end{equation}
It is called \emph{strictly convex} if moreover for every $\gamma\in N_\RR$ and every $\sigma\in\triangle_{\max}$ such that $\gamma\not\in\sigma$, we have
\begin{equation}\label{eq:convex2}
\phi_{D}(\gamma)<\langle m_D(\sigma),\gamma\rangle.
\end{equation}
Intuitively, the (strict) convexity means that the graph of $\phi_D$ lies (strictly) below that of the linear function $\langle m_D(\sigma),\cdot\rangle$ for every $\sigma\in\triangle_{\max}$. 

The following result establishes several equivalences between different types of positivity of line bundles, convexity of associated affine functions and volume of associated polyhedra.
\begin{theorem}\label{thm:Demazure}
	Let $X$ be a smooth projective toric variety, and $D$ be a $\mathcal{T}$-invariant divisor.
	\begin{enumerate}
		\item (Demazure) The line bundle $\mathcal{O}_X(D)$ is globally generated (resp. ample) if and only if it is nef (resp. very ample). This holds precisely when the function $\phi_D$ is convex (resp. strictly convex).
		\item The line bundle $\mathcal{O}_X(D)$ is \emph{big} (see \cite[\S2.2]{Lazarsfeld}) if and only if $P_D$ has strictly positive $n$-dimensional volume.
	\end{enumerate}
\end{theorem}
\begin{proof}
	For the first part, see for example \cite[p. 68, p. 70]{Fulton}, and \cite[Theorem 6.3.12]{CoxLittleSchenck}.
	The second assertion follows from $\operatorname{vol}(D)=n!\operatorname{Vol}_{\RR^n}(P_D)$, established in \cite{ELMM}.
\end{proof}
For every $\sigma\in\triangle_{\max}$, on the affine open set $U_\sigma$ the line bundle $\mathcal{O}_X(D)$ trivializes as $\chi^{m_D(\sigma)}\mathcal{O}_{U_\sigma}$. We see from Theorem \ref{thm:Demazure} (1) that if $\mathcal{O}_X(D)$ is globally generated, then $m_D(\sigma)\in P_D\cap M$ and $\chi^{-m_D(\sigma)}$ lifts to a global section of $\mathcal{O}_X(D)$.

\begin{definition}\label{def:degree}
	For $D$ a $\mathcal{T}$-invariant divisor and  $\mathcal{P}:\sum_{\rho\in\triangle(1)}a_\rho\rho=0$ a relation, we define $\deg_{\mathcal{O}_X(D)}\mathcal{P}$\index{degP@$\deg_{\mathcal{O}_X(D)}\mathcal{P},\deg_{L}\mathcal{P}$}, the $\mathcal{O}_X(D)$-\emph{degree} of $\mathcal{P}$, to be 
	$$\deg_{\mathcal{O}_X(D)}\mathcal{P}=-\sum_{\rho\in\triangle(1)}a_\rho\phi_{D}(\rho).$$
\end{definition}
If the class of a curve $C$ corresponds to $\mathcal{P}$, then from the definition of $\phi_{D}$ and the exact sequences  \eqref{eq:exactseqstr} and \eqref{eq:exactseqstr2}, $\deg_{\mathcal{O}_X(D)}\mathcal{P}$ is nothing but the intersection number $\langle D,C\rangle$, i.e. the $\mathcal{O}_X(D)$-degree of the curve $C$.

\subsection{Primitive collections}\label{se:primitive}
This notion is introduced by Batyrev in classifying higher dimensional smooth toric varieties. We refer to \cite[Definitions 2.6-2.10]{Batyrev} and the book \cite[Definition 6.4.10]{CoxLittleSchenck} for more details.
\begin{definition}[Batyrev]\label{def:primitive}
	A subset of generators $\mathcal{I}\subset\triangle(1)$ is called a \textit{primitive collection} if the members of $\mathcal{I}$ do not generate a cone of $\triangle$ but those belonging to any proper subset of $\mathcal{I}$ do. Since $\triangle$ is complete, there exists a unique cone $\sigma$ containing the vector $\sum_{\rho\in \mathcal{I}}\rho$ in its relative interior. If $\sigma=\mathbf{0}$, we call this collection \textit{centred}, and we say that the relation
	$\mathcal{P}:\sum_{\rho\in \mathcal{I}}\rho=0$
	is a \textit{centred primitive relation}. We usually write $\mathcal{I}=\mathcal{P}(1)$. Its \textit{cardinality} is $\sharp\mathcal{I}=\sharp\mathcal{P}(1)$. 
\end{definition} 
The $\mathcal{O}_X(D)$-degree (Definition \ref{def:degree}) of a centred primitive collection $\mathcal{I}$ is thus $-\sum_{\rho\in\mathcal{I}}\phi_D(\rho)$.
The following result gives sense to this notion.
\begin{theorem}[Batyrev, Chen-Fu-Hwang]\label{thm:batyrevprimitive}
	There exists a centred primitive collection for every smooth projective toric variety.
\end{theorem}
\begin{proof}
	See \cite[Proposition 3.2]{Batyrev} or \cite[Corollary 3.3]{ChenFuHwang}.
\end{proof}

We are going to show that centred primitive collections give rise to  rational curves that are ``minimal'' amongst deformation families on a smooth projective toric variety. For this we need to introduce some more notions and we refer to \cite[II.2 \& IV.2]{Kollar} for details. 
For $X$ a smooth proper uniruled variety, let $\operatorname{RatCurve^n}(X)$\index{RatCurve@$\operatorname{RatCurve^n}(X)$} be the normalized space of rational curves on $X$, $p:\operatorname{Univ^{rc}}(X)\to\operatorname{RatCurve^n}(X)$\index{UnivrcX@$\operatorname{Univ^{rc}}(X)$} be the universal family and $q:\operatorname{Univ^{rc}}(X)\to X$ be the cycle map. We say that an irreducible component $\mathcal{K}$ of $\operatorname{RatCurve^n}(X)$ is \emph{minimal} if $q|_{p^{-1}(\mathcal{K})}$ is dominant and $p^{-1}(\mathcal{K})\times_X \{x\}$ is proper for a general point $x$ of $X$. Members of $\mathcal{K}$ are called \emph{minimal rational curves}. For any line bundle $L$ on $X$, the $L$-\emph{degree} of $\mathcal{K}$ is the $L$-degree of any of its members. A minimal rational curve is free (see Section \ref{se:freeveryfree} below, i.e. its deformation family covers a general point of $X$) and does not admit any deformation that ``breaks'' into a reducible curve.

Now suppose that $X=X(\triangle)$ is smooth projective toric. Fix a centred primitive relation $\mathcal{I}=\{\rho_1,\cdots,\rho_{m+1}\}$ and its associated relation $\mathcal{P}$. For $i\in\{1,\cdots,m+1\}$, the elements in $\mathcal{I}\setminus \{\rho_i\}$ generate a cone $\sigma_i$ in $\triangle$. Let $\Sigma_{\mathcal{I}}$ be the subfan of $\triangle$ consisting of all faces of $\sigma_i,1\leqslant i\leqslant m+1$. It defines an open toric subvariety $Y$ of $X$ isomorphic to $\PP^m\times \Gm^{n-m}$. Let $\operatorname{pr}_2$ be the projection  $Y\to\Gm^{n-m}$.
Every line contained in any fibre of $\operatorname{pr}_2$ has anticanonical degree $m+1$, and its class corresponds exactly to the relation $\mathcal{P}$. The deformation family of any such line covers $Y$, which is open and dense in $X$. And for every $x\in \Gm^n\subset Y$, any deformation of a line with fixed point $x$ is a line in $\PP^m\times \{\operatorname{pr}_2(x)\}$ which is proper.
The irreducible component in $\operatorname{RatCurve^n}(X)$ corresponding to $\mathcal{P}$ is thus minimal. 

The following result shows that any minimal component arises in this way.
\begin{theorem}[Chen-Fu-Hwang \cite{ChenFuHwang}, Proposition 3.2]\label{th:BCFH}
	Let $X$ be a smooth complete toric variety. Then the minimal components of anticanonical degree $d\in\NN_{\geqslant 2}$ are in one-to-one correspondence with centred primitive collections of cardinality $d$.
\end{theorem}

The following result, due to Batyrev, shows that when $L$ is nef, rational curves belonging to the minimal components of minimal $L$-degree are indeed those of minimal $L$-degree amongst all curves passing through a general point\footnote{In general this is not always true, when $L$ is not sufficiently positive.}. 

\begin{theorem}[Batyrev]\label{thm:batyrev}
Let $X$ be a smooth projective toric variety. 
\begin{enumerate}
	\item If $L$ is ample, then a positive relation has minimal $L$-degree if and only if it is a centred primitive relation of minimal $L$-degree.
	\item If $L$ is globally generated, then the minimal $L$-degree of positive relations equals the minimal $L$-degree of centred primitive relations. 
\end{enumerate}
\end{theorem}
\begin{proof}
	If $L$ is ($\QQ-$)ample, the proof of \cite[Proposition 3.2]{Batyrev} (in a slightly different setting) actually shows that
	$\mathcal{P}\mapsto\deg_L(\mathcal{P})$ attains minimum precisely when $\mathcal{P}$ is centred primitive.
	
	Now suppose that $L=\mathcal{O}_X(D)$  is globally generated (or equivalently by Theorem \ref{thm:Demazure} (1) that $L$ is nef), where $D=\sum_{\rho\in\triangle(1)}a_\rho D_{\rho}$. Since the ample cone is the relative interior of the nef cone by Kleiman's theorem \cite[Theorem 1.4.23]{Lazarsfeld}, for every $\delta>0$ sufficiently small, we can choose an ample $\QQ$-divisor $D_\delta=\sum_{\rho\in\triangle(1)}a_{\rho,\delta}D_{\rho}$ such that for every $\rho\in\triangle(1)$, 
	$$|a_\rho-a_{\rho,\delta}|\leqslant \delta.$$ 
	For every relation $\mathcal{P}:\sum_{\rho\in\triangle(1)}b_\rho\rho=0$, let $\|\mathcal{P}\|=\max_{\rho\in\triangle(1)}|b_\rho|$. Pick a positive relation $\mathcal{P}_0:\sum_{\rho\in\triangle(1)}b_{\rho,0}\rho=0$ such that 
	$$\deg_L(\mathcal{P}_0)=\min_{\substack{\mathcal{P} \text{ positive}}}\deg_L(\mathcal{P})\leqslant \min_{\substack{\mathcal{P} \text{ centred primitive}}}\deg_L(\mathcal{P}).$$
	Now we have
	\begin{align*}
		\deg_L(\mathcal{P}_0)=\nonumber&-\sum_{\rho\in\triangle(1)}b_{\rho,0}\phi_D(\rho)\\
		\geqslant& -\sum_{\rho\in\triangle(1)}b_{\rho,0}\phi_{D_\delta}(\rho)-n\|\mathcal{P}_0\|\delta\\
		=&\deg_{\mathcal{O}_X(D_\delta)}(\mathcal{P}_0)-n\|\mathcal{P}_0\|\delta\\
		\geqslant&\min_{\substack{\mathcal{P} \text{ positive}}}\deg_{\mathcal{O}_X(D_\delta)}(\mathcal{P})-n\|\mathcal{P}_0\|\delta\\ \nonumber
		=&\min_{\substack{\mathcal{P} \text{ centred primitive}}}\deg_{\mathcal{O}_X(D_\delta)}(\mathcal{P})-n\|\mathcal{P}_0\|\delta\\ \nonumber
		\geqslant&\min_{\substack{\mathcal{P}\text{ centred primitive}}} \deg_L(\mathcal{P})-n(\|\mathcal{P}_0\|+1)\delta, \nonumber
	\end{align*} where for the last equality we use Theorem \ref{thm:batyrev} (1).
	This shows that 
		$$\min_{\substack{\mathcal{P} \text{ positive}}}\deg_L(\mathcal{P})=\min_{\substack{\mathcal{P} \text{ centred primitive}}}\deg_L(\mathcal{P}),$$
	as desired.
\end{proof}

\subsection{Free and very free curves}\label{se:freeveryfree}
For more details see \cite[II.3]{Kollar}.
\begin{definition}
	Fix $X$ a smooth variety over $K$. Let $f:\PP^1\to X$ be a rational curve and $d\in\NN$. We say that $f$ is $d$-free if 
	the sheaf $f^*T_X\otimes \mathcal{O}_{\PP^1}(-d)$ is globally generated. We write ``free'' for $0$-free and ``very free'' for $1$-free.
\end{definition}
By Grothendieck's theorem \cite[V. Exercise 2.6]{Hartshorne}, any locally free sheaf $\mathcal{F}$ of finite rank on $\PP^1$ splits, i.e., there exist integers $a_1,\cdots,a_m$ such that
$$\mathcal{F}\simeq \mathcal{O}(a_1)\oplus \cdots\oplus\mathcal{O}(a_m),$$
we define $\mumin(\mathcal{F})=\min_{1\leqslant i\leqslant m} a_i$. With this notation, $\mathcal{F}$ is ample if and only if $\mumin(\mathcal{F})\geqslant 1$\index{mumin@$\mumin$}. The rational curve defined by $f$ above is free (resp. very free) if and only if $\mumin(f^*T_X)\geqslant 0$ (resp. $\geqslant 1$).
\begin{example}[Centred primitive collections]\label{rmk:primitivecol}
	Now let $X$ be smooth projective toric of dimension $n$. One can show that for a rational curve $f:\PP^1\to X$ corresponding to a centred primitive collection $\mathcal{I}$,
	$$f^*T_X\simeq \mathcal{O}(2)\oplus \underbrace{\mathcal{O}(1)\oplus \cdots\oplus\mathcal{O}(1)}_{\sharp \mathcal{I}-2}\oplus \underbrace{\mathcal{O}\oplus\cdots\oplus \mathcal{O}}_{n-\sharp \mathcal{I}+1}.$$
	As we have seen in Section \ref{se:primitive}, such a curve is a line $l$ in the projective space $\PP^{\sharp \mathcal{I}-1}$ lying in a fibre of the toric subvariety $Y=\PP^{\sharp \mathcal{I}-1}\times \Gm^{n+1-\sharp \mathcal{I}}$ via the projection $\operatorname{pr}_2$ onto $\Gm^{n+1-\sharp \mathcal{I}}$. Hence 
	$$f^*T_X=T_{\PP^{\sharp \mathcal{I}-1}}|_l \oplus \mathcal{O}_{\PP^1}^{\oplus (n-\sharp \mathcal{I}+1)}.$$
	So, unless $X=\PP^n$, we have $\sharp \mathcal{I}<n+1$ and thus $\mu_{\operatorname{min}}(f^*T_X)=0$, hence $f$ is not very free \footnote{This gives evidence that, if $X\neq \PP^n$, for every $Q\in \mathcal{T}\simeq\Gm^{n}$, the closed subvariety $\PP^{\sharp \mathcal{I}-1}\times \{\operatorname{pr}_2(Q)\}$ is locally accumulating when approximating $Q$ (see Definition \ref{def:essconst} and Remark \ref{rmk:sweptout} below).}.
\end{example}
We now prove the following useful criterion for detecting very free curves. It also outlines a general procedure to compute the pull back of cotangent bundle for rational curves on toric varieties intersecting the open orbit. 
\begin{theorem}\label{thm:fTXtoric}
	A positive relation $\mathcal{P}$ represents very free rational curves if and only if $\Vect_\QQ\{\rho:\rho\in\mathcal{P}(1)\}=N_\QQ$.
\end{theorem}
We begin with a well-known lemma, whose proof is straightforward.
\begin{lemma}\label{le:twoquotient}
	\hfill
	\begin{enumerate}
    \item 
	Let $a_1,a_2\in\NN$ and $f_i\in H^0(\PP^1,\mathcal{O}_{\PP^1}(a_i))\setminus \{0\}$. Suppose that $\gcd(f_1,f_2)=1$. Then we have the exact sequence:
	\begin{align*}
	\begin{array}{ccccccccc}
	0& \longrightarrow&\mathcal{O}_{\PP^1}&\longrightarrow &\mathcal{O}_{\PP^1}(a_1)\oplus \mathcal{O}_{\PP^1}(a_2)&\longrightarrow &\mathcal{O}_{\PP^1}(a_1+a_2)&\longrightarrow &0\\
	& & h&\longmapsto &(hf_1,hf_2)  & & & &\\
	& & & &(g_1,g_2)&\longmapsto &f_2g_1-f_1g_2. & &
	\end{array}
	\end{align*}
	\item Let $a_1,\cdots,a_n\in\NN$ and $f_i\in H^0(\PP^1,\mathcal{O}_{\PP^1}(a_i))\setminus\{0\}$ satisfying $\gcd(f_i,f_j)=1$ for all $1\leqslant i\neq j\leqslant n$. Let $\mathcal{G}$ be the (locally free) quotient of
	$$\mathcal{F}=\bigoplus_{i=1}^{n}\mathcal{O}_{\PP^1}(a_i)$$
	by the image of $\mathcal{O}_{\PP^1}$ under $(f_i)_{i=1}^n$.
	Then we have $\mumin(\mathcal{G})\geqslant \mumin(\mathcal{F})$.
		\end{enumerate}
\end{lemma}
\begin{proof}[Proof of Theorem \ref{thm:fTXtoric}]
	Let us first discuss how to compute $f^*T_X$ for toric varieties.
	Write $\triangle(1)=\{\rho_1,\cdots,\rho_{n+r}\}$ and consider the generalized Euler exact sequence (\cite[Theorem 8.1.6]{CoxLittleSchenck}) of sheaves of $\mathcal{O}_X$-modules:
	\begin{equation*}
		\xymatrix{
			0\ar[r]&\Omega_X^1(X)\ar[r] & \bigoplus_{i=1}^{n+r}\mathcal{O}_X(-D_{\rho_i})\ar[r] & \operatorname{Pic}(X)\otimes_{\ZZ}\mathcal{O}_X\ar[r]& 0},
	\end{equation*}
	whose dual gives rise to
	\begin{equation*}\label{eq:exacteuler}
	\xymatrix{
		0\ar[r]&A_1(X)\otimes_{\ZZ} \mathcal{O}_X\ar[r] & \bigoplus_{i=1}^{n+r}\mathcal{O}_X(D_{\rho_i})\ar[r] & T_X\ar[r]& 0}.
	\end{equation*}
	Fix a positive relation $\mathcal{P}:\sum_{i=1}^{n+r} c_i\rho_i=0$. By relabelling we may assume that $\mathcal{P}(1)=\{\rho_i:c_i\neq 0\}=\{\rho_1,\cdots,\rho_m\}$. By Theorem \ref{thm:positiverelation}, let $f:\PP^1\to X$ be non-constant intersecting $\mathcal{T}$ and corresponding to $\mathcal{P}$. Choose a general lift $(f_i)_{i=1}^{n+r},f_i\in H^0(\PP^1,\mathcal{O}_{\PP^1}(c_i))$ of $f$, that is, $\gcd(f_i,f_j)=1$. The pull back by $f$ of the exact sequence above gives
		\begin{equation}\label{eq:fTX}
	\xymatrix{
		&\mathcal{O}_{\PP^1}^{\oplus r}\ar[r]^-{\phi_{\mathcal{P}}} & \bigoplus_{i=1}^{n+r}\mathcal{O}_{\PP^1}( \deg f^*(\mathcal{O}_X(D_{\rho_i})))\ar[r] & f^*T_X\ar[r]& 0},
	\end{equation}
	and we have $f^*(\mathcal{O}_X(D_{\rho_i}))\simeq\mathcal{O}_{\PP^1}(c_i)$. 
	Then we get from \eqref{eq:fTX} that
\begin{equation}\label{eq:fTX2}
	f^*T_X\simeq \bigoplus_{i=1}^{n+r}\mathcal{O}_{\PP^1}(c_i)/\operatorname{Im}(\mathcal{O}_{\PP^1}^{\oplus r}),
\end{equation}
where $\operatorname{Im}(\mathcal{O}_{\PP^1}^{\oplus r})$ is the sub-bundle generated by the image of $\mathcal{O}_{\PP^1}^{\oplus r}$ under $\phi_{\mathcal{P}}$ in $\oplus_{i=1}^{n+r}\mathcal{O}_{\PP^1}(c_i)$.
	Any non-trivial relation (not necessarily positive) $\mathcal{Q}:\sum_{i=1}^{n+r}w_i\rho _i=0$ defines a morphism 
	\begin{equation}\label{eq:relationgivemorphism}
	\begin{split}
	i_{\mathcal{Q}}:\mathcal{O}_{\PP^1}&\to \bigoplus_{\substack{i\in\{1,\cdots,n+r\}\\ c_i\neq 0}} \mathcal{O}_{\PP^1}(c_i)\hookrightarrow \bigoplus_{i=1}^{n+r} \mathcal{O}_{\PP^1}(c_i);\\
	h&\mapsto  \left(w_i h f_i\right)_{i:c_i\neq 0}.
	\end{split}
	\end{equation}
	To compute $f^*T_X$ it suffices to choose any $r$ linearly independent relations $(\mathcal{P}_j)_{1\leqslant j\leqslant r}$ and compute the sub-bundle generated by $i_{\mathcal{P}_j}(\mathcal{O}_{\PP^1}),1\leqslant j\leqslant r$.
	
We start by proving the \textbf{sufficiency}. Suppose that the $\rho_i$'s generate the ambient space $N_{\QQ}=\QQ^n$. Then $\#\mathcal{P}(1)=m\geqslant n+1$ and we may suppose that $\{\rho_1,\cdots,\rho_n\}$ is a $\QQ$-basis of $N_{\QQ}$. So for every $n+1\leqslant k\leqslant n+r$,  there exist integers $b_k\neq 0$ and $a_{j,k},1\leqslant j\leqslant n$ such that we have the following relations
\begin{equation}\label{eq:krelation}
\mathcal{Q}_k:b_k\rho_k-\sum_{i=1}^n a_{i,k}\rho_{i}=0,\quad n+1\leqslant k\leqslant n+r,
\end{equation}
each one giving rise to a morphism $i_{\mathcal{Q}_k}$ as in \eqref{eq:relationgivemorphism}. These $r$ relations are linearly independent and form a $\QQ$-basis of $A_1(X)_\QQ$.
We define 
\begin{equation}\label{eq:sheafF}
\mathcal{F}= \bigoplus_{i=1}^{n+r}\mathcal{O}_{\PP^1}( \deg f^*(\mathcal{O}_X(D_{\rho_i})))= \left(\bigoplus_{i=1}^{m}\mathcal{O}_{\PP^1}(c_i)\right)\oplus \mathcal{O}_{\PP^1}^{\oplus n+r-m},
\end{equation}
and $M_1$ (resp. $M_2$) to be the sub-bundle of $\mathcal{F}$ generated by images of the morphisms $i_{\mathcal{Q}_k},m+1\leqslant k\leqslant n+r$ (resp. $n+1\leqslant k\leqslant n+r$). We then define
$$\mathcal{G}=\mathcal{F}\diagup M_1,\quad \mathcal{H}=f^*T_X=\mathcal{F}\diagup M_2.$$
We now show that 
\begin{equation}\label{eq:quotientbystep}
\mathcal{G}\simeq \bigoplus_{i=1}^{m}\mathcal{O}_{\PP^1}(c_i).
\end{equation}
Since $c_i>0$ for all $1\leqslant i\leqslant m$, $\mathcal{G}$ is thus ample.
Since
$$ \deg f^*(\mathcal{O}_X(D_{\rho_i}))=c_i>0, 1\leqslant i\leqslant m;\quad \deg f^*(\mathcal{O}_X(D_{\rho_i}))=0,m+1\leqslant i\leqslant n+r,$$
the polynomials $f_k, m+1\leqslant k\leqslant n+r$ are of degree 0, so they are non-zero constants.
Define polynomials $g_{i,k}, (1\leqslant i\leqslant n,m+1\leqslant k\leqslant n+1)$ via the equalities
 $$b_k f_k g_{i,k}= a_{i,k}f_i.$$
 Consider the automorphism of $\mathcal{F}$ defined by
 \begin{align*}
\Psi:\mathcal{F}\longrightarrow & \mathcal{F}\\
(h_1,\cdots,h_{n+r})
\longmapsto &\left(H_1,\cdots,H_{n},h_{n+1},\cdots,h_{n+r}\right),
\end{align*}
where $H_i=h_i-\sum_{k=m+1}^{n+r}h_kg_{i,k},1\leqslant i\leqslant n$. Let us show that 
 \begin{equation}\label{eq:imageunderauto}
\begin{split}
\Psi(M_1)=\{(0,\cdots,0,F_1,\cdots,F_{n+r-m}):F_i\in\mathcal{O}_{\PP^1}\}=0\oplus \mathcal{O}_{\PP^1}^{\oplus n+r-m}\subset \mathcal{F},
\end{split}
\end{equation}
thus the claim \eqref{eq:quotientbystep} reduces to \eqref{eq:imageunderauto}. 
Indeed, for $m+1\leqslant k\leqslant n+r$, recalling \eqref{eq:krelation} \eqref{eq:sheafF}, the morphism $i_{\mathcal{Q}_{k}}$ factorises as
        $$i_{\mathcal{Q}{k}}:\mathcal{O}_{\PP^1}\to \left(\bigoplus_{i\in\{1,\cdots,n\},a_{i,k}\neq 0} \mathcal{O}_{\PP^1}(c_i)\right)\oplus \mathcal{O}_{\PP^1} \hookrightarrow\mathcal{F}.$$
Composed with the automorphism $\Psi$, it becomes
\begin{align*}
\Psi\circ i_{\mathcal{Q}_{k}}:h\longmapsto &\Psi(a_{1,k}hf_1,\cdots,a_{n,k}hf_n, 0,\cdots,0,b_{k}hf_{k},0,\cdots,0)\\
&= (0,\cdots,0,b_{k}f_{k}f,0,\cdots,0).
\end{align*}
Recall that $f_k\in K^*$ and $b_k\neq 0$. Hence
 \begin{align*}
\operatorname{Im}(\Psi\circ i _{\mathcal{Q}_{k}}(\mathcal{O}_{\PP^1}))=\{(0,\cdots,0,h,0,\cdots,0): h\in\mathcal{O}_{\PP^1}\}.
\end{align*}
This proves the claim \eqref{eq:imageunderauto}.

Now let $M_3$ denote the sub-bundle generated by the image of $M_2$ via the projection $\pi:\mathcal{F}\to\mathcal{G}$. Note that $M_3$ is generated by the images of $\Psi\circ i_{\mathcal{Q}_k},n+1\leqslant k\leqslant m$ in $\mathcal{G}$.
We arrive at, since the lift $(f_i)_{1\leqslant i\leqslant n+r}$ is general,
$$\mathcal{H}=\mathcal{F}\diagup M_2\simeq \mathcal{G}\diagup M_3\simeq  \left(\bigoplus_{i=1}^{m}\mathcal{O}_{\PP^1}(c_i)\right)\diagup M_3,$$
Since $\mathcal{G}$ is ample by \eqref{eq:quotientbystep}, it remains to apply Lemma \ref{le:twoquotient} to conclude that
$$\mu_{\operatorname{min}}(\mathcal{H})\geqslant \mu_{\operatorname{min}}(\mathcal{G})=\min_{1\leqslant i\leqslant m} (c_i)>0,$$
which says that $f^*T_X$ is also ample.

We now prove the \textbf{necessity}. Suppose that  $V=\operatorname{Vect}_\QQ\{\rho:\rho\in\mathcal{P}(1)\}\neq N_\QQ$. 
We may assume that $\{\rho_1,\cdots,\rho_q\}$ is a $\QQ$-basis of $V$ for certain $q\leqslant\min(m=\sharp \mathcal{P}(1),n-1)$, and we complete it into $\{\rho_1,\cdots,\rho_q,\rho_{m+1},\cdots,\rho_{m+n-q}\}$, a $\QQ$-basis of $N_\QQ$.
The $r$ relations (where $b_k\neq 0,a_{i,k},d_{j,k}$ are integers)
\begin{equation}\label{eq:kkrelation}
\mathcal{R}_k:b_{k}\rho_k=\sum_{i=1}^{q}a_{i,k}\rho_i+\sum_{j=1}^{n-q}d_{j,k} \rho_{m+j},\quad q+1\leqslant k\leqslant m, m+n-q+1\leqslant k\leqslant n+r.
\end{equation}
are linearly independent.
Preserving the notation $\mathcal{F}$ in \eqref{eq:sheafF}, let $M_4$ (resp. $M_5$) be the sub-bundle of $\mathcal{F}$ generated by the images of the morphisms $i_{\mathcal{R}_k},m+n-q+1\leqslant k\leqslant n+r$ (resp. $k
\in\{q+1,\cdots,m,m+n-q+1,\cdots,n+r\}$) given by the relations \eqref{eq:kkrelation} and let $\mathcal{L},\mathcal{K}$ be defined as
$$\mathcal{L}=\mathcal{F}\diagup M_4,\quad \mathcal{K}=\mathcal{F}\diagup M_5.$$
Arguing as before, on proves that (again since $(f_i)_{1\leqslant i\leqslant n+r}$ is general),
\begin{equation}\label{eq:L}
	\mathcal{L}\simeq\left(\bigoplus_{i=1}^m \mathcal{O}_{\PP^1}(c_i)\right)\oplus \mathcal{O}_{\PP^1}^{\oplus n-q}.
\end{equation}
	Denote by $M_6$ the sub-bundle of $\mathcal{L}$ generated by the images of $i_{\mathcal{R}_k},k
	\in\{q+1,\cdots,m\}$, so that $\mathcal{K}\simeq \mathcal{L}\diagup M_6$. However, since $\rho_k\in V$ for $q+1\leqslant k\leqslant m$, we have $d_{j,k}=0$ for every $1\leqslant j\leqslant n-q$. Thus for every such $k$ the morphism
	$i_{\mathcal{R}_k}$ factorises as
	$$i_{\mathcal{R}_k}:\mathcal{O}_{\PP^1}\to \bigoplus_{i=1}^m \mathcal{O}_{\PP^1}(c_i)\hookrightarrow \mathcal{F}.$$
	So up to automorphism $M_6$ is contained in the first direct-sum factor of $\mathcal{L}$ in \eqref{eq:L}. Write $\overline{M_6}$ for it.
	Consequently, $$ f^*T_X\simeq \mathcal{K}=\mathcal{L}\diagup M_6\simeq \left(\bigoplus_{i=1}^m \mathcal{O}_{\PP^1}(c_i)\diagup \overline{M_6} \right) \oplus \mathcal{O}_{\PP^1}^{\oplus n-q}$$
	is not ample since it possesses $n-q>0$ trivial factors.
\end{proof}
\section{Universal torsors and Cox coordinates}\label{se:Sectiontorsor}

The notion of \emph{universal torsors} is first introduced by Colliot-Thélène and Sansuc in \cite{ct-sansuc}. Following Salberger \cite{Salberger}, in this section we shall sketch an explicit construction in the toric setting. Let $X$ be a smooth projective toric variety with open orbit $\mathcal{T}$. We continue to use the notation in Section \ref{se:Sectiongeometry}.

\subsection{Construction and parametrization}
The exact sequence \eqref{eq:exactseqstr} gives rise to 
\begin{equation}\label{eq:toriexact}
\xymatrix{ 1\ar[r]& \mathcal{T}_{\operatorname{NS}}\ar[r]& \Gm^{\Delta(1)}\ar[r]&\mathcal{T}\ar[r]& 1}\index{TNS@$\mathcal{T}_{\operatorname{NS}},\widetilde{\mathcal{T}}_{\operatorname{NS}}$}
\end{equation}
between split tori over $K$ (see \cite[(10.2)]{Salberger}), where $\mathcal{T}_{\operatorname{NS}}$ is the Néron-Severi torus (associated to $\operatorname{NS}(X)$\index{NSX@$\operatorname{NS}(X)$}). 
 The affine space $\mathbf{A}^{\triangle(1)}$ is identified with the spectrum of the Cox ring $\operatorname{Cox}(X)$\index{Cox@$\operatorname{Cox}(X)$} \cite{Cox} of $X$, whose points are $(n+r)$-tuples $(X_{\rho})_{\rho\in\triangle(1)}$ indexed by $\triangle(1)$. Consider the open subset
\begin{equation}\label{eq:nonvanishing}
\mathfrak{T}=\mathbf{A}^{\Delta(1)}\setminus \bigcup_{\substack{\mathcal{I}\subset\Delta(1)\\ \cap_{\rho\in\mathcal{I}} D_{\rho}=\varnothing}} \left(\bigcap_{\rho\in\mathcal{I}} (X_{\rho}=0)\right)=\mathbf{A}^{\Delta(1)}\setminus \left(\bigcap_{\sigma\in\Delta_{\operatorname{max}}}\left(\prod_{\rho\not\in\sigma(1)}X_{\rho}=0\right)\right)\index{frakT@$\mathfrak{T},\widetilde{\mathfrak{T}}$}.
\end{equation}
Then we have the geometry quotient (\cite[Definition 3.1.1]{Skorobogatov}) $\pi:\mathfrak{T}\to X\simeq \mathfrak{T}\sslash \mathcal{T}_{\operatorname{NS}}$.
\begin{theorem}[Colliot-Thélène \& Sansuc \cite{ct-sansuc} \S2.3, Salberger \cite{Salberger} \S8]
	The quasi-affine variety $\mathfrak{T}$ is a universal torsor (unique up to $K$-isomorphism\footnote{Since $\mathcal{T}_{\operatorname{NS}}$ is split, the unicity follows from the Hilbert 90: $\hetale(K,\mathcal{T}_{\operatorname{NS}})=1$. See \cite[\S2.2]{ct-sansuc}.}) over $X$ under $\mathcal{T}_{\operatorname{NS}}$. \footnote{It is called \emph{principal universal torsor} in \cite[p. 191]{Salberger}}
\end{theorem}

We now write the morphism $\pi$ in coordinates. Choose a maximal cone $\sigma\in\triangle_{\max}$ and let $\sigma(1)=\{\rho_1,\cdots,\rho_n\}=\sigma\cap\triangle(1)$. Since $X$ is smooth, the lattice $N$ is generated by $\rho_1,\cdots,\rho_{n}$. Let $\{\rho_1^*,\cdots,\rho_n^*\}$ be the dual basis. Now the restriction of $\pi$ to the affine neighbourhood $U_\sigma\simeq \mathbf{A}^n$ can be written as, according to \eqref{eq:exactseqstr},
\begin{align}
\pi:\pi^{-1} U_\sigma&\longrightarrow U_\sigma \nonumber\\
(X_1,\cdots,X_{n+r})&\longmapsto \left(\prod_{j=1}^{n+r}X_{j}^{\langle\rho_i^*,\rho_j\rangle }\right)_{1\leqslant i\leqslant n}. \label{eq:unitorpara}
\end{align}

The exact sequence \eqref{eq:toriexact} clearly extends to 
$$\xymatrix{ 1\ar[r]& \widetilde{\mathcal{T}}_{\operatorname{NS}}\ar[r]& \widetilde{\Gm^{\Delta(1)}}\ar[r]&\widetilde{\mathcal{T}}\ar[r]& 1}$$
between split $\mathcal{O}_K$-tori \cite[(10.3)]{Salberger}. So does the construction of $\mathfrak{T}$  and we denote by  $\widetilde{\pi}:\widetilde{\mathfrak{T}}\to \widetilde{X}$ the morphism between toric schemes, which is a smooth $\mathcal{O}_K$-model of $\pi:\mathcal{T}\to X$ \cite[Remark 8.6 (b)]{Salberger}. If $\operatorname{Cl}_K$\index{CLK@$\operatorname{Cl}_K$} is non-trivial, in order to parametrize all rational points, it is necessary to introduce ``twisted'' torsors \cite[p. 20-22]{Skorobogatov} by elements in $\operatorname{Cl}_K^r$.
Following \cite[\S2]{Robbiani} and \cite[\S2.1]{Pieropan}, we introduce the following notation.
 Let $\mathcal{C}$ be a set of ideals as representatives of $\operatorname{Cl}_K$. For any $r$-tuple $\mathbf{c}=(\mathfrak{c}_1,\cdots,\mathfrak{c}_r)\in\mathcal{C}^r$, we identify it as a class $[\mathbf{c}]$ in $\hetale(\widetilde{X},\widetilde{\mathcal{T}}_{\operatorname{NS},\tilde{X}})$ via the morphism (see \cite[Théorème 1.5.1]{ct-sansuc}) $\operatorname{Cl}_k^r=\hetale(\Spec(\mathcal{O}_K),\widetilde{\mathcal{T}}_{\operatorname{NS}})\to \hetale(\widetilde{X},\widetilde{\mathcal{T}}_{\operatorname{NS},\tilde{X}})$.  The twisted torsor $\ctildepi:\ctildeT\to \widetilde{X}$ is a universal torsor of class $[\mathfrak{T}]-[\mathbf{c}]$ in $\hetale(\widetilde{X},\widetilde{\mathcal{T}}_{\operatorname{NS},\tilde{X}})$\index{CcrTc@$\mathcal{C},\mathbf{c},\ctildeT$}. Fix a basis $\mathcal{D}=\{[D_{\rho_1}],\cdots,[D_{\rho_r}]\}$\index{Dcal@$\mathcal{D}$} for $\operatorname{Pic}(X)$ over $\ZZ$. For a divisor $D$, write \begin{equation}\label{eq:bj}
 [D]=\sum_{j=1}^r b_j[D_{\rho_j}],\quad b_j\in\ZZ,1\leqslant j\leqslant r
 \end{equation} in terms of the basis $\mathcal{D}$. We define the fractional ideal
\begin{equation}\label{eq:cD}
\mathbf{c}^D=\prod_{j=1}^{r}\mathfrak{c}_j^{b_j}.
\end{equation}
The following result says that any rational point of $X$ admits a lift to an $\mathcal{O}_K$-point in some twist of $\mathfrak{T}$. Lifts differ by the action of the Néron-Severi torus. 
\begin{theorem}\label{thm:parabyunitor}
	The set $\ctildeT(\mathcal{O}_K)$ contains precisely the $(n+r)$-tuples $(X_{\rho})_{\rho\in\triangle(1)}\in\bigoplus_{\rho\in\triangle(1)}\mathbf{c}^{D_{\rho}}\subset K^{\triangle(1)}$ satisfying the coprimality condition
\begin{equation}\label{eq:coprime}
	\sum_{\sigma\in\triangle_{\max}}\prod_{\rho\in\triangle(1)\setminus \sigma(1)}X_{\rho}\mathbf{c}^{-D_{\rho}}=\mathcal{O}_K.
\end{equation}
	Moreover, we have
	$$X(K)=\bigsqcup_{\mathbf{c}\in\mathcal{C}^r} \widetilde{\pi_\mathbf{c}}(\ctildeT(\mathcal{O}_K)).$$
\end{theorem}
\begin{proof}
	This is a reformulation of \cite[\S2.3]{ct-sansuc} originally stated for fields. See \cite[p. 15]{Robbiani} and also \cite[Theorem 2.7]{Frei-Pieropan}, \cite[p. 419]{Pieropan}.
\end{proof}

\subsection{Heights on toric varieties}\label{se:toricheight}
In this section we follow \cite{Salberger} and derive height formulas based on the combinatorial data of the fan $\triangle$. We write $\triangle(1)=\{\rho_1,\cdots,\rho_{n+r}\}$.
Let $D=\sum_{i=1}^{n+r}a_{\rho_i} D_{\rho_i}$ be a $\mathcal{T}$-invariant divisor considered as an element of $\ZZ^{\triangle(1)}$ in the exact sequence \eqref{eq:exactseqstr}. We suppose that the line bundle $L=\mathcal{O}_X(D)$ is globally generated. 

For every $\sigma\in\triangle_{\max}$, the associated character $\chi^{m_D(\sigma)}$ of the element $m_D(\sigma)\in M$ (recall \eqref{eq:mdsigma}) generates $L$ on $U_\sigma$, and $\chi^{-m_D(\sigma)}$ lifts to a global section of $L$.
For $\nu\in \mathcal{M}_K$, $P_\nu\in X(K_\nu)$, and any $s\in H^0(X,L)$, 
we define the $\nu$-adic \emph{norm} to be
$$\|s(P_\nu)\|_{D,\nu}=\inf_{\substack{\sigma\in\triangle_{\max}\\ P_\nu\in U_\sigma(K_\nu)}}\left|\frac{s}{ \chi^{m_D(\sigma)}}(P_\nu)\right|_{\nu}.$$ 

For a point $P_0=(X_1,\cdots,X_{n+r})\in K^{\triangle(1)}$, let (whenever it is well-defined)
\begin{equation}\label{eq:XP0}
\mathbf{X}(P_0)^D=\prod_{i=1}^{n+r}X_i^{a_{\rho_i}},\index{XP0D@$\mathbf{X}(P_0)^D,\mathbf{X}(P_0)^{D(\sigma)}$}
\end{equation}
and for every $\sigma\in\triangle_{\max}$, 
let
\begin{equation}\label{eq:Dsigma}
D(\sigma)=D+\sum_{i=1}^{n+r}\langle m_D(\sigma),\rho_{i}\rangle D_{\rho_i}
=\sum_{\rho\in\triangle(1)}a_{\rho,\sigma}D_\rho,\index{Dsigma@$D(\sigma)$}
\end{equation}
where for every $\rho\in\triangle(1)$,
\begin{equation}\label{eq:arhosigma}
a_{\rho,\sigma}=a_\rho+\langle m_D(\sigma),\rho\rangle.
\end{equation}
Since $L$ is globally generated, we have $\langle m_D(\sigma),\rho\rangle\geqslant -a_{\rho}$ for every $\rho\in\triangle(1)$, and hence $a_{\rho,\sigma}\geqslant 0$. In particular $a_{\rho,\sigma}=0$ for $\rho\in\sigma(1)$. So $D(\sigma)$ is an effective divisor with support in $\cup_{\rho\in\triangle(1)\setminus\sigma(1)} D_\rho$ (see \cite[p.61--68]{Fulton} or \cite[Proposition 8.7]{Salberger}).  Viewing $D(\sigma)$ as an element of $\ZZ^{\triangle(1)}$, the expression \eqref{eq:XP0} for $\mathbf{X}(P_0)^{D(\sigma)}$ is well-defined for every $P_0=(X_1,\cdots,X_{n+r})\in K^{\triangle(1)}$.
\begin{propdef}[Salberger]\label{propdef:Salbergerheight}
	The function $H_L:X(K)\to \RR_{> 0}$\index{HL@$H_L$} defined by the formula
	$$H_L(P)=\prod_{\nu\in\mathcal{M}_K}\|s(P)\|_{D,\nu}^{-1},$$
	where $s\in H^0(X,L)$ is such that $s(P)\neq 0$, is an \emph{Arakelov Height}. It does not depend on the choice of $s$. Its equivalence class only depends on the class of $D$ in $\operatorname{Pic}(X)$. 
	Suppose that $P\in X(K)$ lifts to $P_0\in \ctildeT(\mathcal{O}_K)$ for some $\mathbf{c}\in \mathcal{C}^r$, then 
\begin{equation}\label{eq:heightsup}
	H_L(P)=\prod_{\nu\in \mathcal{M}_K}\sup_{\sigma\in\Delta_{\operatorname{max}}} |\mathbf{X}(P_0)^{D(\sigma)}|_\nu.
\end{equation}
\end{propdef}
\begin{proof}
	This is a combination of \cite[Propositions 9.2, 9.8, 10.5, 10.12 and 10.14]{Salberger}.
\end{proof}

We now generalise \cite[Proposition 11.3]{Salberger} and \cite[Proposition 2]{Pieropan} which give the formula for $H_{\omega_X^{-1}}$.
\begin{proposition}\label{prop:height}
	With the notation in Proposition-Definition \ref{propdef:Salbergerheight}, we have
	$$H_L(P)=\frac{1}{\Norm(\mathbf{c}^D)}\prod_{\nu\in\mathcal{M}_K^\infty}\sup_{\sigma\in\Delta_{\operatorname{max}}} |\mathbf{X}(P_0)^{D(\sigma)}|_\nu. \footnote{A brief reason why there are only archimedean factors left is due to the coprimality condition \eqref{eq:coprime} in Theorem \ref{thm:parabyunitor} which bounds
		$\sup_{\sigma\in\Delta_{\operatorname{max}}} |\mathbf{X}(P_0)^{D(\sigma)}|_\nu$ for all finite places $\nu$.}$$
\end{proposition}
\begin{proof}
	Write $P_0=(X_\rho)_{\rho\in\triangle(1)}\in K^{\triangle(1)}$. 
	Fix $\mathfrak{p}\in\operatorname{Spec}(\mathcal{O}_K)$ and denote by $\nu\in\mathcal{M}_K$ the corresponding place. 	We now compute the $\nu$-adic part of the height in \eqref{eq:heightsup}. For every $\rho\in\triangle(1)$, define $m_{\rho,\mathfrak{p}},\mathfrak{X}_{\rho,\mathfrak{p}}\in\ZZ$ to be
	$$m_{\rho,\mathfrak{p}}=\operatorname{ord}_{\mathfrak{p}}(\mathbf{c}^{D_\rho}),\quad \mathfrak{X}_{\rho,\mathfrak{p}}=\operatorname{ord}_{\mathfrak{p}}(X_\rho\mathcal{O}_K).$$
	Then, on recalling \eqref{eq:Dsigma} \eqref{eq:arhosigma},
\begin{equation}\label{eq:heightnorm}
\begin{split}
	&\sup_{\sigma\in\triangle_{\max}}|\mathbf{X}(P_0)^{D(\sigma)}|_\nu=\sup_{\sigma\in\triangle_{\max}}\left|\prod_{\rho\in\triangle(1)}X_{\rho}^{a_{\rho,\sigma}}\right|_\nu\\ =&\sup_{\sigma\in\triangle_{\max}}\left|\prod_{\rho\in\triangle(1)\setminus\sigma(1)}X_{\rho}^{a_{\rho,\sigma}}\right|_\nu=\Norm(\mathfrak{p})^{-\min_{\sigma\in\triangle_{\max}} \sum_{\rho\in\triangle(1)\setminus\sigma(1)}a_{\rho,\sigma}\mathfrak{X}_{\rho,\mathfrak{p}}}.
	\end{split}
\end{equation}
	Let $(b_j)_{1\leqslant j\leqslant r}$ be such that \eqref{eq:bj} holds. Because for every $\sigma\in\triangle_{\max}$, $[D]=[D(\sigma)]$ in $\operatorname{Pic}(X)$, we get the following equality (recall \eqref{eq:cD})
\begin{equation}\label{eq:CDsigna}
	\operatorname{ord}_\mathfrak{p}(\mathbf{c}^{D(\sigma)})=\sum_{\rho\in\triangle(1)\setminus\sigma(1)}a_{\rho,\sigma}m_{\rho,\mathfrak{p}}=\sum_{j=1}^rb_j m_{\rho_j,\mathfrak{p}}=\operatorname{ord}_\mathfrak{p}(\mathbf{c}^D).
\end{equation}
	Note that $X_\rho\in \mathbf{c}^{D_\rho}$ by Theorem \ref{thm:parabyunitor}, so $X_\rho \mathbf{c}^{-D_\rho}$ is an ideal of $\mathcal{O}_K$. In particular $m_{\rho,\mathfrak{p}}\leqslant\mathfrak{X}_{\rho,\mathfrak{p}}$ for every $\rho\in\triangle(1)$. Thanks to the coprimality condition \eqref{eq:coprime}, we have
	$$\min_{\sigma\in\triangle_{\max}}\operatorname{ord}_\mathfrak{p}\left(\prod_{\rho\in\triangle(1)\setminus \sigma(1)}X_{\rho}\mathbf{c}^{-D_{\rho}}\right)=0.$$
	So there exists $\sigma^\prime\in\triangle_{\max}$ such that $$\operatorname{ord}_\mathfrak{p}\left( \prod_{\rho\in\triangle(1)\setminus \sigma^\prime(1)}X_{\rho}\mathbf{c}^{-D_{\rho}}\right)=\sum_{\rho\in\triangle(1)\setminus\sigma^\prime(1)}(\mathfrak{X}_{\rho,\mathfrak{p}}-m_{\rho,\mathfrak{p}})=0.$$ Therefore $m_{\rho,\mathfrak{p}}=\mathfrak{X}_{\rho,\mathfrak{p}}$ for every $\rho\in\triangle(1)\setminus\sigma^\prime(1)$.
	With this we can now compute the exponent appearing in \eqref{eq:heightnorm}:
	\begin{align*}
	&\min_{\sigma\in\triangle_{\max}} \sum_{\rho\in\triangle(1)\setminus\sigma(1)}a_{\rho,\sigma}\mathfrak{X}_{\rho,\mathfrak{p}}\\
		=&\min_{\sigma\in\triangle_{\max}}\left(\sum_{\rho\in\triangle(1)\setminus\sigma(1)} a_{\rho,\sigma}(\mathfrak{X}_{\rho,\mathfrak{p}}-m_{\rho,\mathfrak{p}})+\sum_{\rho\in\triangle(1)\setminus\sigma(1)}a_{\rho,\sigma}m_{\rho,\mathfrak{p}}\right)\\
		=&\min_{\sigma\in\triangle_{\max}} \left(\sum_{\rho\in\triangle(1)\setminus\sigma(1)}a_{\rho,\sigma}(\mathfrak{X}_{\rho,\mathfrak{p}}-m_{\rho,\mathfrak{p}})\right)+\sum_{j=1}^rb_j m_{\rho_j,\mathfrak{p}}\\
		=&\sum_{j=1}^rb_j m_{\rho_j,\mathfrak{p}},
	\end{align*}
	which is exactly the $\mathfrak{p}$-th order of $\mathbf{c}^{D}$ by \eqref{eq:CDsigna}. 
	
	Finally returning to \eqref{eq:heightsup}, we get
	\begin{align*}
		\prod_{\nu\in\mathcal{M}_K^f}\sup_{\sigma\in\triangle_{\max}}|\mathbf{X}(P_0)^{D(\sigma)}|_\nu&=\prod_{\mathfrak{p}\in\operatorname{Spec}(\mathcal{O}_K)}\Norm(\mathfrak{p})^{-\sum_{j=1}^rb_j m_{\rho_j,\mathfrak{p}}}\\
		&=\Norm(\mathbf{c}^D)^{-1}.
	\end{align*}
	Plugging into Proposition-Definition \ref{propdef:Salbergerheight}, we get the desired formula.
	\end{proof}

\section{Approximation constants}\label{se:Sectionalpha}
In this section we recall briefly the definition of the approximation constant due to McKinnon and M. Roth. 
We refer the reader to \cite[\S2]{McKinnon-Roth1} for a detailed exposition and many illuminative examples. With the help of this $\alpha$-constant, we define the \emph{essential (approximation) constant} (Definition \ref{def:essconst}) and formulate the notion of \emph{locally accumulating varieties} and its variants (Definition \ref{def:achieve}).

Let $X$ be a projective variety over $K$, i.e. a separated reduced projective scheme of finite type over $K$. We fix a rational point $Q\in X(\overline{K})$,  $\nu\in\mathcal{M}_K$, and $L$ a line bundle, to which we associate a height function $H_L:X(K)\to\RR_{>0}$. We can define $\nu$-adic projective distance functions 
on $X(K_\nu)\times X(K_\nu)$ as in \cite[p. 522]{McKinnon-Roth1}. 
We shall frequently use the distance function (partially evaluated at $Q$) of the following form. Fix $|\cdot|_{\bar{\nu}}$ an extension of $|\cdot|_\nu$ to $\overline{K}$. Let $F$ be a finite extension of $K$ such that $Q\in X(F)$. Let $j:X\hookrightarrow \PP_K^N$ be an embedding. Choose an affine neighbourhood $U=j^{-1}(V)$ of $X$ such that $j(Q)=(x_1,\cdots,x_N)\in V(F)$, where $V\simeq \mathbf{A}_K^N$ is a standard affine chart of $\PP_K^N$ with coordinate functions $X_1,\cdots,X_N$. 
 Then
\begin{equation}\label{eq:distance}
d_\nu(\cdot,Q)= \min (1,\max_{1\leqslant i\leqslant N}(|X_i(\cdot)-x_i|_{\bar{\nu}})) \index{dist@$d_\nu(\cdot,Q)$}
\end{equation}
is a $\nu$-adic distance function on $U(K_\nu)$. Any other distance functions arising from different embeddings are equivalent (see \cite[Proposition 2.4, Lemma 2.5]{McKinnon-Roth1}). 
One can show that when $Q\in X(K)$, the function $d_\nu(\cdot,Q)^{-1}$ is a local $\nu$-adic Weil height function associated to the exceptional divisor of the blow up of $X$ at $Q$ (see for example \cite[Lemma 3.1]{McKinnon-Roth2}).

\subsection{The $\alpha$-constant after McKinnon and M. Roth}
\begin{definition}[McKinnon-M. Roth, Definitions 2.7--2.9 \cite{McKinnon-Roth1}]\label{def:appconst}
	For any subvariety $Y\subset X$, we define the (\textit{best}) \textit{approximation constant} $\alpha_{L,\nu}(Q,Y)$\index{alpha@$\alpha_{L,\nu}(Q,Y)$} (depending on $L$ and $\nu$) to be the \emph{infimum} of $A_{L,\nu}(Q,Y)$\index{AQY@$A_{L,\nu}(Q,Y)$}, where $A_{L,\nu}(Q,Y)$ is the following set
	\begin{align*}
	&\{\gamma>0: \exists C>0 ,\exists (P_i)\in (Y(K)\setminus\{Q\})^\NN, d_\nu(P_i,Q)\to 0, \\ &\qquad\text{and for every } i,d_\nu(P_i,Q)^\gamma H_L(P_i)<C\}.
	\end{align*}
\end{definition}
\begin{remarks}\label{rmk:baselocus}
	\hfill
	\begin{enumerate}
		\item  The value of $\alpha$ is independent of the choices of the distance function $d_\nu(\cdot,Q)$ and the height function $H_L$. So this notion is intrinsic for each rational point and should inherit geometric properties from the ambient variety.
		\item If $X$ is toric, then height functions are equivalent under torus action, and the same holds for distance functions as well. So $\alpha_{L,\nu}(\cdot,X)$ is constant on $\mathcal{T}(K)$. 
	\end{enumerate}

\end{remarks}
\begin{examples}\label{ex:Roth}
	\hfill
	\begin{enumerate}
		\item K. Roth's Theorem \eqref{eq:Roth} can be reformulated by using this $\alpha$-constant. 
	    Since $$H\left(\frac{p}{q}\right)=\frac{\max(|p|,|q|)}{\gcd(p,q)}$$ is an $\mathcal{O}(1)$-height, and in Definition \ref{def:appconst} the exponent $\gamma$ is on the distance function, we obtain that for $\theta\in\PP^1(\overline{\QQ})\cap\PP^1(\RR)$,
		$$\alpha_{\mathcal{O}(1),\infty}(\theta,\PP^1)=\frac{1}{\mu(\theta)}=\begin{cases}
		1 &\text{ if } \theta\in\PP^1(\QQ);\\
		\frac{1}{2} & \text{ otherwise}. 
		\end{cases}$$
		\item On combining K. Roth's theorem with the Mordell-Weil theorem, one can show that (see~\cite[p.~98]{Serre}) for $X$ any abelian variety over $K$, $L$ ample and $Q\in X(\overline{K})$, $\alpha_{L,\nu}(Q,X)=\infty$.
	\end{enumerate}
	
\end{examples}

Some useful properties of $\alpha$ are gathered together below.
\begin{proposition}[\cite{McKinnon-Roth1}, Lemma 2.13, Proposition 2.14]\label{prop:propertiesofalpha}
	We have:
	\begin{enumerate}
		\item 	Let $Q\in \PP^n(K)$. Then for every $\nu\in\mathcal{M}_K$, we have $\alpha_{\mathcal{O}(1),\nu}(Q,\PP^n)=1$.
		\item   For any $m\in \NN_{\geqslant 1}$, we have $\alpha_{mL,\nu}(Q,Y)=m\alpha_{L,\nu}(Q,Y)$.
		\item   For any subvarieties $Y_1,Y_2$ of $X$ such that $Y_1\subset Y_2$, we have $\alpha_{L,\nu}(Q,Y_1)\geqslant \alpha_{L,\nu}(Q,Y_2)$.
	\end{enumerate}
\end{proposition}

We frequently use the following two ways to estimate the approximation constant. Firstly, by using Proposition \ref{prop:propertiesofalpha} (3), we can bound $\alpha_{L,\nu}(Q,X)$ from above. Ideal candidates are the rational curves. If there exists $l$ a smooth rational curve through $Q$, then by Proposition \ref{prop:propertiesofalpha}, we get $\alpha_{L,\nu}(Q,X)\leqslant\alpha_{L,\nu}(Q,l)$, which is equal to $\deg_L l$ if $Q\in X(K)$\footnote{Edited on 06.12.2021.}. On the other hand, we have:
\begin{proposition}\label{prop:lowerbd}
	For any closed subvariety $Z$ of $X$ such that $Y\not\subset Z$ and $\inf_{P\in Z(K)}d_\nu(P,Q)>0$, consider the set $$B^Z_{L,\nu}(Q,Y)=\{\gamma\geqslant 0: \exists C>0, d_\nu(P,Q)^\gamma H_{L}(P)\geqslant C ,~\text{for every } P\in (Y\setminus Z)(K)\setminus\{Q\}\}$$ and let $$b^Z_{L,\nu}(Q,Y)=\sup B^Z_{L,\nu}(Q,Y)\index{bZ@$B^Z_{L,\nu}(Q,Y),b^Z_{L,\nu}(Q,Y)$}.$$ Then 
	\begin{enumerate}
		\item $\alpha_{L,\nu}(Q,Y)\geqslant b^Z_{L,\nu}(Q,Y)$.
		\item Assume moreover that $L$ verifies the Northcott property (see \cite[p.~530]{McKinnon-Roth1}) on $Y\setminus Z$, that is, $\#\{P\in (Y\setminus Z)(K):H_L(P)\leqslant C\}<\infty$ for any $C>0$, then $\alpha_{L,\nu}(Q,Y)= b^Z_{L,\nu}(Q,Y)$.
	\end{enumerate} 
\end{proposition}
Our definition the set $B^Z_{L,\nu}(Q,Y)$ is inspired by the classical notion of \emph{irrationality measure}. Recall that $\mu\geqslant 0$ is an irrationality measure of a real number $\theta$ if there exists $C(\mu)>0$ such that  the inequality 
$$\left|\frac{p}{q}-\theta\right|> \frac{C(\mu)}{|q|^{\mu}}$$
holds for any $\frac{p}{q}\in\QQ$.
An equivalent definition of the approximation exponent $\mu(\theta)$ is the \emph{infimum} of all irrationality measures of $\theta$.

In particular, Proposition \ref{prop:lowerbd} implies that an estimate of the shape $d_\nu(P,Q)^\gamma H(P)\geqslant C>0$ valid for every $P\in (X\setminus Z)(K)\setminus \{Q\}$ implies $\gamma\in B^Z_{L,\nu}(Q,X)$. Hence we get the lower bound $\alpha_{L,\nu}(Q,X)\geqslant b^Z_{L,\nu}(Q,X)\geqslant \gamma$.
\begin{proof}[Proof of Proposition \ref{prop:lowerbd}]
	Our argument is similar to \cite[Proposition 2.11]{McKinnon-Roth1}\footnote{Note however that our formulation of the set $B^Z_{L,\nu}(Q,Y)$ is different from \cite[Definition 2.10]{McKinnon-Roth1}.}. First it is clear that if $B^Z_{L,\nu}(Q,Y)$ is non-empty, then it is an interval: $\gamma_0\in B^Z_{L,\nu}(Q,Y)$ implies that $[0,\gamma_0]\subset B^Z_{L,\nu}(Q,Y)$.
	
	We now show (1). For any $\delta>0$, by Definition \ref{def:appconst}, we can find a sequence $(P_i)\in (Y(K)\setminus\{Q\})^\NN$ such that $d_\nu(P_i,Q)\to 0$ and that $d_\nu(P_i,Q)^{\alpha_{L,\nu}(Q,Y)+\delta} H_L(P_i)$ is bounded. This implies that $$d_\nu(P_i,Q)^{\alpha_{L,\nu}(Q,Y)+2\delta} H_L(P_i)\to 0.$$ Since $\inf_{P\in Z(K)}d_\nu(P,Q)>0$ by assumption, all but finitely many elements of $(P_i)$ are in $Y\setminus Z$. Therefore $\alpha_{L,\nu}(Q,Y)+2\delta\not\in B^Z_{L,\nu}(Q,Y)$, and hence $\alpha_{L,\nu}(Q,Y)+2\delta\geqslant b^Z_{L,\nu}(Q,Y)$. So $\alpha_{L,\nu}(Q,Y)\geqslant b_{L,\nu}(Q,Y)$.
	
	We turn to (2). Indeed, for any $\delta>0$, we can find a sequence $(P_i)\in ((Y\setminus Z)(K)\setminus\{Q\})^\NN$ such that $d_\nu(P_i,Q)^{b^Z_{L,\nu}(Q,Y)+\delta} H_L(P_i)\to 0$. 
	Since $L$ verifies the Northcott property on $Y\setminus Z$, by passing to a subsequence if necessary we may assume that $H_L(P_i)\to\infty$. Therefore we must have $d_\nu(P_i,Q)\to 0$. This shows that $b^Z_{L,\nu}(Q,Y)+\delta\in A_{L,\nu}(Q,Y)$ and hence $b^Z_{L,\nu}(Q,Y)+\delta\geqslant \alpha_{L,\nu}(Q,Y)$. This gives the desired equality.
\end{proof}
\begin{remark}\label{rmk:big}
	Assume that $L$ is big, then some power of $L$ defines a rational map $X\dashrightarrow \PP_K^{N^\prime}$ which is birational onto the image on a Zariski open dense set $U$ of $X$ (see \cite[Corollary 2.2.7]{Lazarsfeld}). Hence the line bundle $L$ verifies the Northcott property on $U$. If $Q\in U(\overline{K})$ and $Y\cap U\neq\varnothing$, then it follows from Proposition \ref{prop:lowerbd} that $\alpha_{L,\nu}(Q,Y)=b^{X\setminus U}_{L,\nu}(Q,Y)$. For instance, this is the case if $X$ is toric, $Q$ is in the open orbit $\mathcal{T}$ and $Y$ intersects with $\mathcal{T}$, because we have $\mathcal{T}\subset U$ by torus action.  
\end{remark}

\subsection{Local accumulation}
The definition of essential (approximation) constant first appeared in the work of Pagelot \cite{Pagelot} concerning statistical problems of rational points, and was heavily used in the works \cite{Huang1}, \cite{Huang2}, \cite{Huang3}. To ease notation, we shall omit the subscripts $L,\nu$ in all $\alpha$-constants as they are considered fixed throughout. 

\begin{definition}[Pagelot \cite{Pagelot}]\label{def:essconst}
	With the notation in Definition \ref{def:appconst}, we define the \textit{essential constant} of $Q$ (with respect to $Y$) to be
	$$\aess(Q,Y)=\sup_{\substack{V\subset Y}} \alpha(Q,V),\index{alphaess@$\aess(Q)$}$$
	where $V$ ranges over all Zariski open dense subvarieties\footnote{In \cite[Définition 2.3]{Huang2} and \cite[Définition 2.2]{Huang3}, the essential constant was defined by taking the supremum amongst all dense constructable subsets. This turns out to be equivalent to Definition \ref{def:essconst}, because a constructable subset is dense if and only of it contains an open dense subset.} of $Y$ such that $$\inf_{P\in V(K)} d_\nu(P,Q)=0.$$ We write $\aess(Q)=\aess(Q,X)$. 
\end{definition}

\begin{definition}\label{def:achieve}
	For $Z$ a proper closed subvariety of $X$, 
	\begin{enumerate}
		\item if $\alpha(Q,Z)=\alpha(Q,X)$, we say that \emph{the best approximations (of $Q$) can be achieved} (or \emph{the constant $\alpha(Q,X)$ can be achieved}) on $Z$;
		\item if $\aess(Q,Z)<\aess(Q)$, we say that $Z$ is \emph{locally accumulating} (with respect to $X$);
		\item if $\alpha(Q,X)=\aess(Q,Z)<\alpha(Q,X\setminus Z)$, we say that \emph{the best approximations (of $Q$) are properly achieved} (or \emph{the constant $\alpha(Q,X)$ is properly achieved}) on $Z$;
		\item if $\alpha(Q,Z)=\aess(Q)$, we say that \emph{the generic best approximations (of $Q$) can be achieved} on $Z$. 
	\end{enumerate}

\end{definition}
The notion ``locally accumulating'' first appeared in \cite[Définition 2.1]{Huang1}\footnote{Definition \ref{def:achieve} (2) is stronger than \cite[Définition 2.1]{Huang1}. But they amount to the same thing for all varieties studied in these articles because all essential constants are attainable on some open subset.}.
And the other three conventions were implicitly stated in \cite{Huang2} and \cite{Huang3}.
\begin{remarks}\label{rmk:achieve}
	Let us assume that $\alpha(Q,X)<\infty$.
	\begin{enumerate}
		\item That the best approximations can be achieved on $Z$ amounts to saying that for any $\delta>0$,
		we can find an infinite sequence $(P_i)$ of $K$-rational points, all lying in $Z$, such that $d_\nu(P_i,Q)\to 0$ and $d_\nu(P_i,Q)^{\alpha(Q,X)+\delta} H_L(P_i)$ remains bounded. 
		\item If $Z$ is locally accumulating, then for any $\gamma$ such that $\aess(Q,Z)<\gamma<\aess(Q)$, some open dense subset $U$ of $X$ contains at most finitely many rational points which are solutions of the inequality \eqref{eq:alphadef}, whilst any dense open subset of $Z$ contains an infinite sequence of such solutions. In particular $Z\cap U=\varnothing$.
		\item If the best approximations are properly achieved on $Z$, then for any infinite sequence $(P_i)\in (X(K)\setminus\{Q\})^\NN$ such that $d_\nu(P_i,Q)\to 0$ and $d_\nu(P_i,Q)^{\gamma} H_L(P_i)$ being bounded hold simultaneously with $\alpha(Q,X)<\gamma<\alpha(Q,X\setminus Z)$, then all but finitely of the $P_i$ lie in $Z$. This means that we have to restrict ourselves to $Z$ while looking for a sequence of rational points to compute $\alpha(Q,X)$. In particular, $Z$ is also locally accumulating, because $\alpha(Q,X\setminus Z)\leqslant \aess(Q)$. Moreover, for every subvariety $W$ of $Z$, since $\alpha(Q,X)\leqslant\alpha(Q,W) \leqslant\aess(Q,W)\leqslant\aess(Q,Z)$ by Proposition \ref{prop:propertiesofalpha} (3), all these inequalities are in fact equalities. This means that $Z$ does not contain any locally accumulating subvariety with respect to itself, and that $Z$ is the union of irreducible locally accumulating subvarieties $Z_0$ of $X$, each one verifying $\alpha(Q,X)=\alpha(Q,Z_0)=\aess(Q,Z_0)$.
			\end{enumerate}\end{remarks}
		
		It is easy to see that there is no locally accumulating subvariety if $\dim X=1$ or if $X$ is an abelian variety (see Examples \ref{ex:Roth}). For a fixed rationally connected variety $X$ of dimension $\geqslant 2$ (having at least one $K$-rational point), in the spirit of Conjecture \ref{conj:mckinnon} and the Principle in Section \ref{se:Sectionintro}, we expect that there exists a tower of locally accumulating subvarieties (swept out by free rational curves of varying degrees) with different essential constants, and there are only finitely many possible values of these essential constants.
		In particular $\aess(Q)<\infty$. 
		We may view this as a local analogue of finiteness of the \emph{arithmetic stratification}, conjectured by Manin \cite[LGC]{Manin}.

\subsection{An example}\label{ex:conj}
		
We give a short self-contained analysis for $S_7$ -- the toric del Pezzo surface of degree $7$. For simplicity we work over $\QQ$. We compute the $\alpha$-constants with respect to the ample anticanonical line bundle $\omega_{S_7}^{-1}$ and $\nu=\infty$, and we shall omit all these subscripts. We can assume that $S_7$ is the blow-up of $\PP^2$ (with homogeneous coordinates $[x:y:z]$) in $[1:0:0]$ and $[0:1:0]$. It is easy to see that $\overline{\operatorname{Eff}}(S_7)$ is generated by the class of (the proper transform of) the line $z=0$ and those of the two exceptional divisors $E_1,E_2$, and is therefore simplicial. Let $Q=[1:1:1]$. Let $l_1$ (resp. $l_2$) be the proper transform of the line $(x=z)$ (resp. $(y=z)$) in $\PP^2$. They have the minimal $\omega_{S_7}^{-1}$-degree $2$ amongst all curves through $Q$. Note that any other lines passing through $Q$ have degree $3$ and they cover $S_7\setminus(l_1\cup l_2\cup E_1\cup E_2)$. 

As a special case of Theorem \ref{thm:mainthm}, we claim that for $S_7$,
the best approximations are properly achieved (resp. can be achieved) on the subvariety $l_1\cup l_2$ containing minimal degree rational curves through $Q$ which are free but not very free (resp. on each $l_i,i=1,2$). Every $l_i,i=1,2$ is also locally accumulating. Similarly to Theorem \ref{thm:generic}, the generic best approximations can be achieved on every general line through $Q$, which is very free of minimal degree. Moreover, in this example there is only one possible value for the essential constant of any locally accumulating subvariety.

First of all by using Proposition \ref{prop:propertiesofalpha}, we have the upper bound $$\alpha(Q,S_7)\leqslant \alpha(Q,l_i)=2.$$
For any open dense set $U$, take a line $l$ through $Q$ different from $l_1,l_2$ such that $l\cap U\neq\varnothing$. Since $l\setminus U$ is finite, we have $\alpha(Q,l)=\alpha(Q,l\cap U)=3$, which gives the upper bound $\alpha(Q,U)\leqslant \alpha(Q,l\cap U)=3$. Hence by Definition \ref{def:essconst},
$$\aess(Q)\leqslant 3.$$

Let $(e_i)_{i\in\{1,2\}}$ be the standard basis of $\RR^2$. The fan of $S_7$ consists of $5$ rays, whose primitive generators are
$$\rho_1=e_1,~\rho_2=e_2,~\rho_3=-\rho_1,~\rho_4=-\rho_2,~\rho_5=-\rho_1-\rho_2.$$ We choose the universal torsor $\pi:\mathfrak{T}\to S_7$ embedded into $\operatorname{Spec}(\operatorname{Cox}(S_7))=\mathbf{A}^5$. Write the coordinates $(X_1,\cdots,X_5)$ for $\mathbf{A}^5$. On the affine chart $U_{\sigma_0}$, where $\sigma_0=\RR_{\geqslant 0}\rho_1+\RR_{\geqslant 0}\rho_2$, according to \eqref{eq:unitorpara}, the map $\pi$ is given by
		$$\pi:(X_1,\cdots,X_5)\longmapsto \left(\frac{X_1}{X_3X_5},\frac{X_2}{X_4X_5}\right).$$
		For every $P\in \mathcal{T}(\QQ)$, let $P_0=(X_1,\cdots,X_5)\in \mathfrak{T}(\ZZ)$ be one lift into the torsor $\mathfrak{T}$ satisfying \eqref{eq:coprime}. Note that $X_i\neq 0$ for all $1\leqslant i\leqslant 5$. 
		Define the distance function
		\begin{align*}
		d_\infty(P,Q)&=\max\left(\left|\frac{X_1}{X_3X_5}-1\right|_\infty,\left|\frac{X_2}{X_4X_5}-1\right|_\infty\right)\\
		&=\max\left(\left|\frac{X_1-X_3X_5}{X_3X_5}\right|_\infty,\left|\frac{X_2-X_4X_5}{X_4X_5}\right|_\infty\right).
		\end{align*}
	     Note that $\omega_{S_7}^{-1}=\sum_{i=1}^{5}D_{\rho_i}$ (see \cite[\S4.3]{Fulton}). To estimate the toric height function $H_{\omega_{S_7}^{-1}}$ defined in $\S$\ref{se:toricheight}, we consider the trivialization of $\omega_{S_7}^{-1}$ on $U_{\sigma_0}$, which is determined by $m_{\omega_{S_7}^{-1}}(\sigma_0)=-\rho_1^*-\rho_2^*\in (\ZZ e_1+\ZZ e_2)^\vee$. By \eqref{eq:Dsigma}, this gives rise to $$\omega_{S_7}^{-1}(\sigma_0)=2D_{\rho_{3}}+2D_{\rho_4}+3D_{\rho_5}$$
	    sitting in the class of $\omega_{S_7}^{-1}$.
	    By Proposition \ref{prop:height}, we get
		$$H_{\omega_{S_7}^{-1}}(P)\geqslant |\mathbf{X}(P_0)^{\omega_{S_7}^{-1}(\sigma_0)}|_{\infty}= |X_3^2X_4^2X_5^3|_\infty.$$
		
		Suppose $P\neq Q$, then either $X_1\neq X_3X_5$ or $X_2\neq X_4X_5$. Without loss of generality assume the first one holds. We obtain
		\begin{align*}
			d_\infty(P,Q)^2 H_{\omega_{S_7}^{-1}}(P)&\geqslant \left|\frac{X_1-X_3X_5}{X_3X_5}\right|^2_\infty |X_3^2X_4^2X_5^3|_\infty\\ &=|X_1-X_3X_5|^2|_\infty|X_4^2X_5|_\infty\geqslant 1,
		\end{align*}
		which implies that $$\alpha(Q,S_7)\geqslant b^{S_7\setminus \mathcal{T}}(Q,S_7)\geqslant 2$$ by Proposition \ref{prop:lowerbd}. If moreover $P\not\in l_1\cup l_2$, then
		$$\min(|X_1-X_3X_5|_\infty,|X_2-X_4X_5|_\infty)\geqslant 1,$$ and hence \begin{align*}
			d_\infty(P,Q)^3 H_{\omega_{S_7}^{-1}}(P)&\geqslant \left|\frac{X_1-X_3X_5}{X_3X_5}\right|^2_\infty \left|\frac{X_2-X_4X_5}{X_4X_5}\right|_\infty |X_3^2X_4^2X_5^3|_\infty\\ &=|X_1-X_3X_5|^2|_\infty|X_2-X_4X_5|_{\infty}|X_4|_\infty\geqslant 1.
		\end{align*}
		This shows that $$\aess(Q)\geqslant\alpha(Q,S_7\setminus(l_1\cup l_2))\geqslant b^{S_7\setminus \mathcal{T}}(Q,S_7\setminus(l_1\cup l_2))\geqslant 3$$ by Proposition \ref{prop:lowerbd} and Definition \ref{def:essconst}.
		
		Gathering together these bounds and those we obtained in the beginning, we get
		$$\alpha(Q,l_1\cup l_2)=\aess(Q,l_1\cup l_2)=\alpha(Q,S_7)= 2,$$ $$\alpha(Q,S_7\setminus(l_1\cup l_2))=\aess(Q)=\alpha(Q,l)=3,~\text{ for all } l\neq l_1,l_2.$$
		This proves our claim.
		\footnote{However, if $\nu$ is ultrametric, we need to take more sections of $\omega_{S_7}^{-1}$ into account, as it turns out that the single one $\mathbf{X}(P_0)^{\omega_{S_7}^{-1}(\sigma_0)}$ is insufficient. This is one of the technical point of the proof of Theorem \ref{thm:mainthm}. We postpone the details to $\S$\ref{se:SectionTheorem}.}

\section{The canonical embedding of a number field}\label{se:SectionNumber}
In this section we collect some classical useful facts about algebraic number fields and we refer to standard textbooks (e.g. \cite[\S4.2]{Samuel}) for proofs. Recall that $[K:\QQ]=r_1+2r_2$, where $r_1$ (resp. $r_2$) is the number of real (resp. complex) places of $K$ and that each $\nu\in\mathcal{M}_K^\infty$ defines an embedding $\varsigma_\nu:K\to K_\nu$.
Then the map $\varsigma=(\varsigma_{\nu_1},\cdots,\varsigma_{\nu_{r_1+r_2}}), \nu_i\in\mathcal{M}_K^\infty$\index{Knusigmanu@$K_\nu,\varsigma_\nu$} embeds $K$ into the $\RR$-vector space $\RR^{r_1}\times \CC^{r_2}$.
We want to control uniformly any non-archimedean absolute value using archimedean ones. The following simple observation can be generalised to any fractional ideal, at the expense of adding some extra constant multiple.
\begin{lemma}\label{le:geometryofnumbers}
Let $x\in\mathcal{O}_K\setminus \{0\}$. Then for every $\nu\in\mathcal{M}_K^f$, $|x|_\nu^{-1}$ divides $\prod_{i=1}^{r_1+r_2}|x|_{\nu_i}$. In particular, $|x|_\nu\geqslant \prod_{i=1}^{r_1+r_2}|x|_{\nu_i}^{-1}$
\end{lemma}
\begin{proof}
 This follows directly from the product formula. Alternatively, let $\mathfrak{p}$ denote the prime ideal correspond to $\nu$. Let $m_x=\operatorname{ord}_\mathfrak{p}(x\mathcal{O}_K)$. Then in $\RR^{r_1}\times \CC^{r_2}$, $\varsigma(x\mathcal{O}_K)$ is a sublattice of $\varsigma(\mathfrak{p}^{m_x})$. 
	We thus obtain the following divisibility relation between their co-volumes:
	\begin{equation*}
			|x|_\nu^{-1}=\Norm(\mathfrak{p}^{m_x})\mid \Norm(x\mathcal{O}_K)=|N_{K/\QQ}(x)|_\infty=\prod_{i=1}^{r_1+r_2}|x|_{\nu_i}.\qedhere
	\end{equation*}
\end{proof}

\section{Determination of $\alpha$-constants and locally accumulating subvarieties} \label{se:SectionTheorem}
The goal of this section is to prove the following detailed version of Theorem \ref{thm:mainthm}. 
Throughout this section we write $\triangle(1)=\{\rho_1,\cdots,\rho_{n+r}\}$,
and we fix \begin{equation}\label{eq:Div}
D=\sum_{i=1}^{n+r}\mathfrak{a}_{i}D_{\rho_i}\index{fraka@$\mathfrak{a}_{i}$}
\end{equation} a $\mathcal{T}$-invariant divisor and the line bundle $L=\mathcal{O}_X(D)$ on $X=X(\triangle)$, which we assume to be smooth projective of dimension at least two and split over $K$.
By torus action, we can assume that the point to be approximated is the unit element $(1,\cdots,1)$ by Remark \ref{rmk:baselocus} (2).
Define $\beta\in\NN$ as
\begin{equation}\label{eq:beta}
\beta=\min_{\mathcal{P} \text{ centred primitive}} \deg_L \mathcal{P}.\index{beta@$\beta$}
\end{equation}
	\begin{theorem}\label{thm:mainthm2}
	Suppose that $X$ verifies Hypothesis $(*)$. 
	Let $Q_0=(1,\cdots,1)\in \mathcal{T}(K)$. \index{Q0@$Q_0$}
	\begin{enumerate}
		\item Suppose that $L$ is nef. Then for every place $\nu\in\mathcal{M}_K$, 
		we have $\alpha_{L,\nu}(Q_0,X)=\beta$.
		\item Suppose that $L$ is ample and $X\neq \PP^n$. 
		Then the constant $\alpha_{L,\nu}(Q_0,X)$ \emph{is properly achieved} on a proper closed subvariety $Y$ which is a finite union of $Y_i\simeq \PP^{\mathcal{N}_i}$, each one being the fibre $\PP^{\mathcal{N}_i}\times \{1\}$ of an open toric subvariety of $X$ isomorphic to $\PP^{\mathcal{N}_i}\times \Gm^{n-\mathcal{N}_i}$. Furthermore, if there exist two different such $Y_i,Y_j$, then $Y_i\cap Y_j=Q_0$.
	\end{enumerate}
\end{theorem}
\begin{remark}\label{rmk:sweptout}
	As seen from the discussion before Theorem \ref{th:BCFH}, every such $Y_i$ is swept out by a family of minimal rational curves corresponding to a centred primitive collection of $L$-degree $\beta$ and of cardinality $\mathcal{N}_i+1$ through $Q_0$, each one realizing $\alpha_{L,\nu}(Q_0,X)$.
\end{remark}
\subsection{Some more toric geometry}
We start by proving several technical lemmas.
We first of all translate Hypothesis $(*)$ in the beginning into a combinatorial one.
\begin{lemma}\label{le:hypequivalence}
	Hypothesis $(*)$ is equivalent to
	\begin{center}
		$(**)$ there exists $\sigma_0\in\triangle_{\max}$ such that all generators in $\triangle(1)\setminus\sigma_0(1)$ are linear combinations of those in $\sigma_0(1)$ with negative integer coefficients.\index{Hyp2@Hypothesis $(**)$}
	\end{center}
\end{lemma}
\begin{proof}
The pseudo-effective cone is generated by the boundary divisors by \cite[Lemma 15.1.8]{CoxLittleSchenck}:
$$\overline{\operatorname{Eff}}(X)=\RR_{\geqslant 0}[D_{\rho_1}]+\cdots+\RR_{\geqslant 0}[D_{\rho_{n+r}}]\subset \operatorname{Pic}(X)_\RR.\index{EffX@$\overline{\operatorname{Eff}}(X)$}$$ For every $rn$-tuple of real numbers $(\mathfrak{b}_{i,j})_{1\leqslant i\leqslant n,1\leqslant j\leqslant r}$\index{bij@$\mathfrak{b}_{i,j}$}, observe the following equivalence:
\begin{equation}\label{eq:useofhypothesis1}
  [D_{\rho_i}]=\sum_{j=1}^r \mathfrak{b}_{i,j} [D_{\rho_{n+j}}],1\leqslant i\leqslant n \Leftrightarrow\rho_{n+j}=-\sum_{i=1}^n \mathfrak{b}_{i,j} \rho_i,1\leqslant j\leqslant r.
\end{equation}
Indeed, both systems of equations are equivalent to the existence of $m_i\in M_\RR,1\leqslant i\leqslant n$ such that $$\langle m_i,\rho_k\rangle=\begin{cases}
1 &\text{ if } i=k,\\ 0 &\text{ otherwise},
\end{cases}1\leqslant k\leqslant n; \quad \langle m_i,\rho_{n+j}\rangle=-\mathfrak{b}_{i,j},1\leqslant j\leqslant r,$$ and in particular, $\{\rho_1,\cdots,\rho_n\}$ is a $\RR$-basis of $N_\RR$. This is obvious for the system on the right-hand-side of \eqref{eq:useofhypothesis1} by taking $\{m_1,\cdots,m_n\}$ to be the $\RR$-dual basis of $\{\rho_1,\cdots,\rho_n\}$ and applying $m_i$ to every $\rho_{n+j},1\leqslant j\leqslant r$. 
The left system results from the image of $h(m_i)\in\RR^{\triangle(1)},1\leqslant i\leqslant n$ via the map $i$ in the exact sequence \eqref{eq:exactseqstr} tensored by $\RR$. 

  Therefore, assuming Hypothesis $(*)$, that is, by relabelling if necessary, \begin{equation}\label{eq:pseudo}
  \overline{\operatorname{Eff}}(X)=\RR_{\geqslant 0} [D_{\rho_{n+1}}]+\cdots+\RR_{\geqslant 0} [D_{\rho_{n+r}}],
  \end{equation} then we get \begin{equation}\label{eq:bij0}
  \mathfrak{b}_{i,j}\geqslant 0 \text{ for all } 1\leqslant i\leqslant n,1\leqslant j\leqslant r, 
  \end{equation}
  and hence necessarily $\rho_1,\cdots,\rho_n$ form the set of generators of a maximal cone and all $\mathfrak{b}_{i,j}$ are integers, thanks to the completeness and regularity of the fan. We conclude that 
  \begin{equation}\label{eq:useofhypothesis2}
  \sigma_0=\RR_{\geqslant 0}\rho_{1}+\cdots+\RR_{\geqslant 0}\rho_{n}\in\triangle_{\max},\index{sigma0@$\sigma_0$}
  \end{equation}
  which means that Hypothesis $(**)$ holds with $\sigma_0$. On the other hand, assuming Hypothesis $(**)$, i.e. \eqref{eq:bij0}, we deduce \eqref{eq:pseudo} and that $\mathfrak{b}_{i,j}\in\NN$ in the same way. Now the equivalence between Hypotheses $(**)$ and $(*)$ is proved. 
\end{proof}

	Under Hypothesis $(**)$, let $\sigma_0$ be as in \eqref{eq:useofhypothesis2}. The right-hand-side of \eqref{eq:useofhypothesis1} gives rise to
\begin{equation}\label{eq:primitive0}
	\mathcal{P}_{n+j}:\rho_{n+j}+\sum_{i=1}^{n}\mathfrak{b}_{i,j}\rho_i=0,\quad 1\leqslant j\leqslant r,\index{Pn+j@$\mathcal{P}_{n+j}$}
\end{equation}
	which are all positive relations. If $L$ is globally generated, then we have (recall Definition \ref{def:degree}, \eqref{eq:Div} and $\beta$ \eqref{eq:beta})
\begin{equation}\label{eq:primitive1}
	\deg_L(\mathcal{P}_{n+j})=\mathfrak{a}_{n+j}+\sum_{i=1}^{n}\mathfrak{a}_{i}\mathfrak{b}_{i,j}\geqslant \beta
\end{equation}
by Theorem \ref{thm:batyrev} (2).
We keep using the notation $$\sigma_0,\mathcal{P}_{n+j},1\leqslant j\leqslant r,(\mathfrak{a}_i)_{1\leqslant i\leqslant n+r},(\mathfrak{b}_{i,j})_{1\leqslant i\leqslant n,1\leqslant j\leqslant r}$$ throughout the rest of this section.

We next prove lemmas about centred primitive collections. The first one seems well-known.
\begin{lemma}\label{le:primitive1}
	Let $\mathcal{I}_1,\mathcal{I}_2$ be two different centred primitive collections. Then $\mathcal{I}_1\cap \mathcal{I}_2=\varnothing$.
\end{lemma}
\begin{proof}
	If there exists $\rho\in\mathcal{I}_1\cap \mathcal{I}_2$, then by Definition \ref{def:primitive}, we can write
	$$-\rho=\sum_{\rho_i\in\mathcal{I}_1\setminus \{\rho\}}\rho_{i}=\sum_{\rho_{j}\in\mathcal{I}_2\setminus\{\rho\}}\rho_{j},$$
	which yields two expressions of $-\rho$ as positive combinations of generators of cones in $\triangle$. Therefore they are the same, i.e., $\mathcal{I}_1=\mathcal{I}_2$.
\end{proof}

For $1\leqslant i\leqslant n$, let $\sigma_{i}$ denote the maximal cone \emph{adjacent} to $\sigma_0$\index{sigmai@$\sigma_i$}, i.e.
\begin{equation}\label{eq:sigmai00}
\sigma_{i}\cap \sigma_0=\RR_{\geqslant 0}\rho_1+\cdots +
\widehat{\RR_{\geqslant 0}\rho_{i}}+\cdots+\RR_{\geqslant 0} \rho_n,
\end{equation}
where ``$\widehat{\quad }$'' means this term does not appear in the summation.
The existence of exactly $n$ such maximal cones follows from the completeness and the regularity of $\triangle$ \footnote{hence each codimension $1$ cone is the common face of a unique pair of maximal cones, see \cite[Lemma 8.9]{Salberger}}.
Write for some $1\leqslant j_{i}\leqslant r$,
\begin{equation}\label{eq:sigmai0}
\sigma_{i}=\RR_{\geqslant 0}\rho_1+\cdots +
\widehat{\RR_{\geqslant 0}\rho_{i}}+\cdots+\RR_{\geqslant 0} \rho_n+\RR_{\geqslant 0}\rho_{n+j_{i}}.\index{rhonji@$\rho_{n+j_{i}}$}
\end{equation}
Our second lemma is concerned with particular coefficients of the relations $(\mathcal{P}_{n+j})_{1\leqslant j\leqslant r}$.
\begin{lemma}\label{le:biji=1}
	Under Hypothesis $(**)$, for each $1\leqslant i\leqslant n$, we have $\mathfrak{b}_{i,j_{i}}=1$.
\end{lemma}
\begin{proof}
	 The transition matrix $\mathfrak{M}$ between $\sigma_0(1)$ and $\sigma_{i}(1)$ satisfies $|\det \mathfrak{M}|=|\mathfrak{b}_{i,j_{i}}|$. Then $|\mathfrak{b}_{i,j_{i}}|=1$ because the fan $\triangle$ is regular. Therefore necessarily $\mathfrak{b}_{i,j_{i}}=1$ since $\mathfrak{b}_{i,j_{i}}\geqslant 0$ under Hypothesis $(**)$.
\end{proof}
Our next lemma says that the maximal cone $\sigma_0$ contains all except one of the elements of any centred primitive collection, so do its adjacent cones.
\begin{lemma}\label{le:primitive3}
	Under Hypothesis $(**)$, for every centred primitive collection $\mathcal{I}$, we have
	$$\#(\mathcal{I}\setminus\sigma_0(1))=1.$$
	Moreover, for each $1\leqslant i_0\leqslant n$, (recall the index $j_{i_0}$ in \eqref{eq:sigmai0},) we have $\rho_{n+j_{i_0}}\in \sigma_{i_0}(1)\cap \mathcal{I}$ if and only if $\rho_{i_0}\in \mathcal{I}$.
\end{lemma}
\begin{proof}
	Since $\mathcal{I}\not\subset\sigma(1)$ for every $\sigma\in\triangle_{\max}$, let $\rho_{n+j_1}\in \mathcal{I}\setminus \sigma_0(1)$ for certain $1\leqslant j_1\leqslant r$. We can write 
\begin{equation}\label{eq:twoposcom}
	\rho_{n+j_1}=-\sum_{i=1}^n \mathfrak{b}_{i,j_1} \rho_i=-\sum_{\rho\in\mathcal{I}\setminus\{\rho_{n+j_1}\}}\rho.
\end{equation}
	As before this also gives two expressions of $-\rho_{n+j_1}$ in terms of positive combinations of bases of cones and hence they coincide. Hence $\mathcal{I}\setminus\{\rho_{n+j_1}\}\subset \sigma_0(1)$. 
	
	Now fix $i_0$ and recall the relation $\mathcal{P}_{n+j_{i_0}}$ in \eqref{eq:primitive0}.
	 If  $\rho_{n+j_{i_0}}\in \sigma_{i_0}(1)\cap \mathcal{I}$, which means $j_{i_0}=j_1$, then $\mathfrak{b}_{i_0,j_1}=1$ by Lemma \ref{le:biji=1}. So the equality \eqref{eq:twoposcom} shows that $\rho_{i_0}\in\mathcal{I}$. Conversely, if $\rho_{i_0}\in\mathcal{I}$, by moving terms in $\mathcal{P}_{n+j_{i_0}}$, we get
	$$\rho_{n+j_{i_0}}+\sum_{i\in\{1,\cdots,n\}\setminus\{i_0\}}\mathfrak{b}_{i,j_{i_0}}\rho_i=-\rho_{i_0}=\sum_{\rho\in\mathcal{I}\setminus \{i_0\}}\rho,$$
	an equality between two positive combinations of generators of $\sigma_{i_0}$. So they coincide and in particular $\rho_{n+j_{i_0}}\in\mathcal{I}$.
\end{proof}

The crucial use of Hypothesis $(**)$ will be clear from the next proposition.
It provides us some kind of ``strong positivity'' for the relations $(\mathcal{P}_{n+j})_{j=1}^r$. In geometric terms, every curve intersecting with the open orbit $\mathcal{T}$ and the boundary divisor $D_{\rho_{n+j}}$ has $L$-degree greater than some multiple of $\beta$ \eqref{eq:beta}.
\begin{proposition}\label{prop:sigmai00}
	Suppose that $L$ is globally generated. 
	Then for every $1\leqslant i_0\leqslant n,1\leqslant j_0\leqslant r$, we have, under Hypothesis $(**)$, (recall $D=\sum_{i=1}^{n+r}\mathfrak{a}_{i}D_i$, Definition \ref{def:degree} and $\beta$ \eqref{eq:beta})
	$$\deg_L(\mathcal{P}_{n+j_0})=\mathfrak{a}_{n+j_0}+\sum_{i=1}^{n}\mathfrak{a}_{i}\mathfrak{b}_{i,j_0}\geqslant \mathfrak{b}_{i_0,j_0}\beta.$$
	Suppose that $L$ is ample. If moreover there exists $i_0\in\{1,\cdots,n\}$ (resp. $j_0\in\{1,\cdots,r\}$) such that $\rho_{i_0}$ (resp. $\rho_{n+j_0}$) does not belong to any centred primitive collections of $L$-degree $\beta$, then for all $1\leqslant j_0\leqslant r$ (resp. for all $1\leqslant i_0\leqslant n$), we have
$$\deg_L(\mathcal{P}_{n+j_0})=\mathfrak{a}_{n+j_0}+\sum_{i=1}^{n}\mathfrak{a}_{i}\mathfrak{b}_{i,j_0}> \mathfrak{b}_{i_0,j_0}\beta.$$
\end{proposition}
\begin{proof}
	We begin with the first part, i.e. assume that $L$ is globally generated. We fix indices $i_0\in\{1,\cdots,n\},j_0\in\{1,\cdots,r\}$ and look at the maximal cone $\sigma_{i_0}$ \eqref{eq:sigmai00}.
If $j_0=j_{i_0}$, the desired inequality is nothing but \eqref{eq:primitive1} because $\mathfrak{b}_{i_0,j_{i_0}}=1$ by Lemma \ref{le:biji=1}. From now on suppose $j_0\neq j_{i_0}$. 
We write $\rho_{n+j_0}$ in terms of the generators of the cone $\sigma_{i_0}$, namely $\{\rho_1,\cdots,\widehat{\rho_{i_0}},\cdots,\rho_n,\rho_{n+j_{i_0}}\}$, using the fact that $\mathfrak{b}_{i_0,j_{i_0}}=1$:
\begin{align*}
\rho_{n+j_0}&=-\sum_{i=1}^n \mathfrak{b}_{i,j_0}\rho_i\\
&=
\mathfrak{b}_{i_0,j_0}\left(\sum_{\substack{i\in\{1,\cdots,n\}-\{i_0\}}} \mathfrak{b}_{i, j_{i_0}}\rho_i+\rho_{n+j_{i_0}}\right)-\sum_{\substack{i\in\{1,\cdots,n\}-\{i_0\}}} \mathfrak{b}_{i ,j_0}\rho_i\\
&=\mathfrak{b}_{i_0 ,j_0}\rho_{n+j_{i_0}}-\sum_{\substack{i\in\{1,\cdots,n\}-\{i_0\}}} (\mathfrak{b}_{i,j_0}-\mathfrak{b}_{i_0, j_0}\mathfrak{b}_{i,j_{i_0}})\rho_i.
\end{align*}
Using the assumption that $L$ is globally generated, the piecewise linear function $\phi_D$ is convex.
In particular, its graph lies ``below'' that of the linear function $\langle m_D(\sigma_{i_0}),\cdot\rangle$, where $$m_D(\sigma_{i_0})=-\left(\sum_{i\in\{1,\cdots,n\}\setminus\{i_0\}}\mathfrak{a}_i\rho_i^* +\mathfrak{a}_{n+j_{i_0}}\rho_{n+j_{i_0}}^*\right).$$
Applying $\phi_D$ to the above equality of $\rho_{n+j_0}$ we get (by \eqref{eq:convex1})
\begin{equation}\label{eq:eq1}
	\begin{split}
		-\mathfrak{a}_{n+j_0}&=\phi_D(\rho_{n+j_0})\\
		&\leqslant \langle m_D(\sigma_{i_0}),\rho_{n+j_0}\rangle\\ &= \left(\sum_{\substack{i\in\{1,\cdots,n\}-\{i_0\}}} \mathfrak{a}_{i}(\mathfrak{b}_{i,j_0}-\mathfrak{b}_{i_0, j_0}\mathfrak{b}_{i,j_{i_0}})\right)-\mathfrak{a}_{n+j_{i_0}}\mathfrak{b}_{i_0,j_0}.
	\end{split}
\end{equation}
So again by \eqref{eq:primitive1} and that $\mathfrak{b}_{i_0,j_{i_0}}=1$,
\begin{align}
		\mathfrak{a}_{n+j_0}+\sum_{i=1}^n \mathfrak{a}_{i}\mathfrak{b}_{i,j_0}&\geqslant \mathfrak{b}_{i_0,j_0}\left(\mathfrak{a}_{n+j_{i_0}}+\sum_{i=1}^n \mathfrak{a}_i\mathfrak{b}_{i,j_{i_0}}\right)\nonumber\\
		&=\mathfrak{b}_{i_0,j_0}\deg_L(\mathcal{P}_{n+j_{i_0}})\geqslant \mathfrak{b}_{i_0,j_0}\beta.\label{eq:eq2}
\end{align}

		Now assume that $L$ is ample and let $\rho_{i_0}$ be as in the assumption. That is, $\rho_{i_0}$ is not a member of any centred primitive $\mathcal{I}$ with $\deg_L(\mathcal{I})=\beta$. Recall $\sigma_{i_0}$ and the index $j_{i_0}$ \eqref{eq:sigmai0}. Fix $j_0\in\{1,\cdots,r\}$. If $j_0=j_{i_0}$, then by Lemma \ref{le:primitive3}, $\rho_{n+j_0}\not\in\mathcal{I}(1)$ for every centred primitive $\mathcal{I}$ of $L$-degree $\beta$. Hence by Theorem \ref{thm:batyrev} (1), $\deg_L(\mathcal{P}_{n+j_0})=\mathfrak{a}_{n+j_0}+\sum_{i=1}^n \mathfrak{a}_{i}\mathfrak{b}_{i,j_0}>b_{i_0,j_{i_0}}\beta=\beta$. 
		If $j_0\neq j_{i_0}$, then $\rho_{n+j_0}\not\in\sigma_{i_0}(1)$. So the strict convexity of the function $\phi_D$ \eqref{eq:convex2} yields that the inequality \eqref{eq:eq1} above is strict. 
		Now assume that $\rho_{n+j_0}$ satisfies the second assumption. That is, $\rho_{n+j_0}$ is not a member of any centred primitive $\mathcal{I}$ with $\deg_L(\mathcal{I})=\beta$. Fix $i_0\in\{1,\cdots,n\}$. If $j_0\neq j_{i_0}$, that is, $\rho_{n+j_0}\not\in\sigma_{i_0}(1)$, then as before the inequality \eqref{eq:eq1} is strict. If $j_0=j_{i_0}$, we have $\mathfrak{b}_{i_0,j_0}=1$. Since by assumption, the positive relation $\mathcal{P}_{n+j_0}$ is not centred primitive of $L$-degree $\beta$, the inequality \eqref{eq:eq2} is now strict by Theorem \ref{thm:batyrev} (1). 
\end{proof}
\subsection{Proof of Theorem \ref{thm:mainthm2}}\label{se:proofofmainthm2}
With all these preparations, we are going to prove our main theorem.
\subsubsection{Preliminaries and sketch of the proof}
	To ease notation we shall use the simplification
	$\alpha(Q_0,Y)=\alpha_{L,\nu}(Q_0,Y)$ and ``centred primitive collection'' will be abbreviated as ``CPC''.
	We shall use the affine neighbourhood $U_{\sigma_0}$ induced by the maximal cone $\sigma_0$ \eqref{eq:useofhypothesis1}, in which the parametrization is given by (see \eqref{eq:unitorpara}):
		\begin{equation}\label{eq:parahyp}
		\begin{split}
		\pi:\pi^{-1}U_{\sigma_0}&\longrightarrow U_{\sigma_0}\\
	(X_1,\cdots,X_{n+r})&\longmapsto(y_1,\cdots,y_n)= \left(\frac{X_1}{\bY_1},\cdots,\frac{X_n}{\bY_n}\right),
		\end{split}
	\end{equation} where  for the sake of convenience we introduce, for each $1\leqslant i\leqslant n$, 
$$\boldsymbol{Y}_i:=\prod_{j=1}^r
X_{n+j}^{\mathfrak{b}_{i,j}}.$$

	To fix a parametrisation, choose $\mathcal{C}$ a set of ideals of $\mathcal{O}_K$ as representatives of $\operatorname{Cl}_K$ and choose $\mathcal{D}=\{[D_{\rho_{n+1}}],\cdots,[D_{\rho_{n+r}}]\}$ as a basis for $\operatorname{Pic}(X)$ (see \eqref{eq:useofhypothesis1}). The choice of the set $\mathcal{C}$ and the equivalent Hypothesis $(**)$ (Lemma \ref{le:hypequivalence}) guarantee that for every $r$-tuple $\mathbf{c}\in\mathcal{C}^r$, $\mathbf{c}^{D_\rho}$ is an ideal of $\mathcal{O}_K$ for every $\rho\in\triangle(1)$, so that
	$\bigoplus_{\rho\in\triangle(1)}\mathbf{c}^{D_{\rho}}\subset \mathcal{O}_K^{\triangle(1)}$. By Theorem \ref{thm:parabyunitor}, for every $P=(y_1,\cdots,y_n)\in \mathcal{T}(K)$, we can choose $\mathbf{c}\in\mathcal{C}^r$ and $P_0 =(X_{1},\cdots,X_{n+r})\in \ctildeT(\mathcal{O}_K)$ to be one lift for $P$ satisfying $X_i\neq 0,1\leqslant i\leqslant n+r$ and \eqref{eq:coprime}.
		Recall that $Q_0=(1,\cdots,1)$. For $\nu\in\mathcal{M}_K$, we shall work with the $\nu$-adic distance function (see \eqref{eq:distance})
	\begin{equation}\label{eq:distdecreasing}
		d_\nu(P,Q_0)=\min\left(1,\max_{1\leqslant i\leqslant n}\left|\frac{X_i(P_0)}{\bY_i(P_0)}-1\right|_\nu\right).\index{distQ0@$d_\nu(\cdot,Q_0)$}
	\end{equation}

	A large part of the proof is devoted to showing inequalities of the form
	$$d_\nu(P,Q_0)^\gamma H_L(P)\geqslant C>0,$$
	uniformly for $P\in\mathcal{T}(K)\setminus\{Q_0\}$.
    Before going into the long details, let us sketch the main ideas. 
    In order for $P$ to approximate $Q_0$ with respect to a fixed place $\nu\in\mathcal{M}_K$, that is,
\begin{equation*}
    \max_{1\leqslant i\leqslant n}\left|\frac{X_i}{\bY_i}-1\right|_\nu \to 0,
\end{equation*}
    if $\nu\in\mathcal{M}_K^\infty$, the $\nu$-adic values of
    the denominators $(\bY_i,1\leqslant i\leqslant n)$ and the numerators $(X_i,1\leqslant i\leqslant n)$ both tend to infinity, and they have almost equal sizes. However in ultrametric cases things are different. It is their differences $(X_i-\bY_i,1\leqslant i\leqslant n)$ that should be sufficiently divisible by powers of the prime ideal $\mathfrak{p}$ associated to $\nu$, but both of them could have very small $\mathfrak{p}$-adic orders, and hence their $\nu$-adic values could be both bounded from below. 
    Salberger's height formula (Proposition \ref{prop:height}) furnishes us some flexibility of selecting maximal cones so as to control the growth at archimedean places of the numerators and the denominators at the same time. Meanwhile the decreasing of the distance \eqref{eq:distdecreasing} can also be controlled by the contribution from all archimedean places (Lemma \ref{le:geometryofnumbers}). It remains to carefully compare them and deduce that the growth of the height ``compensates for'' the decreasing of (some power of) the distance.
    Let us now put all these ideas into practice. 
    
    We first prove part (1) in $\S$\ref{se:thm61}, then prove part (2) in $\S$\ref{se:ample}, assuming stronger positivity condition on $L$. We shall fix throughout the rest of this section a place $\nu\in\mathcal{M}_K$ and $P\in \mathcal{T}(K)\setminus\{Q_0\}$ with a fixed lift $P_0\in \ctildeT(\mathcal{O}_K)$ for certain $\mathbf{c}\in\mathcal{C}^r$.
   \subsubsection{Assume that $L$ is nef}\label{se:thm61}
     Recall the maximal cone $\sigma_0$ \eqref{eq:sigmai00}. We now determine the divisor $D(\sigma_0)$. By \eqref{eq:Dsigma},
     \begin{equation}\label{eq:Dsigma00}
    \begin{split}
    	D(\sigma_0)&=D+\sum_{\rho\in\triangle(1)}\langle m_D(\sigma_0),\rho\rangle D_\rho\\
    	&=\sum_{j=1}^{r}\left(\mathfrak{a}_{n+j}+\sum_{i=1}^{n}\mathfrak{a}_i\mathfrak{b}_{i,j}\right)D_{\rho_{n+j}}=\sum_{j=1}^{r}\deg_L(\mathcal{P}_{n+j})D_{\rho_{n+j}},
    \end{split}
     \end{equation}
   viewed as an element in $\ZZ^{\triangle(1)}$. This gives rise to the section (see \eqref{eq:XP0})
\begin{equation}\label{eq:Dsigma0}
    \mathbf{X}^{D(\sigma_0)}=\prod_{j=1}^{n+r}X_{n+j}^{\deg_L(\mathcal{P}_{n+j})}.
\end{equation}

    We next observe that, since $P\neq Q_0$, there exists $i_0\in\{1,\cdots,n\}$ such that $y_{i_0}\neq 1$ (see \eqref{eq:parahyp}). Assume that $P$ lifts to $P_0\in \ctildeT(\mathcal{O}_K)\subset  \mathcal{O}_K^{\triangle(1)}$ for some $\mathbf{c}\in \mathcal{C}^r$. Then 
    \begin{equation}\label{eq:zneq0}
    	\bY_{i_0}(P_0)-X_{i_0}(P_0)\neq 0.
    \end{equation}
    
    First let us suppose that $\nu\in\mathcal{M}_K^\infty$.  Let $S_1\subset \mathcal{M}_K^\infty\setminus\{\nu\}$ be such that $\nu'\in S_1$ implies $$\left|\bY_{i_0}(P_0)\right|_{\nu'}\geqslant \left|X_{i_0}(P_0)\right|_{\nu'},$$ and let $S_2=\mathcal{M}_K^\infty\setminus (S_1\cup\{\nu\})$, so that by the product formula,
    \begin{equation}\label{eq:dnuinfty}
    \begin{split}
    	\left|\bY_{i_0}(P_0)-X_{i_0}(P_0)\right|_{\nu} \geqslant &\prod_{\substack{\nu^\prime\in\mathcal{M}_K^\infty\\ \nu'\neq\nu}}\left|\bY_{i_0}(P_0)-X_{i_0}(P_0)\right|_{\nu'}^{-1}\\ \gg& \prod_{\substack{\nu^\prime\in\mathcal{M}_K^\infty\\ \nu'\neq\nu}}\max\left(\left|\bY_{i_0}(P_0)\right|_{\nu'},\left|X_{i_0}(P_0)\right|_{\nu'}\right)^{-1}\\ =&\prod_{\nu'\in S_1}\left|\bY_{i_0}(P_0)\right|_{\nu'}^{-1}\times \prod_{\nu'\in S_2}\left|X_{i_0}(P_0)\right|_{\nu'}^{-1}.
    \end{split}    \end{equation}  Recall the maximal cone $\sigma_{i_0}$ \eqref{eq:sigmai00} adjacent to $\sigma_0$, the index $j_{i_0}$ \eqref{eq:sigmai0} and the positive relation $\mathcal{P}_{n+j_{i_0}}$ in \eqref{eq:primitive0}. Let $c_{n+j},j\in\{1,\cdots,r\}\setminus \{j_{i_0}\}$ (resp. $c_{i_0}$) denote the coefficient of the term $D_{\rho_{n+j}}$ (resp. $D_{\rho_{i_0}}$) in $D(\sigma_{i_0})$ (also viewed as an element in $\ZZ^{\triangle(1)}$), so that
\begin{equation}\label{eq:Di0}
\mathbf{X}^{D(\sigma_{i_0})}=X_{i_0}^{c_{i_0}}\prod_{j\in\{1,\cdots,r\}\setminus \{j_{i_0}\}}X_{n+j}^{c_{n+j}}.
\end{equation}
Then by \eqref{eq:convex1} and \eqref{eq:primitive1},
\begin{equation}\label{eq:ci0}
\begin{split}
c_{i_0}&=-\phi_D(\rho_{i_0})+\langle m_D(\sigma_{i_0}),\rho_{i_0}\rangle=\mathfrak{a}_{i_0}+\langle m_D(\sigma_{i_0}),\rho_{i_0}\rangle\\
&=\mathfrak{a}_{i_0}+\mathfrak{a}_{n+j_{i_0}}+\sum_{i\in\{1,\cdots,n\}\setminus \{i_0\}} \mathfrak{a}_{i}\mathfrak{b}_{i,j_{i_0}}=\deg_L(\mathcal{P}_{n+j_{i_0}})\geqslant \beta.\end{split}\end{equation}  Appealing to \cite[Lemma 8.9, Remark 11.23]{Salberger}, we have \begin{equation}\label{eq:sigma0sigmai0}
    \frac{\mathbf{X}^{D(\sigma_{i_0})}}{\mathbf{X}^{D(\sigma_0)}}=\left(\frac{X_{i_0}}{\bY_{i_0}}\right)^{\deg_L(\mathcal{P}_{n+j_{i_0}})}.
\end{equation}
On the other hand, it follows from Proposition \ref{prop:sigmai00} that 
\begin{equation}\label{eq:sigmai00bd}
	\deg_L(\mathcal{P}_{n+j})\geqslant  \mathfrak{b}_{i_0,j}\deg_L(\mathcal{P}_{n+j_{i_0}}),\quad 1\leqslant j\leqslant r. 
\end{equation}
    Now using the height formula (Proposition \ref{prop:height}), combined with \eqref{eq:dnuinfty} \eqref{eq:sigma0sigmai0} \eqref{eq:sigmai00bd} and the product formula, we obtain
    \begin{equation}\label{eq:dnuinftybd}
    	    \begin{split}
    		&\Norm(\mathbf{c}^D)H_L(P)d_\nu(P,Q_0)^{\deg_L(\mathcal{P}_{n+j_{i_0}})}\\ \gg & \left|\frac{\bY_{i_0}(P_0)-X_{i_0}(P_0)}{\bY_{i_0}(P_0)}\right|_\nu^{\deg_L(\mathcal{P}_{n+j_{i_0}})}\prod_{\nu^\prime\in S_1\cup\{\nu\}} |\mathbf{X}(P_0)^{D(\sigma_0)}|_{\nu^\prime}\times \prod_{\nu^\prime\in S_2} |\mathbf{X}(P_0)^{D(\sigma_{i_0})}|_{\nu^\prime}\\
    		\gg& \prod_{\nu^\prime\in\mathcal{M}_K^\infty}\prod_{j=1}^{r} \left|X_{n+j}(P_0)^{\deg_L(\mathcal{P}_{n+j})-\mathfrak{b}_{i_0,j}\deg_L(\mathcal{P}_{n+j_{i_0}})}\right|_{\nu'}\gg 1.
    	\end{split}
    \end{equation}
  
   Next let us suppose that $\nu\in\mathcal{M}_K^f$. Adjusting the set $S_1$ or $S_2$ to contain $\nu$, we notice by Lemma \ref{le:geometryofnumbers} that
   \begin{equation}\label{eq:padicdistance}
   	\begin{split}
   		\left|\bY_{i_0}(P_0)-X_{i_0}(P_0)\right|_\nu
   		\geqslant &\prod_{\nu^\prime\in\mathcal{M}_K^\infty}\left|\bY_{i_0}(P_0)-X_{i_0}(P_0)\right|_{\nu^\prime}^{-1}\\
   		\gg &\prod_{\nu'\in S_1}\left|\bY_{i_0}(P_0)\right|_{\nu'}^{-1}\times \prod_{\nu'\in S_2}\left|X_{i_0}(P_0)\right|_{\nu'}^{-1}.
   	\end{split}
   \end{equation} Based on a similar argument as \eqref{eq:dnuinftybd}, we also obtain \begin{equation}\label{eq:nuhtbd}
   \Norm(\mathbf{c}^D)H_L(P)d_\nu(P,Q_0)^{\deg_L(\mathcal{P}_{n+j_{i_0}})} \gg 1.
\end{equation}

   The lower bound \eqref{eq:nuhtbd} being uniform for $P\in\mathcal{T}(K)\setminus \{Q_0\}$ and $\mathbf{c}\in\mathcal{C}^r$, we have proven that, by Proposition \ref{prop:lowerbd}, $$\alpha(Q_0,X)\geqslant b^{X\setminus\mathcal{T}}_{L,\nu}(Q_0,X)\geqslant \min_{1\leqslant i_0\leqslant n} \deg_L(\mathcal{P}_{n+j_{i_0}})=\beta.$$
    On the other hand, combining with Proposition \ref{prop:propertiesofalpha}, we get an equality since
    $$\alpha(Q_0,X)\leqslant \alpha(Q_0,l)=\beta,$$
    where $l$ is any smooth rational curve through $Q_0$ whose class is represented by a CPC of $L$-degree $\beta$ (Theorem \ref{thm:positiverelation}). This finishes the proof of Theorem \ref{thm:mainthm2} (1).
    
    \subsubsection{Assume from now on that $L$ ample and $X\neq \PP^n$}\label{se:ample}
     Let $Y$ be the union of all $Y_i$, each one being the fibre over $1\in \Gm^{n-\mathcal{N}_i}$ of the toric subvariety $\PP^{\mathcal{N}_i}\times \Gm^{n-\mathcal{N}_i}$ corresponding to a CPC $\mathcal{I}_i$ of $L$-degree $\beta$ and of cardinality $\sharp \mathcal{I}_i=\mathcal{N}_i+1$. By Theorem \ref{thm:batyrevprimitive}, $Y$ is non-empty. Moreover $Y$ is proper Zariski closed by Example \ref{rmk:primitivecol} and Remark \ref{rmk:sweptout}.
         
     To analyse the structure of $Y$, let us fix any CPC $\mathcal{I}$. By Lemma \ref{le:primitive3}, the cone $\sigma_0$ contains $\sharp\mathcal{I}-1$ elements of $\mathcal{I}$, so by relabelling one can assume that $\mathcal{I}\cap \sigma_0(1)=\{\rho_1,\cdots,\rho_{\sharp\mathcal{I}-1}\}$. Let $Y_{\mathcal{I}}$ be the fibre $\PP^{\sharp\mathcal{I}-1}\times \{1\}$ of the open toric subvariety $\PP^{\sharp\mathcal{I}-1}\times \Gm^{n-\sharp\mathcal{I}+1}$ associated to the subfan constructed from $\mathcal{I}$.  In the parametrization given by $U_{\sigma_0}$ \eqref{eq:parahyp}, we have
     \begin{equation*}
     (U_{\sigma_0}\cap Y_{\mathcal{I}})(K)=(\overbrace{*,\cdots,*}^{ \sharp\mathcal{I}-1\text{ coordinates}},1,\cdots,1).
     \end{equation*}
     We conclude from this analysis that 
     \begin{equation}\label{eq:Y}
     (U_{\sigma_0}\cap Y)(K)=\bigcup_{\mathcal{I} \text{ CPC}:\deg_L(\mathcal{I})=\beta} \{(y_1,\cdots,y_n)\in K^n:\rho_i\in\sigma_0(1)\setminus \mathcal{I}\Rightarrow y_i=1\}.
     \end{equation}
     By Lemma \ref{le:primitive1}, any two different CPCs $\mathcal{I}_1,\mathcal{I}_2$ have no common generator, so $Y_{1}\cap Y_{2}=Q_0$. 
         
    In what follows, we shall prove that there exists $\delta>0$ such that, uniformly for every place $\nu$ and for every $P\in (\mathcal{T}\setminus Y)(K)$,
    \begin{equation}\label{eq:goalcase12}
    d_\nu(P,Q_0)^{\beta+\delta}H_L(P)\geqslant C_{K,\delta,\triangle}
    \end{equation} for certain constant $C_{K,\delta,\triangle}>0$ depending only on $\delta>0$, the number field $K$, and the combinatorial data of the fan $\triangle$. Then Proposition \ref{prop:lowerbd} implies that
    $$\alpha(Q_0,X\setminus Y)=b^{X\setminus\mathcal{T}}_{L,\nu}(Q_0,X\setminus Y)\geqslant\beta+\delta>\beta.$$ 
    On the other hand, we note that, every Zariski open dense subset $U$ of $Y$ is covered by rational lines $l$ in each piece of $\PP^{\mathcal{N}_i}$ intersecting with $U$ and passing through $Q_0$ of degree $\beta$ (see Section \ref{se:primitive}), and so $\alpha(Q_0,U)=\beta$. This implies $$\alpha(Q_0,X\setminus Y)>\beta=\aess(Q_0,Y)=\alpha(Q_0,X).$$ So the constant $\alpha(Q_0,X)$ is properly achieved on $Y$ (Definition \ref{def:achieve} (3)) and \emph{a fortiori} $Y$ is locally accumulating (Definition \ref{def:achieve} (2)) by Remark \ref{rmk:achieve} (3).

    We separate our discussion into the following two cases regarding the geometry of $X$ and the coordinates of $P$. 
    
    {\bfseries Case (I).}
    \emph{Assume that there is some $\rho_{i_0}\in\sigma_0(1)$, $i_0\in\{1,\cdots,n\}$, which is not a member of any CPC of $L$-degree $\beta$, and such that $y_{i_0}(P)\neq 1$.} 
   It suffices to take $\delta_1>0$ such that the minimum of $\deg_L(\mathcal{P}_{n+j_{i_0}})$ amongst all such $i_0$ is $\geqslant \beta+\delta_1$, so that, according to \eqref{eq:nuhtbd}, $$\Norm(\mathbf{c}^D)H_L(P) d_\nu(P,Q_0)^{\beta+\delta_1}\gg 1.$$ 
   
    {\bfseries Case (II).} \emph{Assume that for every $\rho_i\in\sigma_0(1)$, either $\rho_i$ belongs to a CPC of $L$-degree $\beta$, or else $y_i(P)=1$.} 
    We would like to show the existence of a $\delta_2>0$ such that $$\Norm(\mathbf{c}^D)H_L(P) d_\nu(P,Q_0)^{\beta+\delta_2}\gg 1$$ uniformly for all such $P$.
    
    Consider the collection of pairs of generators of $\sigma_0$: \begin{equation}\label{eq:Gindices}
    	\begin{split}
    		 &\mathfrak{G}_\triangle=\{(i,k)\in\{1,\cdots,n\}^2:\rho_i\in\mathcal{I}_i,\rho_k\in\mathcal{I}_k\\ &\qquad\text{ where } \mathcal{I}_i,\mathcal{I}_k \text{ are two different CPCs of } L\text{-degree }\beta\}.\index{Gfrak@$\mathfrak{G}_\triangle$}
    	\end{split}
    \end{equation}
    Then under the assumption of \textbf{Case (II)}, necessarily $\mathfrak{G}_\triangle\neq\varnothing$ by \eqref{eq:Y}, and there exists $(i_0,k_0)\in\mathfrak{G}$ such that $y_{i_0}(P),y_{k_0}(P)\neq 1$. 
    (Otherwise $P$ would be in $Y(K)$.)  
    By relabelling if necessary, we may assume that $(i_0,k_0)=(1,2)$ and we write, for $i=1,2$, (Lemma \ref{le:primitive3}) $\{\rho_{n+i}\}=\mathcal{I}_i\setminus \sigma_0(1)$ so that $\mathcal{I}_i=\mathcal{P}_{n+i}$ and $\deg_L(\mathcal{P}_{n+1})=\deg_L(\mathcal{P}_{n+2})=\beta$.

    Concerning the parametrisation \eqref{eq:parahyp}, we have
    $$\bY_1=X_{n+1}\prod_{j=3}^{r}X_{n+j}^{\mathfrak{b}_{1,j}},\quad \bY_2=X_{n+2}\prod_{j=3}^{r}X_{n+j}^{\mathfrak{b}_{2,j}},$$
    where, if $r=2$, the product $\prod_{j=3}^{r}$ is understood as being $1$.  
    If $r\geqslant 3$, we also need to analyse the exponents $\mathfrak{b}_{i,j}$ and the relations $\mathcal{P}_{n+j}$ for $3\leqslant j\leqslant r$. Fix such a $j$. If $\rho_{n+j}$ belongs to some CPC $\mathcal{I}$, then $\mathfrak{b}_{i,j}=0$ for $i=1,2$ since otherwise $\rho_i\in\mathcal{I}_i\cap \mathcal{I}$, a contradiction to Lemma \ref{le:primitive1} and Lemma \ref{le:primitive3}. 
    If $\rho_{n+j}$ does not belong to any CPC of $L$-degree $\beta$, then by Proposition \ref{prop:sigmai00}, we conclude that in any case, 
    \begin{equation}\label{eq:bij}
    	\deg_L(\mathcal{P}_{n+j})>\mathfrak{b}_{i,j}\beta,\quad i\in\{1,2\}, j\geqslant 3.
    \end{equation}
    
    To derive lower bounds for the product of the height and the distance, we may assume as before that $\nu\in\mathcal{M}_K^f$, the archimedean cases being similar.
    A key tool is the elementary inequality 
    \begin{equation}\label{eq:geomineq}
    	\max(Z_1,Z_2)\geqslant Z_1^{\lambda_1}Z_2^{\lambda_2},\text{ for every } Z_1,Z_2,\lambda_1,\lambda_2>0 \text{ with }\lambda_1+\lambda_2=1.
    \end{equation}
    Recall that since $y_1(P),y_2(P)\neq 1$, we have $$X_1(P_0)\neq \bY_1(P_0),\quad X_2(P_0)\neq \bY_2(P_0).$$
    Using Lemma \ref{le:geometryofnumbers} and combining \eqref{eq:geomineq} with $\lambda_1=\lambda_2=\frac{1}{2}$, we can bound the distance from below via:
    \begin{equation}\label{eq:yi1}
    	\begin{split}
    		d_\nu(P,Q_0)&\geqslant \max_{i\in\{1,2\}}(|y_i(P)-1|_\nu)\\ &= \max_{i\in\{1,2\}}\left(\left|\frac{X_i(P_0)-\bY_i(P_0)}{\bY_i(P_0)}\right|_\nu\right)\\
    		&\geqslant  \max_{i\in\{1,2\}}\left(|X_i(P_0)-\bY_i(P_0)|_\nu\right) \\
    		&\geqslant \max_{i\in\{1,2\}}\left(\prod_{\nu^\prime\in\mathcal{M}_K^\infty}|X_i(P_0)-\bY_i(P_0)|_{\nu^\prime}\right)^{-1}\\
    		&\gg \max_{i\in\{1,2\}}\left(\prod_{\nu^\prime\in\mathcal{M}_K^\infty}\max(|X_i(P_0)|_{\nu^\prime},|\bY_i(P_0)|_{\nu^\prime})\right)^{-1}\\
    		&\gg \prod_{i=1}^{2}\left(\prod_{\nu^\prime\in\mathcal{M}_K^\infty}\max(|X_i(P_0)|_{\nu^\prime},|\bY_i(P_0)|_{\nu^\prime})\right)^{-\frac{1}{2}}\\
    		&=\prod_{\nu^\prime\in\mathcal{M}_K^\infty}\left(\prod_{i=1}^{2}\max(|X_i(P_0)|_{\nu^\prime},|\bY_i(P_0)|_{\nu^\prime})\right)^{-\frac{1}{2}}.
    	\end{split}
    \end{equation}

    For $i=1,2$, let $\sigma_i$ be the maximal adjacent cone such that $\sigma_0(1)\setminus \sigma_i(1)=\{\rho_i\}$, which is equivalent to $\rho_{n+i}\in \sigma_i(1)\setminus \sigma_0(1)$ by Lemma \ref{le:primitive3}. Additionally, consider $\sigma_{3}$ (resp. $\sigma_4$), the maximal cone adjacent to $\sigma_1$ (resp. $\sigma_2$) such that $\sigma_1(1)\setminus \sigma_3(1)=\{\rho_2\}$ (resp. $\sigma_2(1)\setminus\sigma_4(1)=\{\rho_1\}$).
    We claim that $\rho_1\not\in\sigma_3(1)$ and $\rho_2\not\in \sigma_4(1)$.
    Otherwise, for example if $\rho_1\in \sigma_3(1)$, since $\rho_{n+1}\in\sigma_1(1)\cap\sigma_3(1)$ and $\mathcal{I}_1\neq\mathcal{I}_2$, then we would have $\mathcal{I}_1\subset \sigma_3(1)$, which is a contradiction to Definition \ref{def:primitive}.
    Then according to the height formula (Proposition \ref{prop:height}), we have
    \begin{equation}\label{eq:heightlower0}
    \Norm(\mathbf{c}^D)H_L(P)\geqslant \prod_{\nu^\prime\in\mathcal{M}_K^\infty}\max_{0\leqslant m\leqslant 4 }\left(\left|\mathbf{X}(P_0)^{D(\sigma_m)}\right|_{\nu^\prime}\right).
    \end{equation}

In order to exploit relations between the sections above, we first write the generators of the rays in $\triangle(1)\setminus\sigma_1(1)$ in terms of those in $\sigma_1(1)$. 
\begin{align*}
	\rho_{1}&=-\sum_{i\in \mathcal{I}_1\setminus\{1\}}\rho_i -\rho_{n+1},\\
	\rho_{n+j}&=-\sum_{\substack{1\leqslant i\leqslant n \\i\in \mathcal{I}_1\setminus\{1\}}} (\mathfrak{b}_{i,j}-\mathfrak{b}_{1,j})\rho_i-\sum_{\substack{1\leqslant i\leqslant n \\ i\notin \mathcal{I}_1}} \mathfrak{b}_{i,j}\rho_i+\mathfrak{b}_{1,j}\rho_{n+1}.
\end{align*}
Therefore, if $c_2\in \mathbb{N}$ denotes the multiplicity of $D_{\rho_2}$ in $D(\sigma_3)$, by ampleness we have $c_2\geqslant 1$. Using again the fact that $\mathcal{I}_1\cap\mathcal{I}_2=\varnothing$ (Lemma \ref{le:primitive1}) and appealing to \cite[Remark 11.23]{Salberger}, we deduce the relation \begin{equation}\label{eq:Dsigma3sigma1}
	\mathbf{X}^{D(\sigma_3)}=\left(\frac{X_2}{\boldsymbol{Y}_2}\right)^{c_2}\mathbf{X}^{D(\sigma_1)}.
\end{equation}
Analysing the maximal cone $\sigma_2$ leads similarly to the relation
\begin{equation}\label{eq:Dsigma4sigma2}
		\mathbf{X}^{D(\sigma_4)}=\left(\frac{X_1}{\boldsymbol{Y}_1}\right)^{c_1}\mathbf{X}^{D(\sigma_2)},
\end{equation} where $c_1\geqslant 1$ denotes the multiplicity of $D_{\rho_1}$ in $D(\sigma_4)$.
   The maximum amongst the four sections of $L$ in \eqref{eq:heightlower0} with respect to every archimedean place $\nu^\prime$ will depend heavily on whether the maximum is achieved on $|X_i|_{\nu^\prime}$ or on $|\boldsymbol{Y}_i|_{\nu^\prime}$ for $i=1,2$ in \eqref{eq:yi1}.
    For this purpose fix such a $\nu^\prime\in\mathcal{M}_K^\infty$ in each of the following subcases.

        $$\text{\bfseries Subcase (II.1).} \quad  \max\left(|X_i(P_0)|_{\nu^\prime},\left|\boldsymbol{Y}_i(P_0)\right|_{\nu^\prime}\right)=|\bY_i(P_0)|_{\nu^\prime},\quad i=1,2.$$
       In this case  $$\max_{0\leqslant m\leqslant 4 }\left(\left|\mathbf{X}(P_0)^{D(\sigma_m)}\right|_{\nu^\prime}\right)=\left|\mathbf{X}(P_0)^{D(\sigma_0)}\right|_{\nu^\prime}.$$ 
        
$$ \text{\bfseries Subcase (II.2).}\quad \max\left(|X_i(P_0)|_{\nu^\prime},\left|\boldsymbol{Y}_i(P_0)\right|_{\nu^\prime}\right)=|X_i(P_0)|_{\nu^\prime},\quad i=1,2.$$ Then we fix $\lambda_1,\lambda_2>0,\lambda_1+\lambda_2=1$ once and for all such $\nu'$. On combining \eqref{eq:sigma0sigmai0} \eqref{eq:Dsigma3sigma1} \eqref{eq:Dsigma4sigma2},  for any $\delta>0$ sufficiently small such that
\begin{equation}\label{eq:delta21}
	\min\left(\lambda_1\beta+\lambda_2c_1,\lambda_1c_2+\lambda_2\beta\right)\geqslant \frac{1}{2}(\beta+\delta),
\end{equation} we obtain
\begin{align*}
	&\frac{\max_{0\leqslant m\leqslant 4 }\left(\left|\mathbf{X}(P_0)^{D(\sigma_m)}\right|_{\nu^\prime}\right)}{|X_1(P_0)X_2(P_0)|_{\nu'}^{\frac{1}{2}(\beta+\delta)}}\\ \geqslant &\frac{\left|\mathbf{X}(P_0)^{D(\sigma_3)}\right|_{\nu^\prime}^{\lambda_1}\left|\mathbf{X}(P_0)^{D(\sigma_4)}\right|_{\nu^\prime}^{\lambda_2}}{|X_1(P_0)X_2(P_0)|_{\nu'}^{\frac{1}{2}(\beta+\delta)}}\\ =& \left|\frac{X_1(P_0)}{\bY_1(P_0)}\right|_{\nu^\prime}^{\lambda_1\beta+\lambda_2c_1}\left|\frac{X_2(P_0)}{\bY_2(P_0)}\right|_{\nu^\prime}^{\lambda_2\beta+\lambda_1c_2}\frac{\left|\mathbf{X}(P_0)^{D(\sigma_0)}\right|_{\nu^\prime}}{|X_1(P_0)X_2(P_0)|_{\nu'}^{\frac{1}{2}(\beta+\delta)}}\\ \geqslant & \frac{\left|\mathbf{X}(P_0)^{D(\sigma_0)}\right|_{\nu^\prime}}{|\bY_1(P_0)\bY_2(P_0)|_{\nu'}^{\frac{1}{2}(\beta+\delta)}}.
\end{align*}

        {\bfseries Subcase (II.3).} 
        The remaining cases are
        \begin{equation}\label{eq:subcase31}
        	\begin{split}
        	&\max(|X_1(P_0)|_{\nu^\prime},|\bY_1(P_0)|_{\nu^\prime})=|X_1(P_0)|_{\nu^\prime}\\ \quad \text{ and }\quad &\max(|X_2(P_0)|_{\nu^\prime},|\bY_2(P_0)|_{\nu^\prime})=|\bY_2(P_0)|_{\nu^\prime};
        	\end{split}
        \end{equation}        
         or 
         \begin{equation}\label{eq:subcase32}
         \begin{split}
         &\max(|X_1(P_0)|_{\nu^\prime},|\bY_1(P_0)|_{\nu^\prime})=|\bY_1(P_0)|_{\nu^\prime}\\ \quad \text{ and }\quad&\max(|X_2(P_0)|_{\nu^\prime},|\bY_2(P_0)|_{\nu^\prime})=|X_2(P_0)|_{\nu^\prime}.
         \end{split}
         \end{equation}
     With loss of generality assuming \eqref{eq:subcase31}, we then have, for $0<\delta<\beta$ small enough
     $$\max_{0\leqslant m\leqslant 4 }\left(\left|\mathbf{X}(P_0)^{D(\sigma_m)}\right|_{\nu^\prime}\right)\geqslant \left|\mathbf{X}(P_0)^{D(\sigma_0)}\right|_{\nu^\prime}^{\frac{1}{2}(\beta-\delta)}\left|\mathbf{X}(P_0)^{D(\sigma_1)}\right|_{\nu^\prime}^{\frac{1}{2}(\beta+\delta)},$$ so that 
     \begin{align*}
     	\frac{\max_{0\leqslant m\leqslant 4 }\left(\left|\mathbf{X}(P_0)^{D(\sigma_m)}\right|_{\nu^\prime}\right)}{|X_1(P_0)\bY_2(P_0)|_{\nu^\prime}^{\frac{1}{2}(\beta+\delta)}}\geqslant  \frac{\left|\mathbf{X}(P_0)^{D(\sigma_0)}\right|_{\nu^\prime}}{|\bY_1(P_0)\bY_2(P_0)|_{\nu'}^{\frac{1}{2}(\beta+\delta)}}.
     \end{align*}
    
    Recalling \eqref{eq:bij}, we now fix $\delta_2>0$ satisfying $$\deg_L(\mathcal{P}_{n+j})\geqslant\frac{1}{2}(\beta+\delta_2)(\mathfrak{b}_{1,j}+\mathfrak{b}_{2,j}) $$ and \eqref{eq:delta21} (with respect to the fixed pair $(\lambda_1,\lambda_2)$). Summarizing what we have obtained in {\bfseries Subcases (II.1) (II.2) (II.3)} and combining \eqref{eq:yi1}, we now conclude \begin{align*}
    	&\Norm(\mathbf{c}^D)H_L(P) d_\nu(P,Q_0)^{\beta+\delta_2}\\ \gg & \prod_{\nu^\prime\in\mathcal{M}_K^\infty} \frac{\left|\mathbf{X}(P_0)^{D(\sigma_0)}\right|_{\nu^\prime}}{|\bY_1(P_0)\bY_2(P_0)|_{\nu'}^{\frac{1}{2}(\beta+\delta_2)}}\\ =& \prod_{\nu^\prime\in\mathcal{M}_K^\infty}\left(\left|X_{n+1}X_{n+2}\right|_{\nu'}^{\frac{1}{2}(\beta-\delta_2)}\prod_{j\geqslant 3} \left|X_{n+j}(P_0)^{\deg_L(\mathcal{P}_{n+j})-\frac{1}{2}(\beta+\delta_2)(\mathfrak{b}_{1,j}+\mathfrak{b}_{2,j})}\right|_{\nu'}\right)\gg 1,
    \end{align*} the implied  constant being uniform for any fixed pair of $\mathfrak{G}_\triangle$ (recall \eqref{eq:Gindices}) under the assumption of {\bf Case II}.

Therefore we complete the proof of Theorem \ref{thm:mainthm2} (2) and thereby finish the proof of Theorem \ref{thm:mainthm2}. \qed

\section{Toric varieties with Picard number $2$}\label{se:SectionPic2}
This section is devoted to studying rational approximations on split toric varieties of small Picard number.
Those of Picard number $1$ are projective spaces and the conclusion follows easily from Proposition \ref{prop:propertiesofalpha}.
We shall be interested in those of Picard number $2$ in what follows. \footnote{Amongst other arithmetic results, Mignot \cite{Mignot} succeeded in establishing the Batyrev-Manin-Peyre conjecture for some hypersurfaces in such varieties.}

We fix throughout this section $X$ a smooth projective toric variety of Picard number $2$ (constructed by the data in $\S$\ref{se:toricpic2intro}), equipped with a line bundle $L$. 
We will see that $X$ automatically satisfies Hypothesis $(*)$, so Theorem \ref{thm:mainthm} applies. We shall be interested in how the generic best approximations are obtained. The main result is Theorem \ref{thm:mainthm3}, the detailed version of Theorem \ref{thm:generic}.

	\subsection{Classification}\label{se:toricpic2intro}
	Complete smooth toric varieties whose Picard rank equals $2$ are classified by Kleinschmidt \cite{Kleinschmidt}. See also \cite[\S7.3]{CoxLittleSchenck}. They are all projective and in fact are projective bundles over projective spaces.
	Kleinschmidt gives an explicit description of their structural fans. Recall that $n$ denotes the dimension. Let $(e_i)_{1\leqslant i\leqslant n}\subset \ZZ^n$ be the canonical base of $\RR^n$. Let $s,t\geqslant 1$\index{st@$s,t$} be such that $s+t=n$ and $a_t\geqslant \cdots\geqslant a_1\geqslant 0$ are integers. Then we define the set of generators (here the labelling starts from $0$ by convention) $\triangle(1)=\{\rho_0,\cdots,\rho_{n+1}\}$ in the following way:
	\begin{align}
	&\rho_i=e_i, \quad 1\leqslant i \leqslant t \text{ and } t+1\leqslant i\leqslant t+s;\nonumber\\
	&\rho_0=-\sum_{i=1}^{t}e_i;\label{eq:relation21}\\
	&\rho_{n+1}=-\sum_{j=1}^{s}e_{t+j} +\sum_{i=1}^{t}a_i e_i.\label{eq:relation22}
	\end{align}
	Define the set of maximal cones $\triangle_{\max}=\{\sigma_{i,t+j},0\leqslant i\leqslant t,1\leqslant j\leqslant s+1\}$, where 
	$$\sigma_{i,t+j}=\RR_{\geqslant 0} \rho_0+\cdots+\widehat{\RR_{\geqslant 0}\rho_i}+\cdots+\RR_{\geqslant 0}\rho_t+\RR_{\geqslant 0}\rho_{t+1}+\cdots+\widehat{\RR_{\geqslant 0}\rho_{t+j}}+\cdots+\RR_{\geqslant 0}\rho_{n+1}.\index{sigmaitj@$\sigma_{i,t+j}$}$$
	The fan $\triangle$ is then constructed by the cones in $\triangle_{\max}$ and their faces, and the toric variety is
	$$X(\triangle)=\mathbf{P}(\mathcal{O}_{\mathbf{P}^s}\oplus\mathcal{O}_{\mathbf{P}^s}(a_1)
	\oplus\cdots \oplus \mathcal{O}_{\mathbf{P}^s}(a_t)),$$ as a projective bundle over $\mathbf{P}^s$. 
	
	\subsection{The pseudo-effective cone}
	We now show that all such varieties verify Hypotheses $(*)\Leftrightarrow(**)$\index{Hyp1@Hypothesis $(*)$}. For this we change the labelling as follows. Let
	$$v_0=\rho_t, \quad v_i=\rho_i~(1\leqslant i\leqslant t-1\text{ and } t+1\leqslant i\leqslant n+1),\quad v_t=\rho_0.$$
	So that equations \eqref{eq:relation21} and \eqref{eq:relation22} become
	\begin{align}
	&v_0=-\sum_{i=1}^{t}v_i,\label{eq:relation31}\\
	&v_{n+1}=-\sum_{i=1}^{t}b_i v_i-\sum_{j=1}^{s}v_{t+j},\label{eq:relation32}
	\end{align}
	where \begin{equation}\label{eq:bi}
	b_t=a_t,\quad b_i=a_t-a_i,~(1\leqslant i\leqslant t-1)
	\end{equation}
	satisfy $b_t\geqslant b_1\geqslant b_2\geqslant\cdots\geqslant b_{t-1}\geqslant 0$.
		Geometrically, this operation is nothing but the isomorphism \cite[Lemma 7.9]{Hartshorne}
	\begin{align*}
	&\mathbf{P}(\mathcal{O}_{\mathbf{P}^s}\oplus\mathcal{O}_{\mathbf{P}^s}(a_1)
	\oplus\cdots \oplus \mathcal{O}_{\mathbf{P}^s}(a_t))\\
	\simeq & \mathbf{P}(\mathcal{O}_{\mathbf{P}^s}(-b_t)\oplus\mathcal{O}_{\mathbf{P}^s}(-b_1)
	\oplus\cdots \oplus \mathcal{O}_{\mathbf{P}^s}(-b_{t-1})\oplus \mathcal{O}_{\mathbf{P}^s}).
	\end{align*}
	In this way we get two generators which are combinations of other generators with negative coefficients. Consequently, the cone 
	\begin{equation}\label{eq:conehyp2}
	\sigma_{t,n+1}=\RR_{\geqslant 0}v_1+\cdots+\RR_{\geqslant 0}v_n
	\end{equation} 
	satisfies \eqref{eq:useofhypothesis1}, i.e., Hypothesis $(**)$\index{Hyp2@Hypothesis $(**)$} is verified. Equivalently, with the notation $b_0=a_t$ (recall that $D_{\rho_t}$ corresponds to $\rho_t=v_0$),
	$$[D_{\rho_i}]=[D_{\rho_t}]+b_i[D_{\rho_{n+1}}]\quad (0\leqslant i\leqslant t-1);\quad [D_{\rho_{t+j}}]=[D_{\rho_{n+1}}]\quad (1\leqslant j\leqslant s),$$
	therefore
	\begin{equation}\label{eq:simplicialeffectivecone}
	\overline{\operatorname{Eff}}(X)=\RR_{\geqslant 0}[D_{\rho_t}]+\RR_{\geqslant 0}[D_{\rho_{n+1}}].\index{EffX@$\overline{\operatorname{Eff}}(X)$}
    \end{equation}

\subsection{Free rational curves}
	Equations \eqref{eq:relation21} and \eqref{eq:relation22} furnish the following two $1$-cycles with corresponding relations
	\begin{equation}\label{eq:cycle1}
	C_1:\sum_{i=0}^t \rho_i=0;\index{C1@$C_1$}
	\end{equation}
	\begin{equation}\label{eq:cycle2}
	C_2:\sum_{j=1}^{s+1}\rho_{r+j}-\sum_{i=1}^{t}a_i\rho_i=0,\index{C2@$C_2$}
	\end{equation}
	the first one being positive. The second one is not positive in general, unless all $a_i=0$, i.e. $X=\PP^s\times\PP^t$. 
	They give rise to two primitive collections (Definition \ref{def:primitive}) $C_1(1)=\{\rho_0,\cdots,\rho_t\}$ and $C_2(1)=\{\rho_{t+1},\cdots,\rho_{n+1}\}$.
	\begin{proposition}\label{prop:nex}
		The semi-group $\operatorname{AE}_1(X)$ of effective $1$-cycles is generated by $C_1$ and $C_2$:
			$$\operatorname{AE}_1(X)=\NN C_1+\NN C_2\subset A_1(X).\index{AE1X@$\operatorname{AE}_1(X)$}$$
	\end{proposition}
\begin{proof}
The relations $C_1$ and $C_2$ are linearly independent and primitive, so they generate the group $A_1(X)$.
		Let $C\in \operatorname{AE}_1(X)$ be the class of an effective curve $E$. Then there exist $p,q\in\ZZ$ such that
		$$C=pC_1+qC_2.$$
		We want to show that $p,q\geqslant 0$. The relation corresponding to $C$ is 
		\begin{equation}\label{eq:pqrelation}
		p\rho_0+\sum_{i=1}^{t}(p-qa_i)\rho_i+\sum_{j=1}^{s+1}q\rho_{t+j}=0.
		\end{equation}
		 Recall that $\{\rho_0,\cdots,\rho_t\},\{\rho_{t+1},\cdots,\rho_{n+1}\}$ are primitive collections. In particular none of them is contained in any maximal cone (Definition \ref{def:primitive}). So $\cap_{i=0}^t D_{\rho_i}=\cap_{j=1}^{s+1}D_{\rho_{t+j}}=\varnothing$. If $q<0$, one would have $\langle D_{\rho_{t+j}},E\rangle=q<0$ for all $1\leqslant j\leqslant s+1$ and so $E\subset \cap_{j=1}^{s+1}D_{\rho_{t+j}}=\varnothing$, which is absurd. 
		 So $q$ must be non-negative.
		 Similarly if $p<0$, then $\langle D_{\rho_i},E\rangle=p-qa_i\leqslant\langle D_{\rho_0},E\rangle =p<0$ for every $1\leqslant i\leqslant t$, which is again impossible since $\cap_{i=0}^t D_{\rho_i}=\varnothing$. 
\end{proof}

Equation \eqref{eq:relation32} furnishes two relations
\begin{equation}\label{eq:cycle3}
C_3:\sum_{i=1}^{t}b_iv_i+\sum_{j=1}^{s+1}v_{t+j}=b_t\rho_0+\sum_{i=1}^{t-1}b_i\rho_i+\sum_{j=1}^{s+1}\rho_{t+j}=0,\index{C3@$C_3$}
\end{equation}
\begin{equation}\label{eq:C1C3}
C_1+C_3: (b_t+1)\rho_0+\rho_t+\sum_{i=1}^{t-1}(b_i+1) \rho_i+\sum_{j=1}^{s+1}\rho_{t+j}=0.\index{C1C3@$C_1+C_3$}
\end{equation}
which are positive and verify (viewed as elements in $\ZZ^{\triangle(1)}$)
$$C_3=a_tC_1+C_2,\quad C_1+C_3=(a_t+1)C_1+C_2.$$
\begin{lemma}\label{le:veryfreec1c3}
	The relation $C_1+C_3$ represents very free rational curves. So does the relation $C_3$ if $b_i\neq 0$ for all $i\in\{1,\cdots,t\}$.
\end{lemma}
\begin{proof}
	The relation $C_1+C_3$ \eqref{eq:C1C3} verifies $$\operatorname{Vect}_\QQ\{\rho:\rho\in(C_1+C_3)(1)\}=\Vect_\QQ\{\rho_0,\cdots,\rho_{n+1}\}=N_\QQ,$$
	so it represents very free rational curves by Theorem \ref{thm:fTXtoric}. 
	
	If $b_i\neq 0$ for all $i\in\{1,\cdots,t\}$, then $\rho_i\in C_3(1)$ for all $i\neq t$ and so $$\operatorname{Vect}_\QQ\{\rho:\rho\in C_3(1)\}=\Vect_\QQ\{\rho_0,\cdots,\rho_{t-1},\rho_{t+1},\cdots,\rho_{n+1}\}=N_\QQ.$$ Theorem \ref{thm:fTXtoric} tells us that some curve of class $C_3$ is very free.
\end{proof}
\begin{remark*}
	Following the strategy of proving Theorem \ref{thm:fTXtoric}, one can show that for every rational curve $l$ intersecting with $\mathcal{T}$, $ T_X|_l$ equals
	\begin{equation*}
	\tiny\begin{cases}
	\mathcal{O}_{\PP^1}(2)\bigoplus \mathcal{O}_{\PP^1}(1)^{\oplus s-1}\bigoplus \left(\bigoplus_{i=1}^t \mathcal{O}_{\PP^1}(b_i)\right)  &\text{if for all } i,b_i\neq 0,[l]=C_3;\\
	\mathcal{O}_{\PP^1}(2)^{\oplus 2}\bigoplus\mathcal{O}_{\PP^1}(1)^{\oplus s-1} \bigoplus\left(\bigoplus_{i\in \{1,\cdots ,t\}\setminus\{t-1\}}\mathcal{O}_{\PP^1}(b_i+1)\right) &\text{if there exists }b_i=0,[l]=C_1+C_3.
	\end{cases}
	\end{equation*}
\end{remark*}

\subsection{The nef cone and the big cone}\label{se:bigness}
Since $X$ has small Picard number, it is relatively easy to determine whether a divisor is big, nef or ample.
Let $D$ be an effective $\mathcal{T}$-invariant $\QQ$-divisor. Write $L=\mathcal{O}_X(D)$ and $[D]=A[D_{\rho_t}]+B[D_{\rho_{n+1}}]$ for $A,B\in\QQ_{\geqslant 0}$ by \eqref{eq:simplicialeffectivecone}.
\begin{lemma}\label{le:bigness}
	\hfill
	\begin{enumerate}
		\item The line bundle $L$ is nef (resp. ample) if and only if $B\geqslant Aa_t$ (resp. $A>0$ and $B>Aa_t$).
		\item The line bundle $L$ is big if and only if $AB>0$. 
	\end{enumerate}
\end{lemma} 
In particular $\operatorname{Nef}(X)\subsetneq\overline{\operatorname{Eff}}(X)$ unless $a_t=0$, i.e., $X\simeq \PP^s\times\PP^t$ where they coincide.
\begin{proof}
	Statement (1) follows from the toric Nakai-Moishezon criterion (see \cite[Theorem 2]{Kleinschmidt} or \cite[Theorem 6.3.13]{CoxLittleSchenck}) combined with Proposition \ref{prop:nex}, because $\deg_L C_1=A,\deg_L C_2=B-Aa_t$. We now quickly show (2) using Theorem \ref{thm:Demazure} (2). \footnote{We can also use the fact that the big cone is the relative interior of $\overline{\operatorname{Eff}}(X)$ by \cite[Theorem 2.2.26]{Lazarsfeld}.}  To this end we want to find $m=\sum_{i=1}^{n}c_iv_i^*\in M_\QQ$ as an interior point of the polyhedron $P_D$, i.e. such that $$\langle m,v_i\rangle=c_i>0,1\leqslant i\leqslant n,$$
	$$\langle m,v_0\rangle=-\sum_{i=1}^{t}c_i>-A\Leftrightarrow \sum_{i=1}^{n}c_i<A,$$
	$$\langle m,v_{n+1}\rangle=-\sum_{i=1}^{t}c_ib_i-\sum_{j=1}^{s}c_{t+j}>-B\Leftrightarrow \sum_{i=1}^{t}c_ib_i+\sum_{j=1}^{s}c_{t+j}<B.$$
	Once $AB>0$, such $(c_i)_{1\leqslant i\leqslant n}$ certainly exists (any sufficiently small $c_i$ suffice), and conversely if such $(c_i)_{1\leqslant i\leqslant n}$ exists then necessarily $AB>0$. This finishes the proof of the lemma. 
\end{proof}

\begin{proposition}\label{prop:veryfreecurvedegree}
Let $D$ be a $\mathcal{T}$-invariant divisor and $L=\mathcal{O}_X(D)$.
\begin{enumerate}
	\item Assume that $L$ is nef, then the minimal $L$-degree of very free rational curves are	\begin{equation*}
	\begin{cases}
	\deg_L C_3 & \text{ if for all }i\in\{1,\cdots,t\},b_i\neq 0;\\
	\deg_L C_1+ \deg_L C_3 &  \text{ if there exists }i\in \{1,\cdots,t\}, b_i=0. 
	\end{cases}
	\end{equation*}
	\item Assume that $L$ is big. Then very free rational curves of minimal $L$-degree have class precisely
	\begin{equation*}
	\begin{cases}
	C_3 & \text{ if for all } i\in\{1,\cdots,t\},b_i\neq 0;\\
	C_1+ C_3 &  \text{ if there exists } i\in \{1,\cdots,t\}, b_i=0. 
	\end{cases}
	\end{equation*}
\end{enumerate}
\end{proposition}
\begin{proof}
	
	Consider the relation of an effective $1$-cycle (Proposition \ref{prop:nex}), 
	$$C=pC_1+qC_2\in\operatorname{AE}_1(X),\quad p,q\in\NN,$$
	which represents very free rational curves. Then $C$ is a positive relation by Theorem \ref{thm:positiverelation}.
	Comparing the coefficients of \eqref{eq:pqrelation}, we get
	\begin{equation}\label{eq:conditiontobeveryfree2}
	p\geqslant q\max_{1\leqslant i\leqslant t}a_i=qa_t.
	\end{equation}
	Observe that $C_1(1)=\{\rho_0,\cdots,\rho_t\}$ is a centred primitive collection. It does not represent very free rational curves, nor does any of its multiples by Example \ref{rmk:primitivecol} (alternatively by Theorem \ref{thm:fTXtoric}), since $X$ is not the projective space. Therefore we must have $q>0$. 
	
	Assume first that $L$ is nef. By Lemma \ref{le:bigness} (1) we have $\deg_L C_1,\deg_L C_3\geqslant 0$. We conclude that 
\begin{equation}\label{eq:ineq1}
	\deg_L C=p\deg_L C_1+q\deg_L C_2\geqslant q(a_t\deg_L C_1+\deg_L C_2)= q\deg_L C_3\geqslant \deg_L C_3.
\end{equation}
	So \eqref{eq:ineq1} implies that $\deg_{L}C_3$ is the minimal of degree amongst all very free rational curves by Lemma \ref{le:veryfreec1c3}.
    If there exists some $b_i=0$, or equivalently, $b_{t-1}=0$ since it is smallest amongst all $b_i$, then $\rho_{t-1}\not\in C_3(1)$. Therefore
    \begin{align*}
    \operatorname{Vect}_\QQ\{\rho:\rho\in C_3(1)\}\subseteq\operatorname{Vect}_\QQ \{\rho_0,\cdots,\rho_{t-2},\rho_{t+1},\cdots,\rho_{n+1}\},
    \end{align*}
    and so $\dim\operatorname{Vect}_\QQ\{\rho:\rho\in C_3(1)\}\leqslant r+s-1=n-1$. By Theorem \ref{thm:fTXtoric}, the class $C_3$ and all of its multiples do not represent very free curves any more. Therefore we must have $p\neq qa_t$, otherwise we would have $C=qC_3$. Hence $p\geqslant qa_t+1$ by \eqref{eq:conditiontobeveryfree2}, and consequently we get
     \begin{equation}\label{eq:ineq2}
     \begin{split}
     \deg_L C&=p\deg_L C_1+q \deg_L C_2\\
     &\geqslant (qa_t+1)\deg_L C_1+q\deg_L C_2\\
     &=q(a_t\deg_L C_1+\deg_L C_2)+\deg_L C_1\\
     &=q\deg_L C_3+\deg_L C_1\\
     &\geqslant \deg_L C_3+\deg_L C_1.
     \end{split}
     \end{equation}
  Thus the relation $C_1+C_3$ achieves the minimum in \eqref{eq:ineq2} by Lemma \ref{le:veryfreec1c3}.
     
     Now assume that $L$ is big. By Lemma \ref{le:bigness} (2), we have $\deg_L C_1,\deg_L C_3> 0$. Then the inequalities \eqref{eq:ineq1} \eqref{eq:ineq2} still hold.	If $b_i\neq 0$ for all $i\in\{1,\cdots,t\}$, then the inequalities in \eqref{eq:ineq1} are all equalities if and only if $p=a_t,q=1$, which means $C=C_3$. So any very free rational curve having $L$-degree $\deg_L C_3$ is represented by $C_3$ by the discussion in the previous paragraph. If $b_{t-1}=0$, then the inequality \eqref{eq:ineq2} is an equality precisely when $q=1,p=a_t+1$, in other words, $C=C_1+C_3$. Hence $C_1+C_3$ is the only class that represents very free rational curves of minimal $L$-degree.
\end{proof}

\subsection{Generic Diophantine approximation}
To state the main result of this section, let us first define what we shall call ``general lines''. They will realize the class of minimal $L$-degree in Proposition \ref{prop:veryfreecurvedegree}. As in Section \ref{se:SectionTheorem} we can assume $Q=Q_0=(1,\cdots,1)\in \mathcal{T}(K)$\index{Q0@$Q_0$}. For a maximal cone $\sigma$, consider its associated affine neighbourhood $U_\sigma\simeq \mathbf{A}^n_K$.
For a $n$-tuple $\mathbf{m}=(m_1,\cdots,m_n)\in K^n$, a \emph{line} through $Q_0$ with parameter $\mathbf{m}$ (with respect to the cone $\sigma$) is, by definition, the rational curve which extends the morphism 
\begin{align*}
\begin{array}{ccc}
\mathbf{A}^1 & \longrightarrow & U_{\sigma}\\
t &\longmapsto & \prod_{\rho_i\in\sigma(1)}\lambda_{\rho_i}(m_i t+1)=(m_1t+1,\cdots,m_nt+1),
\end{array}
\end{align*}
where $\lambda_{\rho_i}$\index{lambdarho@$\lambda_\rho$} is the co-character associated to $\rho_i\in\sigma(1)$. We say that this line is \emph{general} if the parameter $\mathbf{m}\in K^n$ satisfies some open condition.

To compute the class as well as the associated relation of such a general line, we impose the open condition that $\prod_{i=1}^{n}m_i\neq 0$. Write $\sigma(1)=\{\rho_1,\cdots,\rho_n\}$ and consider the unique cone $\tau\in\triangle$ containing $-\sum_{i=1}^{n}\rho_i$ in its relative interior so that
$$\sum_{i=1}^{n}\rho_i+\sum_{\rho\in\tau(1)}c_{\rho}\rho=0,$$
where all $c_\rho$ are (strictly) positive. This is the positive relation for general lines we are looking for. It computes the intersection multiplicities with all boundary divisors (plus the contribution from the point at infinity).

To state our main result, we continue to use the notation in Section \ref{se:toricpic2intro}. All approximation constants are computed with respect to a fixed place $\nu\in\mathcal{M}_K$ in what follows.
\begin{theorem}\label{thm:mainthm3}
	Let $X$ be a split projective smooth toric variety over $K$ with Picard number $2$. Assume that $L=\mathcal{O}_X(D)$ is nef.
	\begin{enumerate}
		\item  We have
		\begin{equation*}
		\aess(Q_0)=
		\begin{cases}
		\deg_L C_3 & \text{ if for all } i\in\{1,\cdots,t\},b_i\neq 0;\\
		\deg_L C_1 +\deg_L C_3 & \text{ if there exists } i\in \{1,\cdots,t\}, b_i=0. 
		\end{cases}
		\end{equation*}
		In all cases, the generic best approximations \emph{can be achieved} (see Definition \ref{def:achieve} (4)) on general lines (with respect to the cone $\sigma_{t,n+1}$ \eqref{eq:conehyp2}) passing through $Q_0$, which are very free of minimal $L$-degree. Consequently, the essential constant equals the minimal $L$-degree of very free rational curves\footnote{Edited after publication on 06.12.2021.}.
		\item Assume moreover that $L$ is big. If either $a_t\geqslant 2$ or $b_{t-1}=0$, then minimal free curves through $Q_0$ form a locally accumulating subvariety.
	\end{enumerate} 
\end{theorem}

\begin{proof}
Fix a general line $l$ under the parametrization of the cone $\sigma_{t,n+1}$ generated by $\{v_1,\cdots,v_n\}$. We start by determining the class of $l$. Consider the element
	$$\rho=\sum_{i=1}^{n}v_i=\sum_{i=0}^{t-1}\rho_i+\sum_{j=1}^{s}\rho_{t+j}.$$
	If $b_i>0$ for all $i$, we have, using \eqref{eq:relation32},
	\begin{align*}
		-\rho&=-\sum_{i=0}^{t-1}\rho_i+\rho_{n+1}+\sum_{i=1}^{t-1}b_i\rho_i+b_t\rho_0\\&=\rho_{n+1}+\sum_{i=1}^{t-1}(b_i-1)\rho_i+(b_t-1)\rho_0. 
	\end{align*}
	This implies that $-\rho$ belongs to a face of $\cap_{j=1}^s\sigma_{t,t+j}$. 
	Therefore the corresponding relation of $l$ is
\begin{align*}
	0=\rho+(-\rho)&=(\sum_{i=0}^{t-1}\rho_i+\sum_{j=1}^{s}\rho_{t+j})+(\rho_{n+1}+\sum_{i=1}^{t-1}(b_i-1)\rho_i+(b_t-1)\rho_0)\\ &=b_t\rho_0+\sum_{i=1}^{t-1}b_i\rho_i+\sum_{j=1}^{s}\rho_{t+j}+\rho_{n+1},
\end{align*}
	namely, the class of $l$ is that of $C_3$ \eqref{eq:cycle3}. On the other hand, if there exists some $b_i=0$, then necessarily $b_{t-1}=0$ and using again \eqref{eq:relation32} we get 
    $$-\rho=\rho_t+\rho_{n+1}+\sum_{i=1}^{t-2}b_i\rho_i+b_t\rho_0\in\bigcap_{j=1}^s \sigma_{t-1,t+j},$$
    which gives the relation
    \begin{align*}
    	0=\rho+(-\rho)&=(\sum_{i=0}^{t-1}\rho_i+\sum_{j=1}^{s}\rho_{t+j})+(\rho_t+\rho_{n+1}+\sum_{i=1}^{t-1}b_i\rho_i+b_t\rho_0)\\ &=(b_t+1)\rho_t+\sum_{i=1}^{t-1}(b_i+1)\rho_i+\sum_{j=1}^{s}\rho_{t+j}+\rho_{n+1}.
    \end{align*}
    So the class of $l$ is that of $C_1+C_3$ \eqref{eq:C1C3}.
    Therefore in any case $l$ is very free by Lemma \ref{le:veryfreec1c3}. we conclude from Definition \ref{def:essconst} that
    \begin{align}
    	\aess(Q_0)&\leqslant \alpha(Q_0,l)
    	=\deg_L (l)\nonumber\\&=\begin{cases}
    	\deg_L C_3, & \text{ if for all } i\in\{1,\cdots,t\},b_i\neq 0;\\
    	\deg_L C_1+\deg_L C_3, & \text{ if there exists } i\in \{1,\cdots,t\}, b_i=0. 
    	\end{cases} \label{eq:aessupperbound}
    \end{align}
    
   Recall \eqref{eq:relation31}, \eqref{eq:relation32}. We shall work under the parametrization given by $\sigma_{t,n+1}$ (see \eqref{eq:unitorpara}):
    \begin{equation}\label{eq:parametrizationusigma0}
    \begin{split}
    \pi:\pi^{-1}(U_{\sigma_{t,n+1}})&\to U_{\sigma_{t,n+1}}\simeq \mathbf{A}^n\\
    (X_0,\cdots,X_{n+1})&\longmapsto \left(\frac{X_0}{X_t X_{n+1}^{b_t}},\frac{X_1}{X_t X_{n+1}^{b_1}}\cdots,\frac{X_{t-1}}{X_t X_{n+1}^{b_{t-1}}},\frac{X_{t+1}}{X_{n+1}},\cdots,\frac{X_n}{X_{n+1}}\right).
    \end{split}
    \end{equation}  
    	We shall work with the toric height function $H_L$ associated to the nef line bundle $L$ defined in $\S$\ref{se:toricheight}.
    	We shall use the distance function \eqref{eq:distdecreasing} defined for all $P=(y_1,\cdots,y_n)$ near $Q_0$:
    $$d_\nu(P,Q_0)=\min\left(1,\max_{1\leqslant i\leqslant n}\left(|y_i-1|_\nu\right)\right).$$
   
   To ease notation, as in $\S$\ref{se:proofofmainthm2} we shall omit the subscripts $L,\nu$ in all $\alpha$-constants below. We define the Zariski closed subset
   \begin{equation}\label{eq:Z}
   Z=\bigcup_{i=1}^n\overline{(y_i=1)}^{\operatorname{Zar}},
   \end{equation}
   where as in $\S$\ref{se:proofofmainthm2}, $(y_i)\subset K^n$ denotes the coordinates of points in $U_{\sigma_{t,n+1}}(K)$. The rest of the proof is devoted to proving
   $$\alpha(Q_0,X\setminus Z)\geqslant \alpha(Q_0,l),$$
   which implies the lower bound $$\aess(Q_0)\geqslant \alpha(Q_0,l).$$
    We shall only consider the cases where $\nu\in\mathcal{M}_k^f$, as the estimation for archimedean places is almost identical.
    For any fixed $P\in(\mathcal{T}\setminus Z)(K)$. Having chosen a set $\mathcal{C}$ comprising integral ideals representing the group $\operatorname{Cl}_K$, let $P_0=(X_i)_{0\leqslant i\leqslant n+1}\in (K\setminus\{0\})^{n+2}$ be one integral lift of $P$ in some twisted torsor $\ctildeT$ for some $\mathbf{c}\in \mathcal{C}^2$ (Theorem \ref{thm:parabyunitor}). With the choice of the $\mathbf{Z}$-basis $\mathcal{D}=\{[D_t],[D_{n+1}]\}$ for $\operatorname{Pic}(X)$,  we have $P_0\in(\mathcal{O}_K\setminus\{0\})^{n+2}$. 
     By Salberger's height formula (Proposition \ref{prop:height}), we have
    \begin{equation}\label{eq:heightlower1}
    \Norm(\mathbf{c}^{D})H_L(P)\geqslant \prod_{\nu^\prime\in\mathcal{M}_K^\infty}\max_{1\leqslant j\leqslant s+1}(|\mathbf{X}(P_0)^{D(\sigma_{t,t+j})}|_{\nu^\prime}).
    \end{equation}
    For every $1\leqslant j\leqslant s+1$, write $D(\sigma_{t,t+j})=c_{t,t+j}D_{\rho_t}+d_{t,t+j}D_{\rho_{t+j}}$, viewed as an element in $\mathbf{Z}^{n+2}$. According to the equations \eqref{eq:relation31} \eqref{eq:relation32} and the relations \eqref{eq:cycle1} \eqref{eq:cycle3}, we have as in \eqref{eq:Dsigma00} \eqref{eq:ci0}, \begin{equation}\label{eq:ctdt}
    c_{t,t+j}=\deg_L C_1,\quad d_{t,t+j}=\deg_L C_3.
    \end{equation} Therefore
    $$\mathbf{X}(P_0)^{D(\sigma_{t,t+j})}=X_t^{\deg_L C_1}X_{t+j}^{\deg_L C_3}.$$
    Since $P\not\in Z(K)$, we have $y_{t+j}=\frac{X_{t+j}}{X_{n+1}}\neq 1$ for all $1\leqslant j\leqslant s$. Now fix any such $j_0$. By Lemma \ref{le:geometryofnumbers}, we have
\begin{equation}\label{eq:distlower1}
\begin{split}
    d_\nu(P,Q_0)\geqslant &\left|\frac{X_{t+j_0}-X_{n+1}}{X_{n+1}}\right|_{\nu}\\
    \geqslant & |X_{t+j_0}-X_{n+1}|_\nu\\
    \geqslant &\prod_{\nu^\prime\in\mathcal{M}_K^\infty} |X_{t+j_0}-X_{n+1}|_{\nu^\prime}^{-1}\\
    \gg & \prod_{\nu^\prime\in\mathcal{M}_K^\infty}\max\left(|X_{t+j_0}|_{\nu^\prime},|X_{n+1}|_{\nu^\prime}\right)^{-1}.
\end{split}
\end{equation}
    On taking the sections $\mathbf{X}(P_0)^{D(\sigma_{t,n+1})}$ and $\mathbf{X}(P_0)^{D(\sigma_{t,t+j_0})}$ in \eqref{eq:heightlower1}, together with \eqref{eq:distlower1}, and by the product formula, we obtain the following lower bound, which is uniform for all $P\in(\mathcal{T}\setminus Z)(K)$:
    \begin{align*}
    	&\Norm(\mathbf{c}^{D})H_L(P)d_\nu(P,Q_0)^{\deg_L C_3}\\
    	\gg &\prod_{\nu^\prime\in\mathcal{M}_K^\infty}\frac{\max\left(\left|X_t^{\deg_L C_1}X_{n+1}^{\deg_L C_3}\right|_{\nu^\prime}, \left|X_t^{\deg_L C_1}X_{t+j_0}^{\deg_L C_3}\right|_{\nu^\prime}\right)}{\left(\max\left(|X_{n+1}|_{\nu^\prime},|X_{t+j_0}|_{\nu^\prime}\right)\right)^{\deg_L C_3}}\\
        \gg &\prod_{\nu^\prime\in\mathcal{M}_K^\infty} |X_t^{\deg_L C_1}|_{\nu^\prime} \gg 1.
    \end{align*}
    So this proves, by Proposition \ref{prop:lowerbd},
    \begin{equation}\label{eq:aesslowerbd}
    \alpha(Q_0,X\setminus Z)\geqslant b_{L,\nu}^{X\setminus\mathcal{T}}(Q_0,X\setminus Z)\geqslant \deg_L C_3.
    \end{equation}
    
    If there exists some $b_i=0$, more is true. First we have $b_{t-1}=0$. It leads us to look at the additional maximal cones $\sigma_{t-1,n+1},\sigma_{t-1,t+j_0}$. For this we rewrite \eqref{eq:relation31} and \eqref{eq:relation32} as (also true for $j_0=s+1$)
    \begin{align*}
    	&v_{t-1}=-\sum_{i=0}^{t-2}v_i-v_t,\\
    	&v_{t+j_0}=-\sum_{i\in\{1,\cdots,t\}\setminus\{t-1\}}b_iv_i-\sum_{j\in\{1,\cdots,s+1\}\setminus \{j_0\}}v_{t+j}.
    \end{align*}
    Therefore we get similarly to \eqref{eq:ctdt},
    $$\mathbf{X}(P_0)^{D(\sigma_{t-1,t+j_0})}=X_{t-1}^{\deg_L C_1}X_{t+j_0}^{\deg_L C_3},$$
    $$\mathbf{X}(P_0)^{D(\sigma_{t-1,n+1})}=X_{t-1}^{\deg_L C_1}X_{n+1}^{\deg_L C_3}.$$
    We now bound the height from below using four maximal cones:
\begin{equation}\label{eq:heightlower2}
    \Norm(\mathbf{c}^{D})H_L(P)\geqslant \prod_{\nu^\prime\in\mathcal{M}_K^\infty}\max(|\mathbf{X}(P_0)^{D(\sigma_{i,t+j})}|_{\nu^\prime}, i\in\{t,t-1\},j\in\{j_0,s+1\}).
\end{equation}
   Since $P\not\in Z(K)$ (recall \eqref{eq:parametrizationusigma0} and \eqref{eq:Z}), thanks to $b_{t-1}=0$, by using additionally the $(t)$-th coordinate $$y_t=\frac{X_{t-1}}{X_t X_{n+1}^{b_{t-1}}}=\frac{X_{t-1}}{X_t}\neq 1$$ at the same time, similarly to \eqref{eq:distlower1}, one has 
    \begin{align*}
   &d_\nu(P,Q_0)\\ \gg& \max\left(\prod_{\nu^\prime\in\mathcal{M}_K^\infty}\max (|X_t|_{\nu^\prime},|X_{t-1}|_{\nu^\prime})^{-1},\prod_{\nu^\prime\in\mathcal{M}_K^\infty}\max(|X_{n+1}|_{\nu^\prime},|X_{t+j_0}|_{\nu^\prime})^{-1}\right).
    \end{align*}
    Integrating this into the lower bound \eqref{eq:heightlower2}, we finally deduce
    \begin{align*}
    	&\Norm(\mathbf{c}^{D})H_L(P)d_\nu(P,Q_0)^{\deg_L C_1 +\deg_L C_3}\\
    	\gg & \prod_{\nu^\prime\in\mathcal{M}_K^\infty}\max\left(|\mathbf{X}(P_0)^{D(\sigma_{t,n+1})}|_{\nu^\prime},|\mathbf{X}(P_0)^{D(\sigma_{t-1,n+1})}|_{\nu^\prime},\right.\\
    	&\quad\quad\quad\quad\quad\left. |\mathbf{X}(P_0)^{D(\sigma_{t,t+j_0})}|_{\nu^\prime},|\mathbf{X}(P_0)^{D(\sigma_{t-1,t+j_0})}|_{\nu^\prime}\right)\\
    	&\times \left(\prod_{\nu^\prime\in\mathcal{M}_K^\infty}\max(|X_{n+1}|_{\nu^\prime},|X_{t+j_0}|_{\nu^\prime})\right)^{-(\deg_L C_3)}\\ &\times \left(\prod_{\nu^\prime\in\mathcal{M}_K^\infty}\max (|X_t|_{\nu^\prime},|X_{t-1}|_{\nu^\prime})\right)^{-(\deg_L C_1)}\gg 1.
    \end{align*}
    This lower bound is again uniform for $P\in(\mathcal{T}\setminus Z)(K)$.
    So under the extra condition that there exists $b_i=0$, by Proposition \ref{prop:lowerbd}, we have proved
    \begin{equation}\label{eq:aesslowerbd2}
    \alpha(Q_0,X\setminus Z)\geqslant b_{L,\nu}^{X\setminus\mathcal{T}}(Q_0,X\setminus Z)\geqslant \deg_L C_1 +\deg_L C_3.
    \end{equation}
    
     In summary, part (1) of Theorem \ref{thm:mainthm3} is now a union of \eqref{eq:aessupperbound} \eqref{eq:aesslowerbd} \eqref{eq:aesslowerbd2} and Proposition \ref{prop:veryfreecurvedegree} (1).
    
    Finally, assume moreover that $L$ is big. If $a_t\geqslant 2$, then $X\neq \PP^s\times\PP^t$, in particular the only centred primitive collection of $X$ is $C_1(1)$. Thus minimal free rational curves all have class $C_1$ by Theorem \ref{th:BCFH}, and those through $Q_0$ form a subvariety $Y_1$ isomorphic to $\PP^t$, which satisfies (Proposition \ref{prop:propertiesofalpha}) $$\alpha(Q_0,Y)=\aess(Q_0,Y_1)=\deg_L C_1.$$  Since $L$ is big and nef, we get by Lemma \ref{le:bigness} that
    $$\deg_L C_3\geqslant a_t\deg_L C_1\geqslant 2\deg_L C_1>\deg_L C_1>0.$$
    Hence we conclude from part (1) of Theorem \ref{thm:mainthm3} that 
    $$\aess(Q_0)\geqslant \deg_L C_3>\aess(Q_0,Y_1),$$ which shows that $Y$ is locally accumulating by Definition \ref{def:achieve} (2).
    If $b_{t-1}=0$, we also deduce from part (1) that $$\aess(Q_0)=\deg_{L} C_1+\deg_{L} C_3>\max(\deg_{L} C_1,\deg_{L} C_3).$$ Note that $C_2(1)$ \eqref{eq:cycle2} may become another centred primitive collection, and this happens precisely when all $a_i=0$, corresponding to $X=\PP^s\times\PP^t$. In that case $C_2$ coincides with $C_3$ \eqref{eq:cycle3}. The deformation locus $Y_2$ of the minimal rational curves of class $C_2$ through $Q_0$ is isomorphic to $\PP^s$ and $$\alpha(Q_0,Y_2)=\aess(Q_0,Y_2)=\deg_L C_3.$$
    So the variety $Y_1\cup Y_2$ satisfies 
    $$\aess(Q_0,Y_1\cup Y_2)=\max(\deg_{L} C_1,\deg_{L} C_3),$$
    and is therefore locally accumulating. The case where there is some $a_i\neq 0$ (assuming $b_{t-1}=0$) is also reduced to the previous one. This completes the proof of part (2) of Theorem \ref{thm:mainthm3}.
\end{proof}

	\bibliographystyle{alpha}
	\bibliography{r}

\renewcommand{\indexname}{Index of notation}
\centering\printindex

\end{document}